\documentclass[12pt]{amsart}
\usepackage[margin=1in]{geometry}
\usepackage{amsmath,latexsym,amssymb,amsthm,graphicx}
\usepackage{mathtools}
\usepackage{amsthm}
\usepackage{thmtools}
\usepackage{subfig}
\usepackage{xfrac}
\usepackage{fix-cm}
\usepackage{color,transparent}
\usepackage{verbatim}
\usepackage{multirow}

\DeclareMathOperator\Cone{Cone}

\usepackage[bookmarks=true,%
    colorlinks=true,%
    linkcolor=blue,%
    citecolor=blue,%
    filecolor=blue,%
    menucolor=blue,%
    urlcolor=blue]{hyperref}
\usepackage{doi}
 \usepackage{multirow}
\usepackage{float}
\usepackage{array}
\usepackage[shortlabels]{enumitem}
\usepackage{pdfsync}
\usepackage{tikz}
\usepackage{tikz-cd}
\usepackage{float}
\usepackage[T1]{fontenc}
\usepackage[utf8]{inputenc}
\usepackage[textsize=small]{todonotes}
\usepackage{cleveref}

\usepackage{hyphenat}
\usepackage{booktabs}
\usepackage{pifont}
\usepackage{pdfpages}
\usepackage[graphicx]{realboxes}
\usepackage{bm} % Bold math symbols
\usepackage{svg}

\setenumerate[1]{label=(\arabic*)}

\numberwithin{equation}{section}
\numberwithin{figure}{section}

\graphicspath{{mpdraws/}{draws/}}

\newcommand{\dsym}{\ensuremath{d_{\operatorname{Lip}}}}
\newcommand{\dasym}{\ensuremath{d_{\operatorname{Th}}}}
\newcommand{\Cf}{\mathcal{C}_{\mathit{f}}}
\newcommand{\Cnf}{\mathcal{C}_{\mathit{nf}}}

\newcommand{\PCf}{\mathbb{P}\mathcal{C}_{\mathit{f}}}
\newcommand{\PCnf}{\mathbb{P}\mathcal{C}_{\mathit{nf}}}
\renewcommand{\P}{\mathbb{P}}
\newcommand{\calP}{\mathcal{P}}

\newcommand\X{\ensuremath{\mathcal{X}}}
\renewcommand\L{\ensuremath{\mathcal{L}}}

\DeclareMathOperator{\adim}{asdim}

\newtheorem{mainthm}{Theorem}[section]
\newtheorem{maincor}[mainthm]{Corollary}

\newtheorem{theorem}{Theorem}[section]
\newtheorem{lem}[theorem]{Lemma}

\newtheorem{cor}[theorem]{Corollary}
\newtheorem{prop}[theorem]{Proposition}

\theoremstyle{definition}
\newtheorem{rem}[theorem]{Remark}

\newtheorem{example}[theorem]{Example}
\newtheorem{defi}[theorem]{Definition}
\newtheorem{question}[theorem]{Question}

\numberwithin{equation}{section}

\DeclareMathOperator{\Thin}{Thin}

\theoremstyle{plain}

 \newcommand\inc{\ensuremath{\prec}}
 \newcommand\sinc{\ensuremath{{%
  \mathrel{\prec\mkern-9mu\raisebox{0.02ex}{\scalebox{0.9}{$\vert$}}}%
}}}

\newcommand{\Hh}{\mathbb H}
\newcommand{\Th}{\mathrm{Th}}

\newcommand{\eps}{\epsilon}

\newcommand\N{\ensuremath{\mathbb{N}}}
\newcommand\Z{\ensuremath{\mathbb{Z}}}

\newcommand\R{\ensuremath{\mathbb{R}}}
\newcommand\C{\ensuremath{\mathcal{C}}}

\def\H{\ensuremath{\mathbb{H}}} % Redefining "long Hungarian umlaut"
\newcommand\D{\ensuremath{\mathcal{D}}}

\newcommand{\PC}{\mathbb{P}\mathcal{C}}

\DeclareMathOperator{\Curr}{Curr}
\DeclareMathOperator{\Teich}{Teich}

\newcommand{\rank}{\operatorname{rank}}
\newcommand{\rankF}{\operatorname{rankF}}

\DeclareMathOperator{\Dil}{Dil}

\newcommand*{\wt}[1]{\widetilde{#1}}

\renewcommand{\epsilon}{\varepsilon}

\newcommand{\E}{\mathbb{E}}

\DeclareMathOperator{\PSL}{\mathit{PSL}}

\DeclarePairedDelimiter\norm{\lVert}{\rVert}

\newcommand{\Simples}{\mathcal{S}}
\newcommand{\Curves}{\mathrm{Curves}}

\newcommand{\A}{\mathcal{A}}

\newcommand{\PML}{\mathbb{P}\mathcal{ML}}
\newcommand{\sys}{\operatorname{sys}}

\DeclareMathOperator{\MCG}{\mathit{MCG}}
\DeclareMathOperator{\CV}{\mathit{CV}}
\DeclareMathOperator{\Out}{\mathit{Out}}

\DeclareMathOperator{\Hit}{Hit}

\newcommand{\PSimplex}{\mathbb{P}\Delta}
\newcommand{\Simplex}{\Delta}

\usetikzlibrary{arrows}
\tikzset{
    labl/.style={anchor=south, rotate=90, inner sep=.5mm}
}
\tikzstyle{every picture}=[> = to]
\tikzset{cdlabel/.style={execute at begin node=$\scriptstyle,execute at end node=$}}
\tikzset{implication/.style={double equal sign distance, -implies}}
\tikzset{biimplication/.style={double equal sign distance, implies-implies}}

\begin{document}
\title{Large-scale rank and geometry of the Lipschitz metric on geodesic currents}
\author[Martínez-Granado]{Dídac Martínez-Granado}
\address{Department of Mathematics\\University of Luxembourg\\Av. de la Fonte 6, Esch-sur-Alzette, L-4364, Luxembourg}       
\email{didac.martinezgranado@uni.lu}
\author[Sapir]{Jenya Sapir}
\address{Department of Mathematics\\
Whitney Hall\\
4400 Vestal Parkway E\\
Binghamton, New York 13902 }
\email{esapir@binghamton.edu}

%\subjclass[2020]{Primary xxx; Secondary xxx}
%\keywords{}

\begin{abstract}
    For a closed orientable surface of genus at least 2, we study the large-scale geometry of the (symmetrized) Lipschitz metric on the space of projective filling geodesic currents. We show that the rank and, hence, asymptotic dimension of this space is infinite. This gives the first example of a natural proper space associated to a closed surface, with a metrically proper isometric action of the mapping class group, that has infinite asymptotic dimension. In contrast, we prove that the Teichm\"uller space equipped with the Lipschitz metric has rank equal to $3g-3$. In particular, the space of projective filling geodesic currents and the Teichm\"uller space are not quasi-isometric.
We study how simplices of geodesic currents provide natural flats for the Lipschitz metric and reveal the simplicial structure of boundary of projective filling currents. We show that, in contrast to the infinite rank of their ambient space, the rank of simplices is uniformly bounded above by $8g-9$. We then establish a sufficient criterion for the rank and asymptotic dimension of a simplex to be equal to $8g-9$ and construct examples satisfying it, thus showing the bound is sharp. 
\end{abstract}
\maketitle

\vspace{-2.5em}
\section{Introduction}

The \textit{rank} of a metric space is a quasi-isometry invariant that measures the extent to which the space exhibits higher-dimensional Euclidean behavior at large scale. The systematic study of rank originates in work of Gromov~\cite[6.$B_2$]{Gromov1993Asymptotic}, as an obstruction to hyperbolicity. In particular, spaces associated to a hyperbolic surface $S$, such as the Teichm\"uller space, $\Teich(S)$,
with various metrics, and the mapping class group $\MCG(S)$, demonstrate both aspects of hyperbolicity~\cite{Minsky1996,Rafi2014}, and spectacular failures thereof~\cite{LenzhenMasur2010}. However, relatively little is known about the rank and the large-scale geometry of Teichmuller space with the symmetrized Thurston metric~\cite{Th98:ThurstonMetric}, also known as the \emph{Lipschitz metric}, and its extension to the space of projective filling geodesic currents $\PCf(S)$~\cite{Sap22:ExtensionThurston}. In this paper we make progress in this direction. We record what is known, and the contributions in this paper, in the table below. 
\begin{table}[h]
\centering
  \makebox[\textwidth][c]{
   {\small
\begin{tabular}{cclp{14em}r}
\toprule
\textbf{Space} & \textbf{Metric} & \textbf{Geometric rank~$(r)$} & \textbf{Asymptotic dim.~$(d)$}& \textbf{References}\\
\midrule
$\Teich(S)$ & Weil--Petersson
& $r=\lfloor(\xi(S)+1)/2\rfloor$
& $ 
d \le 5\xi(S)^2 + \xi(S)$
&  \\

$\Teich(S)$ & Teichm\"uller
& $r=\xi(S)$
& $
d \le 5\xi(S)^2 + \xi(S)$
& See Section~\ref{subsec:history} below \\

$\MCG(S)$ & Word metric
& $r=\xi(S)$
& $ 
d \le 5\xi(S)^2$
& \\
\hline
$\Teich(S)$ & \textbf{Lipschitz}
& $\mathbf{r=\xi(S)}$& unknown
& Thm.~\ref{mainthm:teichrank} \\
$\PSimplex(\vec \mu)$ & \textbf{Lipschitz}
& $\mathbf{0 \leq r\leq 2\xi(S)-\chi(S)}$& \mbox{\textbf{cases when }$\mathbf{r=d=2\xi(S)-\chi(S)}$} &  Thm.~\ref{mainthm:max_rank}, Thm.~\ref{mainthm:criterion}\\
$\PCf(S)$ & \textbf{Lipschitz}
& $\mathbf{r=\infty}$& \textbf{$\mathbf{d=\infty}$} & Thm.~\ref{mainthm:infinite_geo_rank}, Cor.~\ref{maincor:adim_infinite} \\
\bottomrule
\end{tabular}}}
\caption{\small{Values and bounds for geometric rank and asymptotic dimension of spaces associated to closed surfaces: the main results of this paper are in the three last rows. Note that we always have $d \geq r$.}}
\label{table:ranks}
\end{table}

Here, given a closed orientable genus $g$ surface $S$, we set $\xi(S) = 3g-3$ to be its the complexity, and $\chi(S) = 2-2g$ its Euler characteristic. $\PCf(S)$ is the space of projective filling currents, and $\PSimplex(\vec \mu) \subset \PCf(S)$ is any so-called simplex spanned by a sequence $\vec \mu = (\mu_1, \dots, \mu_n)$ of geodesic currents, defined in \Cref{def:simplex} below.

\subsection{Rank and asymptotic dimension for Teichm\"uller space and mapping class group}
\label{subsec:history}
There is a long history of work done on the rank (Definitions~\ref{def:geo_rank},~\ref{def:fat_rank}) and asymptotic dimension (Definition~\ref{defi:adim}) of Teichm\"uller space $\Teich(S)$ with various metrics, as well as the mapping class group $\MCG(S)$ with any word metric (See Table~\ref{table:ranks} for the most up-to-date results). This has proven to be quite a difficult problem, with various methods of approach.

Lower bounds on geometric rank were found first. Sharp lower bounds were obtained for $\Teich(S)$ with the Teichm\"uller metric by Minsky~\cite{Minsky1996ProductRegions}. For $\MCG(S)$ with any word metric sharp lower bounds follow naturally from Masur-Minsky~\cite{MM00}, though they can be found explicitly in Farb--Lubotzky--Minsky~\cite{FarbLubotzkyMinsky2001}. Lastly, Brock--Farb~\cite{BrockFarb2006} obtained a lower bound of $\lfloor \xi(S)/2 \rfloor$ for the rank of $\Teich(S)$ with the Weil-Petersson metric, but conjectured the actual rank of  $\lfloor (\xi(S) +1)/2 \rfloor$ that appears in \Cref{table:ranks}.

Exact computations of geometric rank followed. Behrstock--Minsky~\cite{BehrstockMinsky2008} gave the exact rank of $\Teich(S)$ with the Weil-Petersson metric, as well as an upper bound on the rank of $\MCG(S)$ that matched the previously found lower bound (see also \cite{ham05} for the latter). The first exact rank computation for the Teichm\"uller metric on $\Teich(S)$ appeared in~\cite{EMR17:Rank}, which also provided a  common framework via coarse differentiation that recovers the results for $\Teich(S)$ with the Weil-Petersson metric, as well as for $\MCG(S)$. 
It should be noted that a different line of inquiry due to Bowditch~\cite{Bow16,Bow18:MCG,Bow20} provides an alternative approach to computing rank via coarse median spaces, which were first introduced in~\cite{Bow13}.

Later work of Bestvina--Bromberg--Fujiwara~\cite{BestvinaBrombergFujiwara2015} established finiteness of the  asymptotic dimension of mapping class groups, giving upper bounds that were exponential in the genus. Subsequently, Behrstock--Hagen--Sisto~\cite{BehrstockHagenSisto2017} obtained the quadratic upper bounds appearing in Table~\ref{table:ranks} for the Teichm\"uller metric and Weil-Petersson metrics on $\Teich(S)$, and $\MCG(S)$ with word metric. In fact, they axiomatized the earlier work of Masur-Minsky~\cite{MM99,MM00} via the language of hierarchically hyperbolic spaces: their work~\cite{BHS17:HHS1} gave another unified framework for computing rank, and~\cite{BHS21:Quasiflats} gave further control on the structure of quasi-flats.

The rank for the Lipschitz metric in Teichm\"uller space was less well understood.  Choi-Rafi's~\cite[Theorem~C]{CR07:Comparison} thick-thin decomposition theorem gives the lower bound $3g-3$ for the rank of $(\Teich(S),\dsym)$. In \Cref{mainthm:teichrank} we will find this to be the exact rank, bypassing the machinery in~\cite{EMR17:Rank,BHS17:HHS1}, which is not currently  applicable for $(\Teich(S),\dsym)$.

\subsection{Rank and asymptotic dimension for geodesic currents}
The main focus of this paper is the extension of the rank problem of $\dsym$ to the setting of geodesic currents. Bonahon introduced the space of geodesic currents $\C(S)$~\cite{Bonahon86:EndsHyperbolicManifolds,Bonahon88:GeodesicCurrent} as a cone of measures providing a completion of the set of weighted closed geodesics on $S$. Its projectivization $\PC(S)$ naturally contains both $\Teich(S)$ and its Thurston boundary $\PML(S)$. The second author~\cite{Sap23:LengthComparison,Sap22:ExtensionThurston} introduced the \emph{Lipschitz metric} $\dsym$ on the subspace $\PCf(S) \subset \PC(S)$ of projective filling geodesic currents---the projectivization of the geometric interior of the cone $\C(S)$--- that extends the 
symmetrized Thurston metric on $\Teich(S)$. 

Geodesic currents have  become a fundamental tool in 
low-dimensional topology and Teichmüller theory, with applications including curve counting~\cite{ES16:Counting,RS19:CountingProblems,ES22:GeodesicCount} and compactifications of higher- rank character varieties~\cite{MOT21:Core,OT21:Blaschke,BIPP21:Currents,BIPP21:Crossratios}, due to the fact that many spaces of marked geometric structures on $S$ and families of representations of $\pi_1(S)$ into higher-rank Lie groups, such as Hitchin representations (see~Section~\ref{subsec:hitchin}), embed into $\PC(S)$. 

The intrinsic geometry of $\PC(S)$ has only begun to be understood recently. An asymmetric version extending Thurston's classical asymmetric metric has been studied by Jyothis and the first author~\cite{JMG26:Horoboundary}, developing its horoboundary and isometric rigidity. The asymmetric extended Thurston metric as well as its symmetrization $\dsym$ restrict to geometrically meaningful metrics on many subspaces of $\PCf(S)$, such as families of higher-rank representations of $\pi_1(S)$~\cite{MZ19:PositivelyRatioed,CDPW24:Asymmetric,Nol24:Holder}. Beyond the setting of surfaces and geodesic currents, the Lipschitz metric $\dsym$ also appears naturally in the more general space $\mathcal{D}(\Gamma)$ of proper hyperbolic invariant pseudometrics on hyperbolic groups~\cite{Fur02:Coarse} and~\cite[Proposition~3.5]{OR22:SpaceMetric}, satisfying $\PCf(S) \subset \mathcal{D}(\pi_1(S))$~\cite{cantrell-reyes.manhattan, DRMG23:Duals}. By work of the first author with Cantrell and Reyes, the metric completion of $\mathcal{D}(\Gamma)$ can be isometrically identified with the space of reparameterizations of the geodesic flow~\cite{CMGR26:Flows-v1}, on which $\dsym$ coincides with the \emph{Hilbert metric} introduced by Tholozan~\cite[Theorem~1.31]{tholozan}.

In this paper, we examine the question of the geometric rank and asymptotic dimension of $\PCf(S)$. We also study a natural family of subspaces. A fundamental feature of $\C(S)$ is that it is closed under positive linear combinations. This means that we can consider a \textit{simplex} of geodesic currents, that is, the projectivized positive cone generated by
a finite family of geodesic currents $\vec \mu = (\mu_0, \dots, \mu_n)$, denoted by $\PSimplex(\vec \mu)$.
We will be mostly interested in simplices for which $\sum_i \mu_i$ is filling, which we will call \emph{filling simplices} (see Definition~\ref{def:simplex}). These form natural geodesic subspaces of $\PCf(S)$. 
Indeed, if $[\mu_1+\mu_2] \in \PCf(S)$, it follows that $t \mapsto [(1-t)\mu_1+t\mu_2]$ for $t \in (0,1)$ forms an unparameterized geodesic connecting $[\mu_1]$ to $[\mu_2]$ ~\cite[Theorem~1.2, Proposition~5.6]{COR22:Manhattan} with respect to $\dsym$. 
Furthermore, it follows from~\cite[Proposition~3.1]{LRT11:Bounded} that every point of Teichm\"uller space is at a uniformly bounded $\dsym$-distance from the simplex generated by its short regular marking.

We will work with two notions of rank. The \emph{geometric rank}, denoted by $\rank(X)$, measures the largest dimension of any family of uniformly quasi-isometrically embedded Euclidean cubes into $X$ of arbitrarily large scale. On the other hand, the \emph{flat rank} $\rankF(X)$ is the largest dimension $n$ of any quasi-isometrically embedded Euclidean $n$-orthant in $X$ (see Section~\ref{subsec:ranks}).
They clearly satisfy $\rankF(X) \leq \rank(X)$ (\Cref{lem:flat_rank_lower}).

\subsection{Main results}

Our first result establishes that the space $\PCf(S)$ has infinite geometric rank.

\begin{restatable}{mainthm}{InfiniteRank}
\label{mainthm:infinite_geo_rank}
Let $S$ be a closed orientable surface of genus $\geq 2$.
Then
\[
\rank(\PCf(S),\dsym)=\infty.
\]
\end{restatable}

As an immediate corollary of Theorem~\ref{mainthm:infinite_geo_rank} and \Cref{lem:rank_adim}, this gives infinite asymptotic dimension.
\begin{maincor}
\label{maincor:adim_infinite}
Let $S$ be a closed orientable surface of genus $\geq 2$. Then
\[
\adim(\PCf(S),\dsym) = \infty.
\]
\end{maincor}

To our knowledge, this gives the first example of a natural space associated to a closed surface that has infinite asymptotic dimension while carrying a \textbf{proper metric}~\cite{Sap22:ExtensionThurston} (i.e., closed balls are compact), and a \textbf{metrically proper}~\cite{EM18:Ergodic} isometric action of the mapping class group, i.e., an action by isometries where for every $R>0$ and every $[\mu] \in \PCf(S)$, the set $\{ [\nu] \in \PCf(S) \colon \dsym([\mu],g [\mu]) \leq R \}$ is finite.

We 
construct, for every dimension $d \in \mathbb{N}$, a family of uniformly quasi-isometrically embedded cubes $\{ f_R \colon [0,R]^d \to \PCf(S)\}_{R \in \mathbb{R}_+}$ with quasi-isometry constants independent of $R$ and, in fact, of $d$. Furthermore, for each $R$ we have $f_R([0,R]^d) \subset \PSimplex_R$, where $\PSimplex_R$ is a simplex of geodesic currents depending on $R$ which, strikingly, is quasi-isometric to a line (\Cref{lem:tau_simplex_is_line}). This shows that infinite rank can emerge at infinity even if, at each finite scale, the rank of the ambient simplex is equal to 1.

On the other hand, we give an exact computation of the rank of Teichm\"uller space with the Lipschitz metric.
\begin{mainthm}
\label{mainthm:teichrank}
Let $S$ be a closed orientable surface of genus $g \geq 2$. Then
\[
\rank(\Teich(S),\dsym)=\rankF(\Teich(S),\dsym)=3g-3.
\]
\end{mainthm}
The proof uses an alternative method that bypasses the hierarchical hyperbolicity machinery of~\cite{EMR17:Rank,BehrstockHagenSisto2017}, which is not currently available for the Lipschitz metric.

In fact, it rests on a comparison between $\dsym$ and $d_T$ at the level of asymptotic cones that might be of independent interest. Even though $\dsym$ and $d_T$ are not quasi-isometric~\cite[Theorem~E]{CR07:Comparison}, their discrepancy at the level of asymptotic cones can be localized on a simple multi-curve consisting of ``very short curves''. 

Our global comparison estimate,~\Cref{lem:min-global}, which builds on~\cite{Minsky1996ProductRegions} and~\cite{CR07:Comparison}, shows that on the complementary factors, the Teichm\"uller and Lipschitz asymptotic cones are bilipschitz equivalent, while
the annular factors corresponding to the very short curves have cones bilipschitz equivalent to real trees. Combining these comparisons with results of Bowditch and Durham~\cite{Durham,Bow16} gives that every compact subset of an asymptotic cone of $(\Teich(S),\dsym)$ has topological dimension bounded by $3g-3$ (Lemma~\ref{lem:min-dimension}) which implies then an upper bound on geometric rank. Theorem~\ref{mainthm:teichrank} immediately gives a lower bound to $\adim(\Teich(S),\dsym)$,
\begin{maincor}
Let $S$ be a closed orientable surface of genus $g \geq 2$.
Then
\[
\adim(\Teich(S),\dsym) \geq 3g-3.
\]
\end{maincor}
As an immediate corollary of Theorem~\ref{mainthm:teichrank} as well as Theorem~\ref{mainthm:infinite_geo_rank}, we can distinguish the large-scale geometry of the space of filling geodesic currents and Teichm\"uller space, answering~\cite[Conjecture~1.4]{Sap22:ExtensionThurston}. 

\begin{maincor}
   Let $S$ be a closed orientable hyperbolic surface of genus $\geq 2$.  Then $(\PCf(S),\dsym)$ and $(\Teich(S),\dsym)$ are not quasi-isometric.
\end{maincor}

On the other hand, the proof of \Cref{mainthm:infinite_geo_rank}, given in Section~\ref{sec:infinite}, does not give bounds on the flat rank of $\PCf(S)$. 
Still, simplices of geodesic currents give a natural place to look for quasi-isometrically embedded flats for $(\P \Cf(S),\dsym)$. We say that a finite sequence of geodesic currents $\vec{\mu}=(\mu_0,\cdots, \mu_n)$ is \emph{jointly filling} if their sum is filling in $\Cf(S)$, and call its simplex, $\PSimplex(\vec \mu) \subset \PCf(S)$, \emph{filling}. Since $\C(S)$ is an infinite-dimensional cone inside of an infinite dimensional vector space of signed measures, filling simplices can have arbitrarily large \emph{linear} dimension. For example, any $(n+1)$-tuple $\vec \mu$ of mutually singular geodesic currents is linearly independent~\cite{Glo19:Negatively}, and hence the cone $\Simplex(\vec \mu)$ they span has linear dimension $n+1$. Thus, the topological dimension of its projectivization, the simplex $\PSimplex(\vec \mu)$, is $n$; see Example~\ref{ex:singular_currents}. In general, the geometric rank of a filling simplex $\PSimplex(\vec \mu)$ is at most $n$ (Corollary~\ref{cor:upper_bound_rank}) but it can typically be strictly smaller than $n$. For example, if all the geodesic currents in $\vec \mu$ are filling, $\PSimplex(\vec \mu)$ is bounded, and hence its  rank and asymptotic dimension are zero. It is therefore natural to ask how much the topological dimension of a simplex is detected by its large-scale geometry.

For an arbitrary filling simplex, in sharp contrast with \Cref{mainthm:infinite_geo_rank}, we show that the geometric rank and flat rank of an arbitrary simplex of geodesic currents is uniformly bounded in terms of the topology of $S$, and construct examples where the bounds are realized.

\begin{restatable}{mainthm}{MaxRank}
\label{mainthm:max_rank}
   Let $S$ be a closed orientable hyperbolic surface of genus $g \geq 2$. Let $\vec \mu = (\mu_0, \dots, \mu_n)$ be any jointly filling sequence of geodesic currents. Then
   \[
 \rank(\PSimplex(\vec \mu)) \leq 8g-9.
   \]
\end{restatable}
Note that by \Cref{lem:flat_rank_lower}, this also gives an upper bound on flat rank, as $\rankF(\PSimplex(\vec \mu)) \leq \rank(\PSimplex(\vec \mu))$. 

The boundary of $\PCf(S)$ has a natural cellular decomposition that induces a partial ordering on geodesic currents. The value $8g-9$ corresponds to the length of the longest chain of non-filling geodesic currents for this partial order. 

More concretely, it follows from a detailed analysis of the structure of simplices of geodesic currents. One of its pillars is the canonical decomposition of a geodesic current into a ``filling component'' and  a ``lamination component'' on the complementary subsurfaces given in~\cite[Theorem~1.2]{BIPP21:Currents} as well as~\cite[Proposition~A]{EM18:Ergodic} (see  Eq.~\eqref{eq:current_decomposition} below).  We show that, in order for geometric rank to increase, we can assume that these components evolve monotonically. 
After choosing a common pants-and-annuli decomposition of the surface adapted to
the sequence $\vec \mu$, every pair of pants can become part of the filling
support at most once, while each annulus can contribute at most twice:
once through the lamination component and once through the filling component.
A pants-and-annulus decomposition of a closed surface of genus $g$ contains $2g-2$ pairs of pants and
$3g-3$ annuli, giving
\[
(2g-2)+2(3g-3)=8g-8
\]
available contributions. The initial nonzero geodesic current consumes at least one
contribution, and thus a strictly increasing sequence contains  at most $8g-9$ strict transitions. 

Even though $\Simplex(\vec \mu)$ is in general homeomorphic to a Euclidean orthant, this alone does not imply an upper bound on the asymptotic dimension of $\PSimplex(\vec \mu)$: indeed, there are examples of geodesic and uniformly contractible metric spaces homeomorphic to $\mathbb{R}^n$ with infinite asymptotic dimension~\cite{DFW03,Dra00}. However, in some cases, $\PSimplex(\vec \mu)$ is, in fact, quasi-isometric to a Euclidean orthant. 

In particular, a jointly filling sequence $\vec{\mu}$ is said 
\emph{strictly increasing} if, roughly speaking, every step introduces a ``new direction'' detected by intersection with some closed geodesic (see  Definition~\ref{def:increasing_sequence} for details). 
We prove that for simplices spanned by strictly increasing families, the flat rank, geometric rank and asymptotic dimension are as large as possible and, furthermore, coincide with the topological dimension.

\begin{restatable}{mainthm}{StrIncr}
\label{mainthm:criterion}
   Let $S$ be a closed orientable hyperbolic surface of genus $\geq 2$. Let $\vec\mu=(\mu_0,\ldots,\mu_n)$ be a jointly filling, strictly
increasing sequence of nonzero geodesic currents. Then
\[
\rankF(\PSimplex(\vec \mu)) = \rank(\PSimplex(\vec \mu)) = \adim(\PSimplex(\vec \mu)) = n.
\]
Moreover, there exist $\vec \mu$ so that $n = 8g-9$. Thus, \[\rankF(\PCf(S)) \geq 8g-9\]
\end{restatable}

\subsection{Further questions}
\label{sec:questions}

We list further avenues of research suggested by the main results of this paper.

\paragraph{\textbf{Asymptotic dimension of $\Teich(S)$}} It would be interesting to compute the exact value of $\adim(\Teich(S),\dsym)$, or, more modestly, give upper bounds in terms of a function of $g$. We note that, unlike for $\Teich(S)$ with Weil-Petersson and Teichm\"uller metrics, the machinery in \cite{BehrstockHagenSisto2017} does not apply in this case.

\begin{question}
Can we give an upper bound for $\adim(\Teich(S),\dsym)$ in terms of $g$?
\end{question}

At the moment we do not even know if $\adim(\Teich(S),\dsym)$ is finite. 

\paragraph{\textbf{Infinite rank and dimension}}

 We have not determined the flat rank of projective filling geodesic currents, although
 we know it is at least $8g-9$ by Theorem~\ref{mainthm:max_rank}. For the natural spaces associated to closed surfaces, flat rank and geometric rank coincide. 
\begin{question}
Do we have  $\rankF(\PCf(S),\dsym)=\rank(\PCf(S),\dsym)=\infty$?
\label{que:rank_infinite}
\end{question}

Another natural follow-up question, given the $\adim(\PCf(S),\dsym)$ is infinite, is to determine its growth rate. There are several such notions, eg, ``dimension growth rate'' as defined in~\cite[Definition~1.3]{DS11:InfiniteGrowth}, or ``asymptotic dimension growth rate''~\cite{Dra06:GrowthRate}. These are notions measuring how the multiplicity or number of colors of a given cover vary as a function of a given parameter, such as the separation constant, or the Lebesgue covering constant. Even though they are related, at the moment it is yet unknown if they coincide in general~\cite[Question~1.5]{DS11:InfiniteGrowth}.
Computing these notions of dimension growth for $(\PCf(S),\dsym)$ might give new insights on how these invariants compare, and give finer detail on the geometry of $\PCf(S)$.

\begin{question}
Can we compute or bound the several notions of dimension growth of $(\PCf(S),\dsym)$?
\end{question} 

\paragraph{\textbf{Flats from Hitchin representations}}
\label{subsec:hitchin}
Another promising source of flats for $\dsym$ might be via representations of $\pi_1(S)$ into higher rank Lie groups.
Martone--Zhang~\cite{MZ19:PositivelyRatioed} associate filling geodesic currents to every Hilbert length function of $\Hit_n(S)$, the space of Hitchin representations $\rho \colon \pi_1(S)\to \PSL(n,\mathbb{R})$. This allows to embed, up to an involution, $\Hit_n(S)$ into $\PCf(S)$. Furthermore, there are 
pants-adapted coordinates of Sun--Wienhard--Zhang~\cite{SWZ20,SZG25:Goldman} growing quadratically in $n$. This motivates the following question.

\begin{question}
What are $\rank(\operatorname{Hit}_n(S),\dsym)$ and $\rankF(\operatorname{Hit}_n(S),\dsym)$?
\end{question}

\paragraph{\textbf{The space of pseudometrics, and free groups}}

The space $\PCf(S)$ embeds isometrically into the much larger space $\D_{\pi_1(S)}$ of equivalence classes of Gromov-hyperbolic metrics on $\pi_1(S)$~\cite{DRMG23:Duals,COR22:Manhattan} quasi-isometric to a word-metric, where two metrics are considered equivalent if they have the same stable length spectrum up to global scaling; in fact, this space is equipped with a natural extension of $\dsym$~\cite{Fur02:Coarse}. Hence, by Theorem~\ref{mainthm:infinite_geo_rank}, it follows that
\[
\rank(\D_{\pi_1(S)}, \dsym)=\adim(\D_{\pi_1(S)},\dsym)=\infty.
\]
The space $\D_\Gamma$ has recently received renewed attention~\cite{OR22:SpaceMetric, COR22:Manhattan,CR24:Marked,CRS24:Joint,CT25:Manhattan,CMGR26:Flows-v1}. When $\Gamma=F_n$, the free group of rank $n \geq 2$, it contains an isometrically embedded copy of Culler--Vogtmann Outer space $\CV(F_n)$~\cite{OR22:SpaceMetric,FrancavigliaMartino2011}. The asymptotic dimension of $\Out(F_n)$ remains unknown for $n\geq 3$. 
It would be interesting to know if the the techniques in this paper can be extended to $\D_\Gamma$ for an arbitrary hyperbolic group $\Gamma$, and, in particular, give insight into the asymptotic dimension of $\Out(F_n)$.

\subsection*{Acknowledgments}
The authors thank Kasra Rafi and Eduardo Reyes for helpful discussions related to this work, as well as Jason Behrstock and Matthew Durham for their comments on the first draft of the manuscript. Support for D\'idac Mart\'inez-Granado was provided by the Marie Sklodowska--Curie Action CurrGeo grant (101154865), hosted by the University of Luxembourg; he also thanks the Binghamton Math Department for their hospitality during the visit where this project started. Support for Jenya Sapir was provided by a grant from the Simons Foundation [MPS-TSM-00965278, JS].
The authors did not use AI or LLM tools for any aspect of this research or the writing of this manuscript.

\subsection*{Organization}

In Section~\ref{subsec:geodesic currents} we give the background on geodesic currents and on coarse geometry. Specifically, we define two notions of rank, asymptotic cones, and asymptotic dimension.
In Section~\ref{sec:simplices_currents}, we introduce simplices spanned by geodesic currents and the several properties they can satisfy.
Section~\ref{sec:outlineCD} outlines the proof of Theorems~\ref{mainthm:max_rank} and~\ref{mainthm:criterion}, which runs from Sections~\ref{sec:upper_bound_on_rank_of_simplices} to~\ref{sec:strictly_increasing_seq}.
In Section~\ref{sec:infinite} we prove \Cref{mainthm:infinite_geo_rank}, and 
Section~\ref{app:dsym-rank-minimal}  proves Theorem~\ref{mainthm:teichrank}.
Section~\ref{sec:appendix} is the Appendix, where we collect a couple of technical lemmas.

\section{Background}
\label{subsec:geodesic currents}

Let $S$ be a closed orientable surface of genus $g \geq 2$. Fix a hyperbolic metric on $S$ and let
\[
G(S)=\bigl(S^1\times S^1-\Delta\bigr)/(\mathbb Z/2)
\]
be the space of unoriented, unparametrized geodesics of the universal cover
$\widetilde S$ isometric to the hyperbolic plane $\mathbb H^2$, where $\mathbb Z/2$ acts by exchanging the two endpoints at the boundary at infinity.

A \textbf{geodesic current} on $S$ is a $\pi_1(S)$-invariant, Radon measure on $G(S)$ or, equivalently in this setting, a locally finite Borel measure on $G(S)$. We denote by $\C(S)$ the space of geodesic currents, equipped with the weak$^*$ topology. This topology is metrizable and locally compact, and the projectivization $\PC(S) \coloneqq (\C(S) - \{ 0 \})/\mathbb{R_+}$ is compact \cite[Proposition~11]{Bonahon88:GeodesicCurrent}.

A \textbf{closed curve} $\gamma$ in $S$ is a free homotopy class of closed paths on $S$. We say it is \textbf{essential} if it is non-nullhomotopic. Every closed curve determines a geodesic current as follows. If $\gamma$ is \textbf{primitive}, i.e., not a non-trivial power of another essential curve, we take the counting measure $\delta_\gamma$ on its full lift $\wt{\gamma}$ to $\widetilde S$. If $\gamma$ is non-primitive, i.e., it is of the form $[\gamma]=[\delta^k]$ for some primitive curve $\delta$, and some $k \in \mathbb Z \backslash \{\pm 1\}$ we let $\delta_\gamma \coloneqq k \delta_\delta$. Let $\Curves(S)$ denote the set of all essential closed curves on $S$, and $\Simples(S) \subset \Curves(S)$ the subset of \textbf{simple closed curves}. If $F$ is a surface with non-empty boundary, a \textbf{boundary-parallel curve} $\alpha \subset F$ is a curve homotopic to a boundary component. If a curve is not a boundary-parallel curve, we say is \textbf{non-peripheral}. If $F \subset S$ is a subsurface with boundary, we let $\Curves(F)$ denote the non-peripheral curves of $F$, and similarly for simple curves. Weighted closed curves are dense in $\Curr(S)$~\cite[Proposition~4.4]{Bonahon86:EndsHyperbolicManifolds}, and, consequently, $\Curves(S)$ is dense in $\PC(S)$. By \cite[Proposition~4.5]{Bonahon86:EndsHyperbolicManifolds}, the \textbf{geometric intersection number} of curves, i.e., the minimal number of transverse intersections among all free homotopy class representatives in general position, extends continuously to a symmetric bilinear form
\[
i(\cdot,\cdot)\colon \C(S)\times \C(S)\to \mathbb R_{\geq 0}.
\]
In fact, the weak$^*$-topology can be characterized by convergence of intersection numbers:
\[
\mu_n\to\mu
\qquad\Longleftrightarrow\qquad
i(\mu_n,\gamma)\to i(\mu,\gamma)
\]
for every closed curve $\gamma$ on $S$
\cite[Theorem~12]{DLR10:DegenerationFlatMetrics} (c.f.~\cite[Th\'eor\`eme~2]{Otal90:SpectreMarqueNegative}).

A geodesic current $\mu\in\C(S)$ is called \textbf{filling} if every geodesic on $S$ intersects its support transversely. We denote the space of filling geodesic currents by $\Cf(S)$ and its projectivization by $\PCf(S)$. The space $\Cf(S)$ is an open and dense in $\C(S)$, and hence so is $\PCf(S)$ in $\PC(S)$~\cite[Lemma~3.5]{EM18:Ergodic}.
In particular, as an open subset of the compact space $\PC(S)$, the space $\PCf(S)$ is locally compact.

By \cite[Theorem~1.5]{BIPP21:Currents}, the following are equivalent for $\mu\in\C(S)$:
\begin{enumerate}
\item $\mu$ is filling;
\item $i(\mu,\nu)>0$ for every non-zero $\nu\in\C(S)$;
\item the \textbf{systole} of $\mu$, defined as
\[
\sys(\mu)\coloneqq
\inf_{\gamma} i(\mu,\gamma)
\]
is positive; here the infimum is taken over all (essential) closed curves $\gamma$ on $S$.
\end{enumerate}

Bonahon associated to every marked hyperbolic structure
$\X\in\Teich(S)$ a \textbf{Liouville geodesic current}
$\L_\X\in\Cf(S)$, yielding a proper embedding
\[
\Teich(S)\hookrightarrow \PCf(S), \qquad \mathcal X \longmapsto [\mathcal L_{\mathcal X}].
\]
Moreover,
\begin{equation}
i(\L_\X,\gamma)=\ell_\X(\gamma)
\label{eq:length_inter}
\end{equation}
for every closed curve $\gamma$ \cite{Bonahon86:EndsHyperbolicManifolds,Bonahon88:GeodesicCurrent}.

For $\mu\in\Cf(S)$, its \textbf{entropy} is the positive real number given by~\cite[Lemma~2.2]{JMG26:Horoboundary}
\begin{equation}
h(\mu)
=
\limsup_{n\to\infty}
\frac1n
\log
\#\{\gamma\text{ closed geodesic}: i(\mu,\gamma)<n\}.
\label{eq:entropy}
\end{equation}
It satisfies~\cite[Lemma~2.1]{JMG26:Horoboundary}
\[
h(a\mu)=\frac{1}{a}h(\mu),
\qquad a>0.
\]

For $\mu,\nu\in\C(S)$, let the \textbf{dilatation} of $\mu$ and $\nu$ be given by
\begin{equation}
\Dil(\mu,\nu)
\coloneqq
\sup_{\gamma}
\frac{i(\nu,\gamma)}{i(\mu,\gamma)},
 \label{eq:dilatation}
\end{equation}
where the supremum is taken over all closed curves $\gamma$ where the denominator does not vanish.
At times it will be convenient to know that the supremum can be taken over other collections of curves or geodesic currents without changing the value of the dilatation.
Let $\operatorname{Curves}_f(S)$  denote the spaces of all filling closed curves, and recall that $\Curves(S)$ denotes the space of all closed curves.
 \begin{lem}
\label{lem:dilations}
We have the following equalities:
\[
\Dil(\nu, \mu) = \sup_{\gamma \in \operatorname{Curves}_f(S)}\frac{i(\mu,\gamma)}{i(\nu,\gamma)} = \sup_{\eta \in \PCf(S)}\frac{i(\mu,\eta)}{i(\nu,\eta)}  = \sup_{\eta \in \PC(S)}\frac{i(\mu,\eta)}{i(\nu,\eta)},
\]
where the last supremum is taken over geodesic currents that do not make the denominators vanish.
Furthermore, if $\mu$ and $\nu$ are filling, the last quantities are equal to
\[
\max_{\eta \in \PC(S)}\frac{i(\mu,\eta)}{i(\nu,\eta)}.
\]
\end{lem}
\begin{proof}
First, linearity of intersection number implies that the quantity $\frac{i(\mu,\eta)}{i(\nu,\eta)}$ is invariant under scaling $\eta$. Thus, we can take the supremum over projective classes of geodesic currents. 

The fact that $\PCf(S)$ is a dense subspace of $\PC(S)$ and the continuity of intersection number, shows that
\[
\sup_{\eta \in \PCf(S)}\frac{i(\mu,\eta)}{i(\nu,\eta)}  = \sup_{\eta \in \PC(S)}\frac{i(\mu,\gamma)}{i(\nu,\gamma)}.
\]
If, furthermore, $\mu, \nu$ are both filling, then the supremum can be taken to be a maximum by compactness of $\PC(S)$~\cite{Bonahon88:GeodesicCurrent}.

The fact that $\operatorname{Curves}(S)$ is dense in $\PC(S)$, shows
\[
\Dil(\nu, \mu)  = \max_{\eta \in \PC(S)}\frac{i(\mu,\eta)}{i(\nu,\eta)}.
\]
Finally, we claim that $\operatorname{Curves}_f(S)$ is dense in $\PC(S)$, and hence
\[
\sup_{\gamma \in \operatorname{Curves}_f(S)}\frac{i(\mu,\gamma)}{i(\nu,\gamma)}= \max_{\eta \in \PC(S)}\frac{i(\mu,\eta)}{i(\nu,\eta)}.
\]
To prove the last claimed density, note that since $\PCf(S)$ is dense in $\PC(S)$, and since $\operatorname{Curves}(S)$ is also dense in $\PC(S)$, it follows that given any $\eta \in \PC(S)$, and any neighborhood $U$ of $\eta$ in $\PC(S)$, we can find $\eta_f \in U \cap \PCf(S)$. By density of $\operatorname{Curves}(S)$ in $\PC(S)$, we can take a sequence of closed curves $\gamma_n \to \eta_f$. The systole function is continuous on geodesic currents~\cite[Corollary~1.5(1)]{BIPP21:Currents}, and hence satisfies $\sys(\gamma_n) \to \sys(\eta_f)$ as $n$ goes to infinity. Since $\sys(\eta_f)>0$, for $n$ large enough it follows that $\sys(\gamma_n)>0$ and $\gamma_n \in U$, hence $\gamma_n \in U \cap \PCf(S)$, and since $\gamma_n \in \Curves(S)$, it follows that $\gamma_n \in U \cap \operatorname{Curves}_f(S)$.
\end{proof}
Note that $\Dil(\mu, \nu)$ could be infinite, eg, if $\nu$ is filling but $\mu$ is not. 

The
\textbf{extended asymmetric Thurston metric} is the asymmetric metric
\begin{equation}
\dasym([\mu],[\nu])
=
\log\!\left(
\Dil(\mu,\nu)\frac{h(\nu)}{h(\mu)}
\right),
\label{eq:asym_metric}
\end{equation}
studied by Jyothis and the first author, which extends Thurston's classical metric on $\Teich(S)$ \cite{JMG26:Horoboundary,CRS24:Joint}.
However, to study geometric rank, it is more natural to work with a symmetrized version, the \textbf{extended symmetrized Thurston metric}---which we will call the \textbf{Lipschitz metric} for short---given by
\begin{equation}
\dsym([\mu],[\nu])
\coloneqq \dasym([\mu],[\nu]) + \dasym([\nu],[\mu]) =
\log\!\left(
\Dil(\mu,\nu)\Dil(\nu,\mu)
\right).
\label{eq:sym_metric}
\end{equation}
The metric $\dsym$ was originally introduced by the second author in the setting of geodesic currents~\cite{Sap22:ExtensionThurston} and Cantrell--Reyes from the more general perspective of invariant pseudometrics~\cite{COR22:Manhattan}). 
From the definition, we have
\begin{equation}
\dasym([\nu], [\mu]) \leq \dsym([\nu], [\mu])
\label{eq:asym_sym_comp}
\end{equation}
for all $[\nu],[\mu] \in \PCf(S)$.
The topology induced by $\dsym$ on $\PCf(S)$ coincides with the restriction of the weak$^*$-topology (see~\cite{Sap22:ExtensionThurston} and \cite[Lemma~7.3]{CMGR26:Flows-v1} for details).

\subsubsection*{Non-filling geodesic currents and decomposition}

The space of \textbf{non-filling} geodesic currents will be denoted by $\Cnf(S) \coloneqq \C(S) \backslash \Cf(S)$ and its projectivization by $\PCnf(S)$.
Every non-filling geodesic current admits a canonical decomposition~\cite[Theorem~1.2]{BIPP21:Currents} which can be written (see~\cite[Theorem~1.7]{JMG26:Horoboundary} for the following particular restatement) as
\begin{equation}
\label{eq:current_decomposition}
\mu=\mu_{S'}+\lambda_\mu,
\end{equation}
where $\mu_{F}$ is a filling geodesic current on a proper subsurface with boundary
$F\subset S$ and $\lambda_\mu$ is a measured lamination with support disjoint from $F$.

\subsection{Rank, asymptotic cones, asymptotic dimension}

\subsubsection{Geometric and flat rank}
\label{subsec:ranks}

We first recall the definition of quasi-isometry and coarse isometry.

\begin{defi}
    Let $(X,d_1)$, $(Y, d_2)$ be metric spaces, and let $f \colon X \to Y$. Let $K \geq 1, C \geq 0$. Then we say $f$ is a \textbf{quasi-isometric embedding}  if for all $x_1, x_2 \in X$,
    \[
     \frac 1K d_2(f(x_1), f(x_2)) - C \leq d_1(x_1, x_2) \leq K d_2(f(x_1), f(x_2)) + C.
    \]
    We say $f$ is a \textbf{coarse isometry} if we can choose $K = 1$.

    Moreover, if there exists a $D$ so that, for all $y \in Y$, there is some $x \in X$ so that $d_2(f(x), y) < D$, then we say $X$ is $(K,C)$- \textbf{quasi-isometric} to $Y$, or \textbf{coarsely isometric} to $Y$, respectively. We write $X \simeq_{(K,C)} Y$ or $X \simeq_{(1,C)} Y$, respectively. 
\end{defi}

\subsubsection{Rank}

We will consider two quasi-isometrically invariant notions of rank for a metric space and its subspaces.
They are standard notions that, in particular, have been considered in Teichm\"uller geometry, for example, in~\cite{BrockFarb2006,EMR17:Rank}.

\begin{defi}
Let $(X,d)$ be a metric space. We say $X$ has \textbf{geometric rank} (or just \textbf{rank}) at least $n$, and we write $\rank(X) \geq n$, if there exist constants $K,C>0$ so that for all $R>0$
we have a $(K,C)$-quasi-isometric embedding $f_R \colon ([0,R]^n,d_\E) \to (X,d)$, 
where $d_\E$ denotes the Euclidean metric.
We say that 
\[
\rank(X)=n
\] 
if $n$ is the largest natural number (or zero) with that property. If no such largest number exists, we say $\rank(X)=\infty$. If $A \subset X$ is a subspace of $X$ with the restricted $d$ metric, we say $\rank(A)=n$ if its geometric rank with respect to $d$ is $n$.
\label{def:geo_rank}
\end{defi}

We will use that the notion of rank behaves as expected under finite unions, a property which we were not able to find in the literature.

\begin{prop}
\label{prop:rank_union}
Let $X_1$ and $X_2$ be subsets of a metric space. Then
\[
\operatorname{rank}(X_1\cup X_2)
=
\max\{\operatorname{rank}(X_1),\operatorname{rank}(X_2)\}.
\]
\end{prop}

\begin{proof}
By monotonicity of rank,
\[
\max\{\operatorname{rank}(X_1),\operatorname{rank}(X_2)\}
\leq
\operatorname{rank}(X_1\cup X_2).
\]
It therefore remains to prove the reverse inequality.

Suppose that
\[
\operatorname{rank}(X_1\cup X_2)= n.
\]
Then there exist constants $\lambda\geq 1$ and $C\geq 0$ such that, for every
sufficiently large $R$, there is a $(\lambda,C)$--quasi-isometric embedding
\[
f_R\colon [0,R]^n\longrightarrow X_1\cup X_2.
\]

In what follows, given a subset $A \subset X$, let $N_d(A) \coloneqq \{ x \in X \mid d(x,A) < d\}$ denotes its open $d$-neighborhood. We consider two cases.

\smallskip

\noindent
\textbf{Case 1.}
There exist $d\geq 0$ and $R_0>0$ such that
\[
f_R([0,R]^n)\subset N_d(X_1)
\]
for every $R\geq R_0$.

Choose a map
\[
\pi\colon N_d(X_1)\longrightarrow X_1
\]
such that
\[
d(x,\pi(x))\leq d+1
\]
for every $x\in N_d(X_1)$. Then, for all $x,y\in N_d(X_1)$,
\[
\bigl|d(\pi(x),\pi(y))-d(x,y)\bigr|\leq 2d+2.
\]
Hence
\[
\pi\circ f_R\colon [0,R]^n\longrightarrow X_1
\]
is a $(\lambda,C+2d+2)$--quasi-isometric embedding for every $R\geq R_0$.
Therefore
\[
\operatorname{rank}(X_1)\geq n.
\]
\smallskip

\noindent
\textbf{Case 2.}
For every $d\geq 0$ and every $R_0>0$, there exists $R\geq R_0$ such that
\[
f_R([0,R]^n)\not\subset N_d(X_1).
\]

Let $L>0$ be arbitrary. Choose $d$ such that
\[
d>\lambda\sqrt{n}\,L+C.
\]
Applying the assumption with $R_0=2L$, we obtain $R\geq 2L$ and a point
$q\in[0,R]^n$ such that, writing $p=f_R(\vec q)$,
\[
d(p,X_1)>d.
\]
Choose a subcube $Q\subset[0,R]^n$ of side length $L$ containing $\vec q$. For every
$\vec z\in Q$,
\[
d_{\mathbb E}(\vec z,\vec q)\leq \sqrt{n}\,L,
\]
and therefore
\[
d(f_R(\vec z),p)
\leq \lambda d_{\mathbb E}(\vec z,\vec q)+C
\leq \lambda\sqrt{n}\,L+C
<d.
\]
Thus
\[
f_R(Q)\subset B(p,d).
\]
Since $d(p,X_1)>d$, we have
\[
B(p,d)\cap X_1=\varnothing.
\]
Because $f_R(Q)\subset X_1\cup X_2$, it follows that
\[
f_R(Q)\subset X_2.
\]
The restriction
\[
f_R|_Q\colon Q\longrightarrow X_2
\]
is still a $(\lambda,C)$--quasi-isometric embedding. Since $Q$ is isometric to
$[0,L]^n$, we obtain a $(\lambda,C)$--quasi-isometric embedding
\[
[0,L]^n\longrightarrow X_2.
\]
As $L>0$ was arbitrary, it follows that
\[
\operatorname{rank}(X_2)\geq n.
\]

In either case,
\[
n\leq
\max\{\operatorname{rank}(X_1),\operatorname{rank}(X_2)\}.
\]
Therefore
\[
\operatorname{rank}(X_1\cup X_2)
\leq
\max\{\operatorname{rank}(X_1),\operatorname{rank}(X_2)\},
\]
which proves the claim.
\end{proof}

A related definition of rank is the notion of largest quasi-isometrically embedded flat. 

\begin{defi}\label{def:fat_rank}An $n$-\textbf{flat} (or $n$-\textbf{quasi-flat}) in $X$ is a quasi-isometric embedding of the $n$-orthant $(\mathbb{R}^n_+, d_\E)$ into $X$ (for $n \in \mathbb{N}$ or zero).
Given a metric space $X$, the \textbf{flat rank} of $X$ is the largest $n$ for which $X$ admits a $n$-flat.
\end{defi}
The following property is immediate.
\begin{lem}
\label{lem:compare_rank}
For every metric space $X$, we have
\[
\rankF(X) \leq \rank(X).
\]
\end{lem}
\begin{proof}
Let $n=\rank(X)$, and $f \colon \mathbb{R}^n_+ \to X$ a quasi-isometric embeddings realizing the flat rank. For every $R>0$, we can restrict $f$ to $[0,R]^n$ to obtain the desired quasi-isometric embeddings $f_R$, showing that $\rank(X) \geq n$.
\end{proof}

\subsubsection{Asymptotic cones}
\label{subsec:asymptotic_cone}
Asymptotic cones capture the modes of divergence of a metric space and will be useful in studying geometric rank. Indeed, families of uniformly quasi-isometrically embedded cubes in a metric space will promote to bilipschitz embedded cubes in its asymptotic cone.

We first recall the definition of asymptotic cone. See~\cite[Chapter I, Definition~5.50]{BH11:NonPosCurvature}.

An \textbf{ultrafilter} on a set $X$ is a collection $\omega$ of subsets of $X$ such that:
\begin{itemize}
    \item $\emptyset \notin \omega$ and $X \in \omega$;
    \item if $A,B \in \omega$, then $A \cap B \in \omega$;
    \item if $A \in \omega$ and $A \subseteq B \subseteq X$, then $B \in \omega$;
    \item for every subset $A \subseteq X$, either $A \in \omega$ or
    $X \setminus A \in \omega$, but not both.
\end{itemize}
An ultrafilter $\omega$ is \textbf{principal} if there exists $x \in X$ such that
\[
    \omega = \{A \subseteq X : x \in A\}.
\]
Otherwise, it is \textbf{nonprincipal}.

Let $\omega$ be a nonprincipal ultrafilter on $\N$. We say $A \subset \N$ is \textbf{$\omega$-large} if $A \in \omega$. Then for any functions $f, g: \N \to \R$, we say $f = O_\omega(g)$ if there is some $C > 0$ so that the subset $\{ i \in \N \colon |f(i)|<C|g(i)| \} \in \omega$, in other words, $|f(i)|<C|g(i)|$ on some $\omega$-large set. We say $f(i)$ \textbf{$\omega$-converges} to $L$ and write $f(i) \to_\omega L \in \R \cup \{\infty\}$ if for every $\epsilon>0$, the set $\{i \colon d(f(i),L)<\epsilon\} \in \omega$. In particular, when $\omega$ is a nonprincipal ultrafilter, all convergent sequences also $\omega$-converge, since $\omega$ contains every cofinite set.

\begin{defi}
\label{def:min-cone}
Let $(X,d)$ be a metric space. Fix a nonprincipal ultrafilter $\omega$ on $\N$, and for every $i \in \N$, basepoints $X_i \in X$, and scales $R_i\to_\omega\infty$. 
A sequence of points $(Y_i)$ where $Y_i \in X$ is
\textbf{$d$-admissible} if
\[
 d(X_i,Y_i)=O_\omega(R_i).
\]
Two admissible sequences $(Y_i),(Z_i)$ are equivalent if
\[
 \lim_\omega\frac{d(Y_i,Z_i)}{R_i}=0.
\]
The asymptotic cone
\[
 \Cone_\omega(X,X_i,R_i;d)
\]
is the set of equivalence classes of admissible sequences, with metric
\begin{equation}
 d_\omega([Y_i],[Z_i])
 =\lim_\omega\frac{d(Y_i,Z_i)}{R_i}.
 \label{eq:min-cone-distance}
\end{equation}
\end{defi}

The following lemma shows that the growing $n$-cubes in the definition of geometric rank
give an $n$-dimensional (bilipschitz) cube in the cone.

\begin{lem}
\label{lem:min-cube-cone}
Let $Z$ be a metric space.  Suppose there exist fixed $K \geq 1,C \geq 0$, a sequence of radii
$R_i\to\infty$, and a family of
$(K,C)$--quasi-isometric embeddings
\[
 f_i:[0,R_i]^m\longrightarrow Z.
\]
Then, for every nonprincipal ultrafilter $\omega$, the asymptotic cone $\Cone_\omega(Z,f_i(\vec 0),R_i;d)$ contains a compact subset
bilipschitz to $[0,1]^m$.
\end{lem}
\begin{proof}
If $\mathcal Y$ denotes this asymptotic cone, define
\[
 \mathcal F:[0,1]^m\longrightarrow\mathcal Y,
 \qquad \mathcal F(x)=[f_i(R_i\vec x)].
\]

We first check that the sequence $(f_i(R_i \vec x))$ is admissible.
The quasi-isometry inequalities are given by
\begin{equation}
\frac1K\|\vec u-\vec v\|_\infty-C
\leq d_Z\left(f_i(\vec u),f_i(\vec v)\right)
\leq K\|\vec u-\vec v\|_\infty+C
\label{eq:qi_admissible}
\end{equation}
for all $\vec u, \vec v \in [0,R]^m$ with uniform constants $K,C$ for every $i \in \mathbb{N}$.

For $\vec x \in [0,1]^m$, the point $R_i \vec x$ belongs to $[0,R_i]^m$. Applying~\eqref{eq:qi_admissible} with $\vec u=R_i\vec  x$ and $\vec v=\vec 0$ gives
\[
d_Z\left(f_i(R_i \vec x),f_i(\vec 0)\right) \leq
KR_i  \|\vec x\|_\infty + C.
\]
Since $\|\vec x\|_\infty \leq 1$, it follows that
\[
d_Z\left(f_i(R_i \vec x),f_i(\vec 0)\right) \leq KR_i + C =O(R_i).
\]
Thus $(f_i(R_i\vec x))$ represents a point of the asymptotic cone and so $F$ is well-defined.

Next, applying~\eqref{eq:qi_admissible} with $\vec u=R_i \vec x$ and $\vec v=R_i \vec y$, we get
\[
\frac{R_i}{K} \|\vec x-\vec y\|_\infty -C \leq d_Z\left(f_i(R_i \vec x),f_i(R_i \vec y)\right)
\leq KR_i \|\vec x-\vec y\|_\infty + C.
\]
Dividing by $R_i$, and taking the $\omega$-limit, $C/R_i$ goes to $0$ and thus
\begin{equation}
 \frac1K\|\vec x-\vec y\|_\infty
 \le d_\omega(\mathcal F(\vec x),\mathcal F(\vec y))
 \le K\|\vec x-\vec y\|_\infty.
 \label{eq:min-cube-bilip}
\end{equation}
The upper inequality shows $F$ is Lipschitz, and the lower inequality that it is injective and the inverse on its image is Lipschitz, so $F$ is a Lipschitz embedding.
Thus $\mathcal F([0,1]^m)$ is the required compact subset.
\end{proof}

\subsubsection{Asymptotic dimension}
\label{subsec:adim}
Asymptotic dimension is a coarse metric analog of topological dimension. It was originally introduced by Gromov in~\cite[Section~1.E]{Gromov1993Asymptotic}, although in a slightly different form from its modern version, which can be found in~\cite[Section~3.1]{BellDranishnikov2011}.
Informally, a metric space $X$ has \textbf{asymptotic dimension} at least $n$ if, at every chosen scale, $X$ can be covered by clusters of $n+1$ colors, so that the pieces of every monochromatic cluster are coarsely disjoint relative to the scale, and have bounded diameter.

\begin{defi} 
\label{defi:adim}
Let $(X,d)$ be a metric space. Let $r, D > 0$. Let $\mathcal U$ be a cover of $X$ by open sets  so that every $U \in \mathcal{U}$ has diameter at most $D$. Then we say $\mathcal U$ is \textbf{$D$-bounded}. Suppose further that $\mathcal{U}$ decomposes into $n+1$ subfamilies (colors) $\mathcal U = \mathcal U^0 \cup \cdots \cup \mathcal U^n$ so that, for each $i=0,\cdots,n$, any $U_i, V_i \in \mathcal U^i$ satisfy $d(U_i, V_i)\geq r$. Then we say $\mathcal U$ is an \textbf{$r$-separated} cover of \textbf{multiplicity} $n+1$.

We say that the \textbf{asymptotic dimension} of $X$ is at most $n$ if for each $r> 0$ there exists a $D > 0$ and a $D$-bounded and $r$-separated cover $\mathcal U$ of multiplicity $n+1$.
We define 
\[\adim(X) = n\]
if $n$ is the smallest non-negative integer satisfying the above. If no such $n$ exists, we say $X$ has infinite asymptotic dimension, or $\adim(X)=\infty$.
\end{defi}

This definition is the coarse-metric analog of a version of topological dimension, defined as follows. 

\begin{defi}
\label{defi:lebesgue_dim}
    An open cover $\mathcal{V}$ of a topological space $X$ is a \textbf{refinement} of another open cover $\mathcal{U}$ if for every $V \in \mathcal V$ there exist $U \in \mathcal U$ so that $V \subset U$. The \textbf{Lebesgue topological covering dimension} of $X$ is the smallest non-negative integer $n$ with the property that each open cover of $X$ has a refinement in which no point is included
in more than $n+1$ elements of the cover.  We say such a refined cover has \textbf{multiplicity} $n+1$. We will also refer to this dimension as \textbf{topological dimension} for short, and denote it by $\dim(X)$.
\end{defi}

\begin{example}
We list the asymptotic dimensions of a few standard spaces. For detailed references, we commend the reader to~\cite{Bell2017} as well as~\cite[Section~1.E]{Gromov1993Asymptotic}, \cite{BellDranishnikov2011}.
\begin{itemize}
\item If $X$ is bounded, then $\adim(X)=0$. However, we warn the reader that the converse is not true: consider $X=\{ 10^n \colon n \in \N \}$, with the restricted Euclidean metric on $\R$. This set is unbounded but $\adim(X)=0$ (details are left to the reader).
\item If $X$ is a real tree, then $\adim(X)=1$.
\item Gromov hyperbolic groups or spaces with a geometric action of a Gromov hyperbolic group have finite asymptotic dimension.
\item  For all $n \in \mathbb N$, $\adim(\mathbb R^n)=\adim(\mathbb H^n)=n$ where $\R^n$ is the $n$-Euclidean space, and $\H^n$ is the $n$-hyperbolic space.
\end{itemize}

However, we emphasize that in general topological dimension and asymptotic dimension have little to no relation to each other: 
\begin{itemize} \item $\dim(\mathbb{Z}^n)=0$ and $\adim(\mathbb{Z}^n)=n$
\item while $\dim([0,1]^n)=n$ and $\adim([0,1]^n)=0$ for every $n \in \mathbb N$.
\item There are examples of geodesic and uniformly contractible smooth Riemannian metric spaces homeomorphic to $\mathbb{R}^n$ with infinite asymptotic dimension~\cite{DFW03,Dra00}.
\end{itemize}
\end{example}

Asymptotic dimension is a quasi-isometric invariant~ and it behaves monotonically under quasi-isometric embeddings.
\begin{lem}[{\cite[Propositions~22 and 23]{BellDranishnikov2011}}]
\label{lem:adim_qi}
Let $X,Y$ be metric spaces, and let $f \colon X \to Y$ be a map.
\begin{itemize}
\item If $f$ is a quasi-isometric embedding, then $\adim(X) \leq \adim(Y)$.
\item If $f$ is a quasi-isometry, then $\adim(X)=\adim(Y)$.
\end{itemize}
\end{lem}

From \Cref{lem:adim_qi}, the following result follows immediately.

\begin{lem}
\label{lem:flat_rank_lower}
Let $X,Y$ be metric spaces and $f \colon X \to Y$ a quasi-isometric embedding.
Then,
\[
\rankF(X) \leq \adim(Y).
\]
\end{lem}

Asymptotic dimension, similar to geometric rank (Proposition~\ref{prop:rank_union}), obeys a union rule under finite unions.

\begin{lem}[{\cite[Corollary~26]{BellDranishnikov2011}}]
\label{lem:finite-union-adim}
Let $X_1,\cdots,X_n$ be subspaces of a metric space $(X,d)$.
Then
\[
\adim(\cup_{i=1}^n X_i) \leq \max_{i=1,\cdots,n} \{ \adim(X_i) \}.
\]
\end{lem}

Finally, geometric rank provides a lower bound to asymptotic dimension, a result that we have not found written in the literature.

\begin{lem}
\label{lem:rank_adim}
Let $X$ be a metric space.
Then
\[
\rank(X) \leq \adim(X).
\]
\end{lem}
\begin{proof}
Suppose $f_i \colon [0,R_i]^n \to X$ is a sequence of $(\lambda,K)$-quasi-isometric embeddings, with $R_i \to \infty$ as $i \to \infty$.
Suppose $\adim(X)=m$. For $r>2\lambda + K$ choose $D = D(r) > 0$, and let $\mathcal U$ be a $D$-bounded $r$-separated cover of multiplicity $m+1$ of $X$.
For each $i \in \mathbb{N}$, and $U \in \mathcal U$, define an open subset of the cube $[0,R_i]^n$ by
\[
V_{i,U} \coloneqq N_1(f_i^{-1}(U)) \subset [0,R_i]^n
\]
where $N_1$ denotes the open 1-neighborhood in the Euclidean cube.
The sets $(V_{i,U})_{U \in \mathcal U}$ form an open cover $[0,R_i]^n$, since the sets $U \in \mathcal U$ cover $X$.
Furthermore, for fixed $i,j$, consider the family
\[
\{ V_{i,U} \colon U \in \mathcal U^j \}.
\]
Fix $i$. We claim for every $j=0,\cdots,m$, given $V_{i, U}, V_{i, U'}$ with $U, U' \in \mathcal U^j$ and $U \neq U'$, then $V_{i,U} \cap V_{i,U'} =\emptyset$.
Indeed, suppose $\vec z \in V_{i,U} \cap V_{i,U'}$, for $U \neq U'$. Then there exist $\vec x, \vec x' \in [0,R_i]^n$ so that $d_{\mathbb E}(\vec z, \vec x), d_{\mathbb E}(\vec z,\vec x') < 1$, and $f_i(\vec x) \in U$, $f_i(\vec x') \in U'$.
Hence $d_{\mathbb E}(\vec x, \vec x') <2$. So, since $f_i$ is a $(\lambda,K)$-quasi-isometry, we get
\[
d_X(f_i(\vec x),f_i(\vec x')) \leq 2\lambda + K < r,
\]
contradicting that $U, U' \in \mathcal U^j$ are $r$-separated.
Hence, the full open over $\{ V_{i, U} \colon U \in U\}$ of $[0,R_i]^n$ has multiplicity at most $m+1$ (in the sense of \Cref{defi:lebesgue_dim}).
Furthermore, we claim the diameters of $V_{i,U}$ are uniformly bounded independent of $i$.
Indeed, if $\vec z,\vec z' \in V_{i,U}$, choose $\vec x,\vec x' \in f_i^{-1}(U)$ with $d_{\mathbb E}(\vec z,\vec x)<1$ and $d_{\mathbb E}(\vec z',\vec x')<1$.
Since $f_i(\vec x),f_i(\vec x')\in U$, and $U$ has diameter bounded by $D$, it follows that
\[
d_X(f_i(\vec x),f_i(\vec x'))\leq D.
\]
By $f_i$ being a quasi-isometric embedding, we get
\[
\frac{1}{\lambda} d_{\mathbb E}(\vec x, \vec x')-K \leq D,
\]
and thus
\[
d_{\mathbb E}(\vec z, \vec  z') \leq M \coloneqq 2 + \lambda (D+K)
\]
as desired.
For each $i \in \mathbb{N}$, rescale by $R_i$  the cover $\{ V_{i,U} \colon U \in \mathcal{U}\}$ of $[0,R_i]^n$, to obtain covers $\mathcal V_i$ of $[0,1]^n$ with multiplicity $m+1$ and diameter bounded by $D_i \coloneqq M/R_i$.
Since $D_i \to 0$ as $i \to \infty$, 
it follows that $[0,1]^n$ admits arbitrarily fine open covers of multiplicity at most $m+1$ and thus, by the definition of Lebesgue topological covering dimension (\Cref{defi:lebesgue_dim}), the topological dimension of $[0,1]^n$ is at most $m$. So, we get
\[
\rank(X)=n = \dim \left([0,1]^n \right) \leq m=\adim(X).\qedhere
\]
\end{proof}

From \Cref{lem:flat_rank_lower} and \Cref{lem:rank_adim}, we have the following chain of inequalities.

\begin{lem} \label{lem:rankF<rank<asdim}
For any metric space $(X,d)$, we have
\[
\rankF(X) \leq \rank(X) \leq \adim(X).
\]
\end{lem}

\begin{rem}
These inequalities are strict in most cases:
\begin{itemize}
\item To see the first inequality, consider the real tree obtained as a wedge of finite intervals of increasing lengths $X_1\coloneqq\bigvee_{i=1}^{\infty} [0,n]$ each equipped with the usual Euclidean metric, for which $\rankF(X_1)=0$ and $\rankF(X_1)=\adim(X_1)=1$. 
\item For the second inequality, take $(X_2,d)$ with $X_2\coloneqq\mathbb{H}^n$ the hyperbolic $n$-space for $n\geq 2$, and $d$ the standard hyperbolic metric, for which $\rankF(X_2)=\rank(X_2)=1$ and $\adim(X_2) =n \geq 2$ for all $n \in \N$.
\end{itemize}
More sophisticated examples give the two strict inequalities at once.  We leave the details to the reader.
\end{rem}
\section{Simplices of geodesic currents}
\label{sec:simplices_currents}
Each sequence $\vec \mu = (\mu_1, \dots, \mu_n)$ of geodesic currents spans a \textit{simplex} $\PSimplex(\vec \mu)$, that we formally define below. We then define what it means for a sequence of currents to be \textit{increasing} and \textit{strictly increasing}, and establish some basic properties of such sequences. We will use these special types of sequences to bound the geometric rank of $\PSimplex(\vec \mu)$ for all $\vec \mu$.

\subsection{Simplices of geodesic currents}
We say $\vec \mu=(\mu_0,\cdots,\mu_n)$ is a \textbf{jointly filling} sequence of geodesic currents if $\sum_{i=0}^n\mu_i \in \C_f(S)$. Note that the individual geodesic currents do not themselves need to be filling but any positive linear combination of all of them will be.

Given a vector $\vec x = (x_0, \dots, x_n) \in \R^{n+1}_+$, where
\[
\R^{n+1}_+ \coloneqq \{ (x_0,\cdots,x_n) \mid x_i > 0,\ i=0,\cdots,n\},
\]
we use the notation $\vec x \cdot \vec \mu$ to denote
\begin{equation}
\vec x \cdot \vec \mu = \sum_{i=0}^n x_i \mu_i.
\label{eq:scalar_product}
\end{equation}

\begin{defi}
\label{def:simplex}
Let $\mu_0, \dots, \mu_n \in \C(S)$. We define the \textbf{simplex} $\PSimplex(\vec\mu) \subset \PC(S)$ by
\[
\PSimplex(\vec\mu)  = \{ [\vec x \cdot \vec \mu] \ | \ \vec x \in \R^{n+1}_+\}.
\]
If we want to consider the \textbf{cone} generated by $\vec\mu$ in the space of geodesic currents we will  write $\Simplex(\vec \mu) := \{\vec x \cdot \vec \mu \ | \ \vec x \in \R^{n+1}_+\}$.
\end{defi}

Geodesic currents $\mu_0,\cdots,\mu_n$ are said to be \textbf{linearly independent} if they are linearly independent vectors in the vector space of \textbf{signed geodesic currents}  $\C^s(S)\coloneqq \C(S) - \C(S)$. Intersection number extends linearly to signed geodesic currents~\cite[Appendix~A.1]{JMG26:CurrentsNegative}.

We refer to $\PSimplex(\vec\mu)$ as a  \emph{simplex} even if the
generating geodesic currents are not known to be linearly independent: if they are linearly dependent, $\PSimplex(\vec \mu)$ is an actual simplex, but of topological dimension strictly lower than $n$ (equivalently, the affine dimension of the cone $\Simplex(n)$ is less than $n+1$). In particular, its linear basis will \emph{not} be $\vec \mu$. Thus, in general, $\PSimplex(\vec\mu)$ is the image of the standard
coefficient simplex under the map from $\R^{n+1}_+ \longrightarrow \PC_f(S)$ defined by
\[
\vec x\longmapsto
[\vec x \cdot \vec \mu]
\]
and this map need not be injective. When the geodesic currents
$\mu_0,\ldots,\mu_n$ are linearly independent as signed geodesic currents, $\PSimplex(\vec \mu)$ is an actual
\textbf{$n$-simplex} of the vector space of signed geodesic currents,
that is, the polytope obtained by taking the convex hull of the $n+1$ linearly independent vectors $\mu_0,\ldots,\mu_n$.

 If $\vec \mu$ is jointly filling, then $\Simplex(\vec \mu)$ lies inside $\Cf(S)$, and thus $\PSimplex(\vec \mu) \subset \PCf(S)$. In this case, we say that the simple $\PSimplex(\vec \mu)$ is \textbf{filling}.

\subsection{Increasing and strictly increasing simplices}
We now define increasing and strictly increasing sequences of sequences of geodesic currents. We will later show that any filling simplex decomposes, up to quasi-isometry, into a union of filling simplices on increasing sequences of geodesic currents (\Cref{cor:decomp_increasing_seq}). Moreover, if a jointly filling sequence is strictly increasing, then its simplex has geometric rank equal to its dimension (\Cref{lem:qi_to_Rn}).

\begin{defi}
 Let $\mu, \mu' \in \C(S)$. We say $\mu$ is \textbf{less than} $\mu'$, and we write $\mu \inc \mu'$, if:
 \begin{align*}
\Dil(\mu', \mu)  & < \infty,  \text{ but, }\\
  \Dil(\mu, \mu')  & = \infty.
\end{align*}
Moreover, we say $\mu$ is \textbf{strictly less than} $\mu'$, and write $\mu \sinc \mu'$, if:
 \begin{align*}
\forall \gamma \in \C(S), \quad & i(\mu, \gamma)  \leq i(\mu', \gamma),  \text{ and, }\\
\exists \gamma_0 \in \C(S), \quad & i(\mu, \gamma_0) = 0,  \text{ but } i(\mu', \gamma_0) \neq 0.
\end{align*}
We say a sequence $\vec \mu = (\mu_0, \dots, \mu_n)$ is \textbf{increasing} if $\mu_{i-1} \inc \mu_i$ for all $i = 1, \dots, n$, and \textbf{strictly increasing} if $\mu_{i-1} \sinc \mu_i$ for all $i = 1, \dots, n$.
\label{def:increasing_sequence}
\end{defi}

The following lemma follows immediately from Definition~\ref{def:increasing_sequence}.

\begin{lem}
\begin{enumerate}
\item If $\mu \sinc \mu'$ then $\mu \inc \mu'$.
\item If a sequence $\vec \mu = (\mu_0, \dots, \mu_n)$ is increasing, and for each $i=1,\cdots,n$, there is some $\gamma_i$ so that $i(\mu_{i-1}, \gamma_i) = 0$ but $i(\mu_i, \gamma_i) \neq 0$, then we can find positive scalars $d_1, \dots, d_n$ so that $(\mu_0, d_1 \mu_1, \dots ,d_n \mu_n)$ is strictly increasing.
\end{enumerate}
\label{lem:increasing_properties}
\end{lem}
\begin{proof}
For the first item, the existence of a geodesic current $\gamma_0$ so that $i(\mu, \gamma_0) = 0$ but $i(\mu', \gamma_0) \neq 0$ implies that $\Dil(\mu,\mu')  = \infty$.
The second item follows by choosing $d_i \coloneqq\prod_{j=1}^i \Dil(\mu_i, \mu_{i+1})$ for $i=1,\cdots,n$.
\end{proof}

\begin{lem}
 Both $\inc$ and $\sinc$ define a strict partial order on $\C(S)$.
\label{lem:partial_order}
 \end{lem}
\begin{proof}
\begin{itemize}
\item First we prove that $\inc$ is a strict partial order. Irreflexivity is clear, since $\Dil(\mu,\mu')$ is either finite or infinite. For transitivity, we have
\begin{equation}
\Dil(\mu,\mu'') \leq \Dil(\mu,\mu') \Dil(\mu',\mu'')
\label{eq:dil_inequality}
\end{equation}
for arbitrary geodesic currents $\mu, \mu', \mu''$.
Hence, if $\mu \inc \mu'$ and $\mu' \inc \mu''$, we have, respectively, $\Dil(\mu', \mu) < \infty$ and $\Dil(\mu'',\mu')< \infty$. Therefore, by Eq.~\eqref{eq:dil_inequality}, $\Dil(\mu'',\mu)<\infty$.
We also claim that $\Dil(\mu,\mu'') =\infty$. Indeed, if $\Dil(\mu,\mu'') < \infty$, then again by Eq.~\eqref{eq:dil_inequality} we would have
\[
\Dil(\mu,\mu') \leq \Dil(\mu,\mu'')\Dil(\mu'',\mu').
\]
Since we assumed that $\Dil(\mu'',\mu')<\infty$, it would follow that $\Dil(\mu,\mu')<\infty$, a contradiction.
\item Second, we prove $\sinc$ is a strict partial order. We first observe that it is irreflexive. Indeed, if
$\mu \sinc \mu$, then there would exist some geodesic current $\gamma_0$ such that
\[
i(\mu,\gamma_0)=0
\qquad\text{and}\qquad
i(\mu,\gamma_0)\neq 0,
\]
which is impossible.
We now prove transitivity. Suppose that
\[
\mu \sinc \mu'
\qquad\text{and}\qquad
\mu' \sinc \mu''.
\]
Then, for every $\gamma \in \C(S)$, we have
\[
i(\mu,\gamma)\leq i(\mu',\gamma)\leq i(\mu'',\gamma),
\]
so the first condition in the definition of $\mu\sinc\mu''$ holds.

Moreover, since $\mu\sinc \mu'$, there exists some geodesic current $\gamma_0$ such that
\[
i(\mu,\gamma_0)=0
\qquad\text{and}\qquad
i(\mu',\gamma_0)\neq 0.
\]
Since
\[
i(\mu',\gamma_0)\leq i(\mu'',\gamma_0),
\]
it follows that
\[
i(\mu'',\gamma_0)\neq 0.
\]
Therefore
\[
i(\mu,\gamma_0)=0
\qquad\text{and}\qquad
i(\mu'',\gamma_0)\neq 0,
\]
so $\mu\sinc\mu''$.\qedhere
\end{itemize}
\end{proof}

We note the following property, which implies that simplices spanned by increasing sequences are true simplicies in the vector space of signed geodesic currents. We note that the converse is not true: not every linearly independent set of currents is increasing (see Example~\ref{ex:singular_currents}).
\begin{lem}\label{lem:strictly-increasing-independent}
Every increasing sequence 
\[
    \mu_0 \inc \mu_1 \inc \cdots
    \inc \mu_n
\]
of nonzero geodesic currents is linearly independent as signed geodesic currents.
In particular, if $\vec \mu =(\mu_0,\cdots,\mu_n)$, then $\Simplex(\vec \mu)$ is a simplicial cone of linear dimension $n+1$ and thus
$\PSimplex(\vec \mu)$ has topological dimension $n$.
\end{lem}
\begin{proof}
Since the sequence $(\mu_0, \dots, \mu_n)$ is increasing, we have, in particular, that 
\[
\Dil(\mu_i, \mu_j) = \infty
\]
for all $i < j$.

Suppose, toward a contradiction, that the geodesic currents are linearly
dependent as signed geodesic currents. Then there exist real numbers $a_0,\ldots,a_n$, not all
zero, such that $\sum_{j=0}^{n}a_j\mu_j=0.$
Equivalently, after separating the positive and negative
coefficients, we have an equality of geodesic currents
\[
    \sum_{a_j>0}a_j\mu_j
    =
    \sum_{a_j<0}(-a_j)\mu_j.
\]
Let 
\[
\mu_+ =  \sum_{a_j>0}a_j\mu_j, \qquad \mu_- = \sum_{a_j<0}(-a_j)\mu_j.
\]
We will show that either $\Dil(\mu_+, \mu_-)$ or $\Dil(\mu_-, \mu_+)$ is infinite, contradicting the fact that $\mu_+ = \mu_-$.

Since no $\mu_i$ is equal to zero, at least one positive and one negative coefficient is non-zero. Let $m$ be the largest value so that $a_m > 0$, and $k$ be the largest value so that $a_k < 0$. Suppose without loss of generality that $m > k$.

Then we have
\begin{align*}
    \Dil(\mu_-, \mu_+) & = \Dil \left ( \sum_{a_j<0}(-a_j)\mu_j, \sum_{a_j>0}a_j\mu_j \right)\\
    & = \sup_\gamma \frac{\sum_{a_j>0}a_ji(\mu_j, \gamma)}{\sum_{a_j<0}(-a_j)i(\mu_j, \gamma)} \\
    & \geq  \sup_\gamma \frac{a_m i(\mu_m, \gamma)}{\sum_{a_j<0}(-a_j)i(\mu_j, \gamma)} \\
    & =  \sup_\gamma \frac{a_m \frac{i(\mu_m, \gamma)}{i(\mu_k, \gamma)}}{\sum_{a_j<0}(-a_j)\frac{i(\mu_j, \gamma)}{i(\mu_k, \gamma)}}. 
\end{align*}
Now, we assumed that $k \geq j$ for all $j$ so that $a_j < 0$. In particular, if $a_j < 0$, then 
\[
\sup_\gamma \frac{i(\mu_j, \gamma)}{i(\mu_k, \gamma)} < \infty
\]
as $\mu_j \inc \mu_k$ for $j < k$, and when $j = k$ we have $\Dil(\mu_k, \mu_k) = 1$. Thus, there is some $D > 0$ so that, for all $\gamma$, $\sum_{a_j<0}(-a_j)\frac{i(\mu_j, \gamma)}{i(\mu_k, \gamma)} < D$. Therefore,
\begin{align*}
 \Dil(\mu_-, \mu_+) & \geq \frac{a_m}{D} \sup_\gamma \frac{i(\mu_m, \gamma)}{i(\mu_k, \gamma)} \\
  & = \frac{a_m}{D} \Dil(\mu_k, \mu_m) \\
  & = \infty
\end{align*}
where the last equality follows from our assumption that $a_m > 0$ (in particular, $a_m \neq 0$), and $m > k$ so $\Dil(\mu_k, \mu_m) = \infty$.

 However, $\mu_+ = \mu_-$ implies that
\[
\Dil(\mu_+, \mu_-) = \Dil(\mu_-, \mu_+) =  1
\]
giving us a contradiction.
\end{proof}

\section{Introduction to Theorems \ref{mainthm:max_rank} and \ref{mainthm:criterion}}
\label{sec:outlineCD}
In this section, we give an idea of how to bound the geometric rank of a simplex. We also emphasize that the bound on geometric rank is in sharp contrast with the unboundedness of linear dimension. We highlight this disparity with an example in \Cref{sec:example_rank_dim}.
\subsection{Idea of proof of Theorems \ref{mainthm:max_rank} and \ref{mainthm:criterion}}
Let 
\[
\vec \mu = (\mu_0, \dots, \mu_n)
\]
be any jointly filling sequence of currents. We bound the geometric rank of $\PSimplex(\vec \mu)$ as follows:

\begin{enumerate}
    \item In \Cref{sec:upper_bound_on_rank_of_simplices}, we prove that the dimension of a filling simplex gives an upper bound on its geometric rank. In particular in \Cref{prop:bound_rank_by_length} we show that  
    \[
    \rank(\PSimplex(\vec \mu)) \leq n.
    \]
    In general, this upper bound is very far from sharp. See \Cref{ex:singular_currents}.
     \item In \Cref{sec:decompose_increasing_seq}, we show how to reduce the case of an arbitrary sequence to the case of increasing sequences. We show that if $\vec \mu$ is an \textit{arbitrary} jointly filling sequence of geodesic currents, then $\PSimplex(\vec \mu)$ is coarsely isometric to a finite union of simplices spanned by increasing sequences (\Cref{cor:decomp_increasing_seq}).

    \item In Sections \ref{sec:increasing_seq} and \ref{sec:bound_increasing_seq}, we show that that if $\vec \mu$ is a jointly filling and \textit{increasing} sequence (recall \Cref{def:increasing_sequence}) then 
    \[
    n \leq 8g-9
    \]
    (\Cref{lem:n_at_most_8g-0}). The proof relies on showing that any increasing sequence can be replaced by a strictly increasing sequence of the same length (\Cref{lem:replacement}), and then bounding the length of any strictly increasing sequence. The replacement technique uses a construction from \Cref{sec:maximal_sequence}, but otherwise Sections \ref{sec:increasing_seq} and \ref{sec:bound_increasing_seq} are largely self-contained. \Cref{sec:increasing_seq} is devoted to establishing the necessary technical lemmas, while \Cref{sec:bound_increasing_seq} contains the main ideas.
    
    This establishes that an \textit{arbitrary} filling simplex $\PSimplex(\vec \mu)$ is coarsely isometric to a finite union of simplices of rank at most $8g-9$.

    \item It follows by induction on \Cref{prop:rank_union} that if $X_1, \dots, X_k$ are $k$ subspaces of $\PCf(S)$ of rank at most $r$, then their union also has rank at most $r$. 
    Thus, for an \textit{arbitrary} jointly filling sequence $\vec \mu$, 
    \[
    \rank(\PSimplex(\vec \mu)) \leq 8g-9.
    \]
    This allows us to complete the proof of \Cref{mainthm:max_rank} in \Cref{sec:proof_thm_max_rank}.
    \item We also show that the bound $8g-9$ is tight. 
    
    In \Cref{sec:maximal_sequence}  we construct a jointly filling and strictly increasing sequence with $n = 8g-9$ (\Cref{lem:build_increasing}). A part of this construction is also used when we prove that any increasing sequence can be replaced by a strictly increasing sequence of the same length (\Cref{lem:replacement}).

    \item Then in \Cref{sec:strictly_increasing_seq}, we show that if $\vec \mu$ is \textit{strictly increasing}, then $\PSimplex(\vec \mu)$ is quasi-isometric to $\R^n_+$  (\Cref{lem:qi_to_Rn}), and so flat and geometric ranks, as well as asymptotic dimension, are exactly $n$. This allows us to complete the proof of \Cref{mainthm:criterion} in \Cref{sec:proof_of_str_inc_thm}.
\end{enumerate}

\subsection{Example of the disparity between dimension and rank}
\label{sec:example_rank_dim}
The uniform bound on geometric rank stands in sharp contrast to the unboundedness of the linear rank or, equivalently, topological dimension of simplices. Indeed, the cone of geodesic currents is contained in the vector space of signed geodesic currents, where one can form cones of arbitrarily large \emph{linear rank}, by an idea due to Glorieux~\cite[Theorem~1.2]{Glo19:Negatively}, which we reproduce as Example~\ref{ex:singular_currents} for completeness. By contrast, \Cref{mainthm:max_rank} shows that the flat rank and geometric rank of simplices are subject to a universal topological bound.

\begin{example}
\label{ex:singular_currents}
For any $n \in \mathbb{N}$, let $\mu_0,\cdots,\mu_n$ be any collection of $n+1$ pairwise singular geodesic currents. Then the cone $\Delta(\vec \mu)$ has linear rank $n+1$.
Indeed, suppose $\sum_{i=0}^n a_i \mu_i = 0$ for some collection of real numbers $a_i, i=0,\cdots,n$.
Since every geodesic current $\mu$ induces a geodesic flow invariant and orientation reversal invariant probability measure $\widetilde{\mu}$ on the unit tangent bundle $T^1 S$~\cite[Theorem~3.4]{ES22:GeodesicCount} (c.f.~\cite{Bonahon86:EndsHyperbolicManifolds, Bonahon88:GeodesicCurrent}), 
then there exist Borel subsets $A_i \subset T^1 S$ for $i=0,\cdots,n$ so that $\widetilde{\mu_i}(A_i)=1$ and $\widetilde{\mu_i}(A_j) = 0$ for all $i \neq j$. This forces $a_i=0$ for all $i=0,\dots,n$.
A concrete example is when the induced probability measures $\widetilde{\mu_i}$ are pairwise distinct and ergodic with respect to the geodesic flow. 
It thus follows that the topological dimension of $\PSimplex(\vec \mu)$ is $n$.
This happens, for example, if $\mu_i$ are: 
\begin{itemize}
\item Liouville currents associated to distinct points in Teichm\"uller space~\cite{Bonahon88:GeodesicCurrent};
\item more generally, Liouville currents associated to negatively curved Riemannian metrics with distinct marked length spectra~\cite{Otal90:SpectreMarqueNegative}; or
\item distinct filling multi-curves.
\end{itemize}
On the other hand, if $\mu_i \in \Cf(S)$ for $i=0,\cdots,n$ (as in the above three examples) then by \Cref{lem:dist_to_sum}, $\PSimplex(\vec \mu)$ is bounded, hence it has flat rank, geometric rank and asymptotic dimension equal to zero.
\end{example}

\section{Upper bound on rank of simplices}
\label{sec:upper_bound_on_rank_of_simplices}

We now give a general upper bound on the rank of any simplex.

\begin{prop}
   Let $\vec \mu = (\mu_0, \dots, \mu_n)$ be any jointly filling sequence of geodesic currents. Then the geometric rank and flat rank of $\PSimplex(\vec \mu)$ are at most $n$.
  \label{prop:bound_rank_by_length}
\end{prop}
This proposition is established in \Cref{cor:upper_bound_rank}.

For what follows, we need the following lemma about distances in simplices.

\begin{lem}
\label{lem:dist_to_sum}
 Let $\mu, \nu_1, \dots, \nu_n$ be filling geodesic currents. Let $t_1, \dots, t_n \in \R_+$. Then
 \[
  \dsym\left(\mu, \sum t_i \nu_i\right) \leq 2 \max_i \dsym(\mu, \nu_i).
 \]
 \end{lem}
Note that the distance $\dsym$ is defined for projective filling geodesic currents. However, the above lemma is scale-invariant: scaling any of the $\nu_i$ is equivalent to choosing a different coefficient $t_i$.

\begin{proof}
Let $\mu, \nu_1, \dots, \nu_n$ be filling geodesic currents.  Let $t_1, \dots, t_n \in \R_+$, and consider the geodesic current $\nu \coloneqq \sum t_i \nu_i$. 

Let $\omega_i \coloneqq h(\nu_i)\nu_i$, where $h(\cdot)$ is entropy, as defined in Eq.~\eqref{eq:entropy}. Then if $s_i \coloneqq \frac{t_i}{h(\nu_i)}\sum_j \frac{t_j}{h(\nu_j)}$ and $\omega\coloneqq \sum_i s_i \omega_i$, we have
\begin{align*}
h(\omega_i) & = 1\\
 [\omega_i] & = [\nu_i]  \implies \dsym(\mu, \omega_i) = \dsym(\mu, \nu_i)\\
 \left [\sum s_i \omega_i\right ] & = \left [\sum t_i \nu_i\right ] \implies \dsym(\mu, \omega) = \dsym(\mu, \nu) \\
 \sum s_i & = 1.
\end{align*}
In other words, we can assume WLOG that the entropy of each $\nu_i$ is one, and that the coefficients $t_1, \dots, t_n$ sum to 1.

Note that if $\sum t_i = 1$, then
\[
 \mu = \sum t_i \mu.
\]

Now let $\gamma$ be any filling geodesic current. We now bound:
\begin{align*}
 \frac{i(\nu, \gamma)}{i(\mu, \gamma)} & =  \frac{\sum t_i i(\nu_i, \gamma)}{\sum t_i i(\mu, \gamma)} \\
  & \leq \max_i \frac{i(\nu_i, \gamma)}{i(\mu, \gamma)}
\end{align*}
using \Cref{lem:ratiosum}. Proceeding similarly for the ratio $ \frac{i(\mu, \gamma)}{i(\nu, \gamma)} $, we get, using that $\Cf(S)$ is dense in $\C(S)$, 
\begin{align*}
 \dsym(\nu, \mu) & = \log \sup_\gamma \frac{i(\nu, \gamma)}{i(\mu, \gamma)} + \log \sup_\gamma \frac{i(\mu, \gamma)}{i(\nu, \gamma)} \\
  & \leq \log \sup_\gamma \max_i \frac{i(\nu_i, \gamma)}{i(\mu, \gamma)} + \log \sup_\gamma \max_i \frac{i(\mu, \gamma)}{i(\nu_i, \gamma)} \\
  & = \max_i \log \sup_\gamma \frac{i(\nu_i, \gamma)}{i(\mu, \gamma)} + \max_i \log \sup_\gamma  \frac{i(\mu, \gamma)}{i(\nu_i, \gamma)}\\
  & = \max_i d^{asym}(\mu, \nu_i) +\max_i d^{asym}(\nu_i, \mu) \\
  & \leq 2 \max_i  \dsym(\mu, \nu_i).
\end{align*}
Here we use that the entropies of $\nu_1, \dots, \nu_n$ are all 1, and so the terms $\log \sup_\gamma \frac{i(\nu_i, \gamma)}{i(\mu, \gamma)}$ and $\log\sup_\gamma \frac{i(\mu, \gamma)}{i(\nu_j, \gamma)}$ are exactly the asymmetric Thurston distances (Eq.~\eqref{eq:asym_metric}). Moreover, the asymmetric Thurston distance is smaller than the symmetrized distance by Eq.~\eqref{eq:asym_sym_comp}.
\end{proof}

We introduce two technical definitions that will be useful in our proof of \Cref{prop:bound_rank_by_length}. The following definition says roughly that quasi-isometric embeddings of integral cubes of any size can be approximated by continuous ones in a uniform way, that is, up to a constant that only depends on the quasi-isometry constants but not on the specific embedding.
\begin{defi}
Let $(X,d)$ be a metric space, and let $\lambda \geq 1, K, C > 0$ and $n \in \N$. We say that a family of $(\lambda, K)$-quasi-isometric embeddings
\[
f_i\colon  [0,i]^n \longrightarrow X
\]
for $i \in \N$ is \textbf{$C$-continuously approximable} if there is some  family of continuous functions 
\[
g_i\colon [0,i]^n \longrightarrow X
\]
so that $g_i$ is distance at most $C$ from $f_i$ for all $i$. That is, for all $\vec x \in [0,i]^n$,
\[
d_X\left(f(\vec x),g(\vec x)\right) \leq C.
\]

We say that a metric space $(X,d)$ admits \textbf{continuous uniform approximation
of integral $n$-cubes} if for every $\lambda \geq 1, K > 0$, there is some $C > 0$ so that every such family of $(\lambda, K)$-quasi-isometries is $C$-continuously approximable.
\end{defi}
\begin{rem}
    If $f$ is a $(\lambda, K)$ quasi-isometric embedding, and $g$ is a continuous function at distance $C$ from $f$, then $g$ is a continuous, $(\lambda,K+2C)$-quasi-isometric embedding.
\end{rem}

One way to construct continuous approximations for $n$-cubes of  the subspaces of (projective) geodesic currents we are ultimately interested in will be by interpolating using the convex structure of the space of geodesic currents.

We say $\mathcal{X} \subset \PCf(S)$ is \textbf{convex} if for every pair $[\alpha], [\beta] \in \mathcal{X}$ and $a,b \geq 0$, we have $[a\alpha + b \beta]\in \mathcal{X}$.

\begin{lem}
\label{lem:approximation_distance}
Let $n \in \N$. Suppose $\mathcal{X} \subset \PCf(S)$ be convex.
Then $\mathcal{X}$ admits continuous uniform approximations of $n$-cubes. In particular, for any $m \in \N$, $\lambda \geq 1$ and $K \geq 0$, any $(\lambda, K)$-quasi-isometric embedding $f: [0,m]^n \longrightarrow \mathcal X$ is $\frac 32 \lambda \sqrt n + K$-continuously approximable.
\end{lem}

\begin{proof}
Let $m \in \N$, and let $f \colon [0,m]^n \to \mathcal{X}$ be a $(\lambda,K)$-quasi-isometric embedding.  We will show that $f$ is $\frac 32  \lambda \sqrt n + K$-continuously approximable. Note that our constant depends on $\lambda, K$ and $n$, but not on $m$. Thus, this will show that $\mathcal X$ is uniformly, continuously approximal by $n$-cubes.

Recall that $d_\E$ denotes the standard Euclidean metric.
For each $\vec x \in [0,m]^n$ there exist $\vec z_1, \dots, \vec z_n \in \Z^n$ so that
\begin{enumerate}
 \item $d_\E(\vec x, \vec z_i) \leq \sqrt n$ for $i = 1, \dots, n$, and
 \item we can write
   \[
    \vec x = \sum t_i \vec z_i
   \]
   with $\sum t_i = 1$, $t_i \geq 0$ for each $i$.
\end{enumerate}
This is done by cutting $[0,m]^n$ into $n$-cubes with integer vertices, and then further cutting each $n$-cube into simplices. 
If $\vec x$ lies in a simplex with vertices $\vec z_1, \dots, \vec z_n$, then there is a unique way of writing it as $\vec x = \sum t_i \vec z_i$.
Note that this is true even if $\vec x$ lies in the intersection of two simplices.
Fix an auxiliary filling geodesic current $\eta \in \Cf(S)$ once and for all. Every projective geodesic current $\nu \in \PC(S)$ has a unique lift $\overline{\nu} \in \C(S)$ satisfying
\[
i(\overline{\nu},\eta)=1.
\]
The map $\PC(S) \to C(S)$, given by $\sigma_\eta \colon[\nu] \mapsto \overline{\nu}$ gives a \textbf{continuous section}, i.e., a continuous map $\sigma_\eta$ so that $p \circ \sigma_\eta = \operatorname{id}_{\PC(S)}$, and $p \colon \C(S) \to \PC(S)$ is the natural projection.

We now define our continuous approximation $g$ of $f$.  Consider the map $g \colon [0,m]^n \to \mathcal{X}$ given by
\begin{enumerate}
 \item $g(\vec z) = f(\vec z) \in \mathcal{X}$ for each $z \in \Z^n$.
 \item For $\vec x \in [0,m]^n$ with $\vec x = \sum t_i \vec z_i$ as above, define
 \[
  g(\vec x) = \left[\sum t_i \overline{f(\vec z_i)}\right] 
 \]
where $\overline{f(\vec z_i)} \in \Cf(S)$ is the representative of the projective geodesic current $f(\vec z_i) \in \mathcal{X}$ satisfying $i(\overline{f(\vec z_i)},\eta)=1$ for each $i=1,\cdots,n$. 
\end{enumerate}
Since $\mathcal{X}$ is convex, it follows that  the sum $\sum t_i \overline{f(\vec z_i)} \in \mathcal{X}$, so $g$ is well-defined. Moreover, since $\sum t_i = 1$, we have that  $\sum t_i \overline{f(\vec z_i)}$ is the unique continuous lift of the projective class $[\sum_i t_i \overline{f(\vec z_i)}]$ to $\Cf(S)$ satisfying $i\left(\sum t_i f(\vec z_i), \eta\right)=1$.

Next, we must show that $g$ is continuous. Suppose $\vec x_k \to \vec x$. Since at most finitely many closed simplices contain $\vec x$, up to passing to a subsequence, for all $k$ large enough,
 there is some fixed set $\vec z_1, \dots, \vec z_n$ so that $\vec x_k = \sum t_i^k \vec z_i$, $\vec x = \sum t_i \vec z_i \in \Z^n$, and $t_i^k \to t_i$ as $k \to \infty$. Then $g(\vec x_k) = 
\left[\sum t_i^k \overline{f(\vec z_i)} \right]$. As $k$ goes to infinity, this converges to $\left[\sum t_i \overline{f(\vec z_i)}\right] = g(\vec x)$. Therefore, $g$ is continuous.

Lastly, we prove that $g(\vec x)$ is uniformly close to $f(\vec x)$ for all $\vec x$. In fact, we will show that for $C = \frac 32 \lambda \sqrt n + K>0$, for all $\vec x \in \R^n$, we have
\[
  \dsym(f(\vec x), g(\vec x)) \leq C.
\]
We have $\vec x = \sum t_i \vec z_i$ for $\vec z_i \in \Z^n$ and $\sum t_i = 1$. Let $\vec z \in \Z^n$ be the closest vector in $\Z^n$ to $\vec x$. Then, $d_\E(\vec x, \vec z) \leq \sqrt n/2$.

By the triangle inequality:
 \[
  \dsym(f(\vec x), g(\vec x)) \leq \dsym(f(\vec x), f(\vec z)) + \dsym(f(\vec z), g(\vec z)) + \dsym(g(\vec z), g(\vec x)).
 \]
We approximate all three terms separately.

Since $f$ is a $(\lambda, K)$ quasi-isometric embedding, and $d_\E(\vec x, \vec z) \leq \frac 12 \sqrt n$, we have
\begin{equation}
  \dsym(f(\vec x), f(\vec z)) \leq \frac 12 \lambda \sqrt n + K.
   \label{eq:bound_first}
\end{equation}
Next recall that, by definition of $g$, $g(\vec z) = f(\vec z) \in \mathcal{X}$ for all $\vec z \in \Z^n$. So,
\begin{equation}
 \dsym(f(\vec z), g(\vec z)) = 0.
 \label{eq:bound_middle}
\end{equation}
 Lastly, by \Cref{lem:dist_to_sum} and the fact that $f$ is a quasi-isometric embedding, we have
\begin{equation}
\begin{aligned}
 \dsym(g(\vec z), g(\vec x)) & = \dsym(f(\vec z), \sum_i t_i f(\vec z_i)) \\
 & \leq \max_i \dsym(f(\vec z),  f(\vec z_i))\\
 & \leq \max_i \lambda d_\E(\vec z, \vec z_i) + K \\
 & \leq \lambda \sqrt n + K.
\end{aligned}
\label{eq:bound_third}
\end{equation}

The last line follows from the fact that both $\vec z$ and $\vec z_1, \dots, \vec z_n$ are in the integer simplex about $\vec x$, giving us that $d_\E(\vec z, \vec z_i) \leq \sqrt n$.

Putting Equations~\eqref{eq:bound_first}, ~\eqref{eq:bound_middle}, and ~\eqref{eq:bound_third} together, we get
\[
 \dsym(f(\vec x), g(\vec x)) \leq  \frac32 \lambda \sqrt n + K
\]
for every $\vec x \in [0,m]^n$, as we wanted.
\end{proof}

We will now argue that, if $\mathcal X$ is a convex subspace of $\PCf(S)$ that continuously embeds into $\R^m$, then if $n \geq m$, the existence of a family of uniform quasi-isometric embeddings $f_i \colon [0,i]^n \to \mathcal{X}$ for $i \in \N$ would violate the Borsuk-Ulam theorem (see \Cref{thm:upper_bound_rank}). 

We recall the statement of the Borsuk-Ulam theorem here for convenience. Let $S^m$ denote the \textbf{$m$-dimensional unit sphere}, and $S^m_R$ the \textbf{$m$-dimensional sphere of radius $R$}.

\begin{theorem}[Borsuk-Ulam {\cite{Borsuk1933}}]
\label{thm:Borsuk_Ulam}
Every continuous map $f\colon S^m\longrightarrow \R^m$ maps some pair of antipodal points to the same point. That is, there exists $\vec{p}\in S^m$ such that $f(\vec{p})=f(-\vec{p}).$
\end{theorem}

\begin{lem}
\label{lem:Sphere_embedding}
Let $\mathcal{X}\subset\PCf(S)$ be a convex subspace.
Suppose that, for some $n \in \N$, there exists a family of $(\lambda,K)$-quasi-isometric embeddings
\[
f_i\colon [0,i]^n\longrightarrow\mathcal{X}.
\]
Then, for every sufficiently large $R>0$, depending only on $n$, $\lambda$, and $K$, there exists a continuous map
\[
h\colon S^n_R\longrightarrow\mathcal{X},
\]
such that
\[
h(\vec{p})\neq h(-\vec{p})
\]
for every $\vec{p}\in S^n_R$.
\end{lem}

\begin{proof}
Take a family of $(\lambda, K)$ quasi-isometric embeddings $f_i:[0,i]^n \longrightarrow \mathcal X$. Since $\mathcal X$ is convex, \Cref{lem:approximation_distance} implies that it is $\frac 32 \lambda \sqrt n + K$-continuously approximable.
For each $i\in\N$, let
\[
g_i\colon [0,i]^n\longrightarrow\mathcal{X}
\]
be this continuous uniform approximation. That is, for all $i$, $g_i$ is continuous, and for all $\vec x \in [0,i]^n$, $\dsym(f(\vec x), g(\vec x)) \leq \frac 32 \lambda \sqrt n + K$.

Choose $R>0$ sufficiently large that
\[
\frac{2R}{\lambda}-3\lambda\sqrt n-3K>0.
\]
For any $i\geq 2R$ there is an isometric embedding
\[
\iota\colon S^m_R\hookrightarrow [0,i]^m,
\]
where $S^m_R$ is equipped with the restriction of the Euclidean metric. Fix such an $i$ and define
\[
h=g_i\circ\iota\colon S^m_R\longrightarrow\mathcal{X}.
\]
Since $g_i$ is continuous, so is $h$.

We claim that $h(\vec{p})\neq h(-\vec{p})$ for every $\vec{p}\in S^m_R$. Fix $\vec{p}\in S^m_R$ and set
$\vec x=\iota(\vec{p})
\quad\text{and}\quad
\vec y=\iota(-\vec{p}).$
Since $\vec{p}$ and $-\vec{p}$ are antipodal and $\iota$ is an isometry, $d_{\E}(\vec x,\vec y)=2R$.
As $f_i$ is a $(\lambda,K)$-quasi-isometric embedding,
\[
\dsym(f_i(\vec x),f_i(\vec y))
\geq \frac{2R}{\lambda}-K.
\]
Moreover, we chose $g_i$ so that 
\[
\dsym(f_i(\vec x),g_i(\vec x))
\leq \frac{3}{2}\lambda\sqrt n+K
\]
and
\[
\dsym(f_i(\vec y),g_i(\vec y))
\leq \frac{3}{2}\lambda\sqrt n+K.
\]
Therefore, by the triangle inequality,
\[
\begin{multlined}
\dsym(g_i(\vec x),g_i(\vec y))
\geq \dsym(f_i(\vec x),f_i(\vec y))
-\dsym(f_i(\vec x),g_i(\vec x)) \
-\dsym(f_i(\vec y),g_i(\vec y)),
\end{multlined}
\]
i.e., putting the inequalities together,
\[
\begin{aligned}
\dsym(g_i(\vec x),g_i(\vec y))
&\geq \frac{2R}{\lambda}-K
-2\left(\frac{3}{2}\lambda\sqrt n+K\right) \
&= \frac{2R}{\lambda}-3\lambda\sqrt n-3K
\end{aligned}
\]

As we chose $R$ so that the right-hand side, $\frac{2R}{\lambda}-3\lambda\sqrt n-3K> 0$, we get
\[
h(\vec{p})=g_i(\vec x)\neq g_i(\vec y)=h(-\vec{p}).
\]
Since $\vec{p}\in S^m_R$ was arbitrary, the result follows.
\end{proof}

\begin{theorem}
\label{thm:upper_bound_rank}
Let $\mathcal{X} \subset \PCf(S)$ be a convex subspace equipped with $\dsym$.
Suppose furthermore that  there exists a continuous injective map $\phi \colon \mathcal{X} \to \mathbb{R}^m$.
Then, for all $n > m$, there is no family of $(\lambda,K)$-quasi-isometric embeddings $f_i\colon [0,i]^n \to \mathcal{X}$.

In particular, the geometric rank of $\mathcal X$ satisfies
\[
\rank(\mathcal X)\leq m.
\]
\end{theorem}
Note that the convexity assumption is crucial. For example, there exist non-convex subsets of $\R^2$ that are homeomorphic to $\R$, but still have rank 2 (such as, for example, a spiral.)

\begin{proof}
 Suppose, toward a contradiction, that there exists a family of
$(\lambda,K)$-quasi-isometric embeddings $f_i\colon [0,i]^n\longrightarrow\mathcal X$. 
By \Cref{lem:Sphere_embedding}, there is some $R> 0$, and an embedding $h\colon S_R^n\longrightarrow\mathcal X$ of the radius $R$ $n$-sphere that does not identify any pairs of antipodal points. As $n > m$, we can restrict $h$ to a map on an $m$-sphere, and, abusing notation slightly, get a continuous map
\[
h\colon S_R^m\longrightarrow\mathcal X
\]
such that
\[
h(\vec{p})\neq h(-\vec{p})
\]
for every antipodal pair $\vec{p},-\vec{p}\in S_R^m$.
Consider the continuous composition
\[
S_R^m\xrightarrow{h}\mathcal X\xrightarrow{\phi}\mathbb R^m.
\]
Since $\phi$ is injective, we have
\[
\phi(h(\vec{p}))\neq\phi(h(-\vec{p}))
\]
for every antipodal pair $\vec{p},-\vec{p}\in S_R^m$. This contradicts the
Borsuk--Ulam Theorem (refer to \Cref{thm:Borsuk_Ulam} for the statement).

Therefore, if $n>m$, there is no family of
$(\lambda,K)$-quasi-isometric embeddings
\[
f_i\colon [0,i]^n\longrightarrow\mathcal X
\]
with $\lambda$ and $K$ independent of $i$.\qedhere
\end{proof}

The following corollary about the rank of simplices is then immediate. 

\begin{cor}
\label{cor:upper_bound_rank}
Let $\vec \mu = (\mu_0,\cdots,\mu_m)$ be a jointly filling sequence of geodesic currents. Then  $(\PSimplex(\vec \mu),\dsym)$ has geometric rank satisfying $\rank(\PSimplex(\vec \mu)) \leq m$.
\end{cor}
\begin{proof}
The hypothesis of convexity for $\mathcal{X} \coloneqq \PSimplex(\vec \mu)$ is easily verified.
Let $k+1$ denote the affine dimension of the cone $\Simplex(\vec \mu)$, which satisfies $k+1 \leq m+1$. In fact, up to reordering, suppose $\mu_0, \dots, \mu_k$ is a maximal linearly independent subset of $\mu_0, \dots, \mu_n$. Then every $\nu \in \Delta(\vec \mu)$ can be written uniquely as $\nu = \sum_{i=0}^k x_i \mu_i$, and the map 
\begin{align*}
    \iota : \Delta(\vec \mu) & \longrightarrow \R^{k+1}\\
     \sum_{i=0}^k x_i \mu_i & \mapsto (x_0, \dots, x_k)
\end{align*}
is a continuous embedding that commutes with the projection operation on both sides. Thus, $\mathcal{X}$ admits a continuous injection to $\mathbb{R}^k$, for $k \leq m$.
Hence, by Theorem~\ref{thm:upper_bound_rank}, the conclusion follows.
\end{proof}

\section{Decomposition into increasing sequences}
\label{sec:decompose_increasing_seq}

We now show that any simplex can be decomposed up to quasi-isometry into a union of simplices on increasing sequences (\Cref{cor:decomp_increasing_seq}).

    Let $\mu_0, \dots, \mu_n$ be a set of geodesic currents. If, for all $i$, and for all $\gamma \in \Curves(S)$, we have
    \begin{equation}
     i(\mu_i, \gamma) \leq i(\mu_{i+1}, \gamma)
    \label{eq:dominated}\end{equation}
    then we say that $\mu_0, \dots, \mu_n$ is a \textbf{dominated} sequence of geodesic currents. Note that dominated sequences satisfy the first condition for both increasing and strictly increasing sequences (Definition~\ref{def:increasing_sequence}).

First, we show that any simplex $\Simplex(\vec \mu)$ can be written as a union of simplices on dominated sequences of geodesic currents. Note that we do not need to assume $\vec \mu$ is jointly filling.
\begin{lem}\label{lem:increasing}
Let $\vec\alpha=(\alpha_0,\ldots,\alpha_n)$ be a 
sequence of geodesic currents. Then there exist finitely many  dominated sequences
$\vec\mu^{\,1},\ldots,\vec\mu^{\,N}$ such that
\[
    \Delta(\vec\alpha)
    =
    \bigcup_{j=1}^{N}\Delta(\vec\mu^{\,j}).
\]
\end{lem}

\begin{proof}
For each permutation $\sigma$ of $\{0,\ldots,n\}$, define
$\mu_k^\sigma=\alpha_{\sigma(0)}+\cdots+\alpha_{\sigma(k)}$ for
$k=0,\ldots,n$. For every geodesic current $\gamma$, we have
$i(\mu_k^\sigma,\gamma)\leq i(\mu_{k+1}^\sigma,\gamma)$, so
$\mu_0^\sigma,\ldots,\mu_n^\sigma$ is a dominated sequence. 

For each subset $I\subseteq\{0,\ldots,n\}$ containing $n$, let
$\vec\mu^{\,\sigma,I}=(\mu_k^\sigma)_{k\in I}$, with the indices
written in increasing order.  We claim that
\[
    \Delta(\vec\alpha)
    =
    \bigcup_{\sigma}
    \ \bigcup_{\substack{I\subseteq\{0,\ldots,n\}\\ n\in I}}
    \Delta(\vec\mu^{\,\sigma,I}).
\]

Let $\nu\in\Delta(\vec\alpha)$. Choose a positive coefficient
representation $\nu=\sum_{j=0}^n x_j\alpha_j$, with $x_j>0$. Choose
a permutation $\sigma$ such that
$x_{\sigma(0)}\geq x_{\sigma(1)}\geq\cdots\geq x_{\sigma(n)}>0$, and
define
\[
    y_k=x_{\sigma(k)}-x_{\sigma(k+1)}
    \quad (0\leq k<n),
    \qquad
    y_n=x_{\sigma(n)}.
\]
Then $y_k\geq0$, $y_n>0$, and a telescoping calculation gives
$\nu=\sum_{k=0}^n y_k\mu_k^\sigma$. If
$I=\{k:y_k>0\}$, then $n\in I$ and
$\nu\in\Delta(\vec\mu^{\,\sigma,I})$.

Conversely, fix $\sigma$ and $I\ni n$, and suppose
$\nu=\sum_{k\in I}y_k\mu_k^\sigma$, where $y_k>0$. Expanding the
partial sums gives
\[
    \nu
    =
    \sum_{j=0}^{n}
    \left(
        \sum_{\substack{k\in I\\k\geq j}}y_k
    \right)
    \alpha_{\sigma(j)}.
\]
Since $n\in I$, every coefficient in parentheses is at least
$y_n>0$. Thus $\nu\in\Delta(\vec\alpha)$.

The union is finite because there are only finitely many permutations
$\sigma$ and subsets $I\subseteq\{0,\ldots,n\}$.
\end{proof}

In this lemma, we give a condition for when we can remove an element of a sequence $\vec \mu$ while keeping the simplex the same up to quasi-isometry. Let $d_H$ denote the \textbf{Hausdorff distance} between closed subsets of $(\PCf(S),\dsym)$.
\begin{lem}\label{lem:Reducing_curves}
Let $\vec \mu \coloneqq \mu_0,\ldots,\mu_n$ be a jointly filling sequence on $S$. Suppose that there is a constant $C\geq1$ such that, for
every geodesic current $\gamma$,
\[
    \frac{1}{C}i(\mu_n,\gamma)
    \leq i(\mu_{n-1},\gamma)
    \leq C i(\mu_n,\gamma).
\]
Then the Hausdorff distance with respect to $\dsym$,
between $\PSimplex(\mu_0,\ldots,\mu_{n-1},\mu_n)$ and \\ $\PSimplex(\mu_0,\ldots,\mu_{n-1})$ is at most
$2\log C$. Consequently, there exists a
$(1,4\log C)$--quasi-isometry
\[
    \mathbb P\Delta(\mu_0,\ldots,\mu_{n-1},\mu_n)
    \longrightarrow
    \mathbb P\Delta(\mu_0,\ldots,\mu_{n-1})
\]
whose image is $4\log C$--dense.
\end{lem}
\begin{proof}
Set
\[
    X=\mathbb P\Delta(\mu_0,\ldots,\mu_{n-1},\mu_n)
    \quad\text{and}\quad
    Y=\mathbb P\Delta(\mu_0,\ldots,\mu_{n-1}).
\]
We first show that every point of $X$ lies within $2\log C$ of $Y$.

Let $[\nu]\in X$. Choose any positive coefficient representation
$\nu=\sum_{j=0}^n x_j\mu_j$, and define
\[
    \widehat\nu
    =
    \sum_{j=0}^{n-2}x_j\mu_j
    +(x_{n-1}+x_n)\mu_{n-1}.
\]
Then $[\widehat\nu]\in Y$. 

The hypothesis implies that,
for every geodesic current $\gamma$,
\begin{equation}
    \frac{1}{C}i(\widehat\nu,\gamma)
    \leq i(\nu,\gamma)
    \leq C i(\widehat\nu,\gamma).
    \label{eq:coarsely_equal}
\end{equation}
Indeed, write $\rho=\sum_{j=0}^{n-2}x_j\mu_j$. Then $\nu = \rho + x_{n-1} \mu_{n-1} + x_n \mu_n$, while $\hat \nu = \rho + (x_{n-1} + x_n) \mu_{n-1}$. The common term $i(\rho,\gamma)$  in (\eqref{eq:coarsely_equal}) satisfies the same
inequalities because $C\geq1$, and the remaining inequalities follow
from the comparison between $\mu_{n-1}$ and $\mu_n$. Consequently,
\[
    \operatorname{Dil}(\nu,\widehat\nu)\leq C
    \quad\text{and}\quad
    \operatorname{Dil}(\widehat\nu,\nu)\leq C.
\]
It follows that
$\dsym([\nu],[\widehat\nu])\leq2\log C$.

Conversely, let $[\eta]\in Y$ and choose a positive representation
$\eta=\sum_{j=0}^{n-1}y_j\mu_j$. Define
\[
    \widetilde\eta
    =
    \sum_{j=0}^{n-2}y_j\mu_j
    +\frac{y_{n-1}}{2}\mu_{n-1}
    +\frac{y_{n-1}}{2}\mu_n.
\]
All the coefficients are positive, so $[\widetilde\eta]\in X$.
Replacing the last two coefficients of $\widetilde\eta$ by their sum
recovers $\eta$. The same comparison as above gives
$\dsym([\eta],[\widetilde\eta])\leq2\log C$.

Thus every point of $X$ is within $2\log C$ of $Y$, and every point
of $Y$ is within $2\log C$ of $X$. Therefore
$d_H(X,Y)\leq2\log C$.

For each $x\in X$, choose $F(x)\in Y$ such that
$d_{\mathrm{Lip}}(x,F(x))\leq2\log C$. Then, for all $x,x'\in X$,
\[
    \left|
      d_{\mathrm{Lip}}(F(x),F(x'))
      -
      d_{\mathrm{Lip}}(x,x')
    \right|
    \leq4\log C.
\]
Hence $F$ is a $(1,4\log C)$--quasi-isometric embedding. Moreover,
if $y\in Y$, choose $x\in X$ with
$d_{\mathrm{Lip}}(x,y)\leq2\log C$. Then
$d_{\mathrm{Lip}}(F(x),y)\leq4\log C$, so the image of $F$ is
$4\log C$--dense in $Y$. Thus $F$ is a
$(1,4\log C)$--quasi-isometry.
\end{proof}

Now, we show that simplices on jointly filling and dominated sequences are quasi-isometric to simplices on (perhaps shorter) increasing sequences.

\begin{lem}
Let $\vec\mu=(\mu_0,\ldots,\mu_n)$ be a jointly filling dominated
sequence of geodesic currents. Then there is an increasing subsequence
$\vec\mu'=(\mu_{i_0},\ldots,\mu_{i_m})$ such that
$\PSimplex(\vec \mu)$ and $\PSimplex(\vec\mu')$ are at
finite Hausdorff distance. In particular, they are quasi-isometric.
\label{lem:strictly_increasing}
\end{lem}

\begin{proof}
If $\vec\mu$ is increasing, there is nothing to prove. Otherwise,
there is an index $i\in\{1,\ldots,n\}$ for which
$\operatorname{Dil}(\mu_{i-1},\mu_i)<\infty$. Thus there is a
constant $C_i\geq1$ such that
$i(\mu_i,\gamma)\leq C_i i(\mu_{i-1},\gamma)$ for every geodesic current
 $\gamma$. Since the sequence is dominated, we also have
$i(\mu_{i-1},\gamma)\leq i(\mu_i,\gamma)$. Therefore
\[
    \frac{1}{C_i}i(\mu_i,\gamma)
    \leq i(\mu_{i-1},\gamma)
    \leq C_i i(\mu_i,\gamma).
\]
Reorder the generators so that $\mu_{i-1}$ and $\mu_i$ are the final
two generators. Reordering does not change the simplex, so \Cref{lem:Reducing_curves}
shows that deleting $\mu_i$ changes the projectivized simplex by
finite Hausdorff distance.

The remaining sequence, written in its original order, is still
dominated. Repeating the argument finitely many times produces a
subsequence $\vec\mu'=(\mu_{i_0},\ldots,\mu_{i_m})$ for which
\[
    \operatorname{Dil}(\mu_{i_{j-1}},\mu_{i_j})=\infty
\]
for every $j=1,\ldots,m$. Since the subsequence remains dominated,
$\operatorname{Dil}(\mu_{i_j},\mu_{i_{j-1}})<\infty$. Hence
$\mu_{i_{j-1}}\prec\mu_{i_j}$, and $\vec\mu'$ is increasing.

At each step the two projectivized simplices are at finite Hausdorff
distance. Since only finitely many generators are removed, the
triangle inequality for Hausdorff distance shows that
$\mathbb P\Delta(\vec\mu)$ and
$\mathbb P\Delta(\vec\mu')$ are at finite Hausdorff distance.
\end{proof}
Combining Lemmas \ref{lem:increasing} and \ref{lem:strictly_increasing}, we immediately get that any filling simplex is coarsely isometric to a finite union of filling simplices for increasing geodesic currents.

\begin{cor}
\label{cor:decomp_increasing_seq}
Let $\vec\alpha=(\alpha_0,\ldots,\alpha_n)$ be a jointly filling
sequence of geodesic currents. Then there exist finitely many jointly
filling, increasing sequences
$\vec\mu^{\,1},\ldots,\vec\mu^{\,N}$ and a constant $D\geq0$ such
that 
\[
    d_H\left(
        \mathbb P\Delta(\vec\alpha),
        \bigcup_{j=1}^N\mathbb P\Delta(\vec\mu^{\,j})
    \right)
    \leq D.
\]
Moreover, there is a $(1,2D)$--quasi-isometry
\[
    F:\mathbb P\Delta(\vec\alpha)
    \longrightarrow
    \bigcup_{j=1}^N\mathbb P\Delta(\vec\mu^{\,j})
\]
whose image is $2D$--dense.
\end{cor}

\begin{proof}
By \Cref{lem:increasing}, there are finitely many jointly filling, dominated
sequences $\vec\nu^{\,1},\ldots,\vec\nu^{\,N}$ such that
\[
    \mathbb P\Delta(\vec\alpha)
    =
    \bigcup_{j=1}^N\mathbb P\Delta(\vec\nu^{\,j}).
\]
For each $j$, \Cref{lem:strictly_increasing} gives a jointly filling increasing
subsequence $\vec\mu^{\,j}$ of $\vec\nu^{\,j}$ such that
$\mathbb P\Delta(\vec\nu^{\,j})$ and
$\mathbb P\Delta(\vec\mu^{\,j})$ are at finite Hausdorff distance
with respect to $\dsym$. Choose $D_j\geq0$ such that every point of
either of these two spaces is within $D_j$ of the other, and set
$D=\max_{1\leq j\leq N}D_j$.

Let $x\in\mathbb P\Delta(\vec\alpha)$. By \Cref{lem:increasing}, there is some
$j$ such that $x\in\mathbb P\Delta(\vec\nu^{\,j})$. Hence $x$ is
within $D_j\leq D$ of
$\mathbb P\Delta(\vec\mu^{\,j})$, and therefore within $D$ of
$\bigcup_{k=1}^N\mathbb P\Delta(\vec\mu^{\,k})$.

Conversely, suppose that
$x\in\mathbb P\Delta(\vec\mu^{\,j})$ for some $j$. Then $x$ is
within $D_j\leq D$ of $\mathbb P\Delta(\vec\nu^{\,j})$. Since
$\mathbb P\Delta(\vec\nu^{\,j})$ is contained in
$\mathbb P\Delta(\vec\alpha)$, the point $x$ is within $D$ of
$\mathbb P\Delta(\vec\alpha)$. Thus
\[
    d_H\left(
        \mathbb P\Delta(\vec\alpha),
        \bigcup_{j=1}^N\mathbb P\Delta(\vec\mu^{\,j})
    \right)
    \leq D.
\]

For each $x\in\mathbb P\Delta(\vec\alpha)$, choose
$F(x)\in
    \bigcup_{j=1}^N\mathbb P\Delta(\vec\mu^{\,j})$ 
such that $\dsym(x,F(x))\leq D$. For any
$x,y\in\mathbb P\Delta(\vec\alpha)$, the triangle inequality gives
$\bigl|
        \dsym(F(x),F(y))-\dsym(x,y)
    \bigr|
    \leq 2D$.
Therefore $F$ is a $(1,2D)$--quasi-isometric embedding.

Finally, let
$z\in\bigcup_{j=1}^N\mathbb P\Delta(\vec\mu^{\,j})$. By the
Hausdorff-distance bound, there is
$x\in\mathbb P\Delta(\vec\alpha)$ such that $\dsym(z,x)\leq D$.
Consequently,
$\dsym(z,F(x))\leq\dsym(z,x)+\dsym(x,F(x))\leq2D$.
Thus the image of $F$ is $2D$--dense, and $F$ is a
$(1,2D)$--quasi-isometry.
\end{proof}

\section{Maximal strictly increasing sequence}
\label{sec:maximal_sequence}
\begin{figure}[h!]
     \centering
\includegraphics[width=130mm]{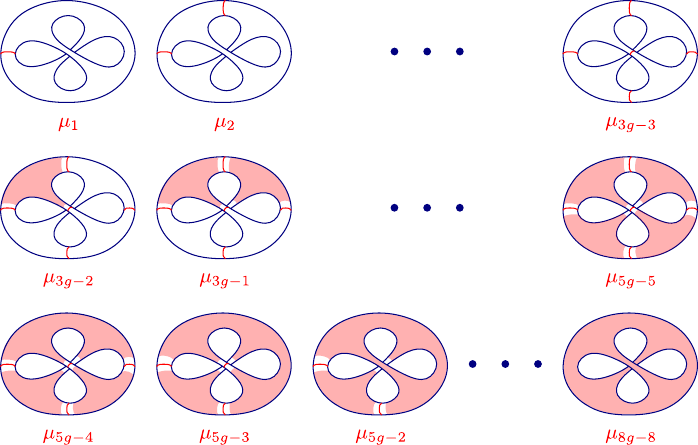}
    \caption{Maximal family of strictly increasing geodesic currents.}
    \label{fig:max_rank}
\end{figure} 

In the following Lemma and its proof, compare Figure~\ref{fig:max_rank}.

\begin{lem}
\label{lem:build_increasing}
Let $S$ be a connected surface of genus $g$ with $b$ boundary
components. Then there exists a jointly filling and strictly increasing sequence
$\vec\mu=(\mu_1,\ldots,\mu_{8g-8+3b})$ such that the final geodesic current
$\mu_{8g-8+3b}$ fills $S$. Moreover, the geodesic currents
$\mu_1,\ldots,\mu_{3g-3+b}$ are simple closed multicurves with pairwise intersection zero.
\end{lem}

\begin{proof}
Set $p=3g-3+b$ and $q=2g-2+b$. Thus $p$ is the number of curves in
a pants decomposition of $S$, and $q$ is the number of complementary
pairs of pants.

Choose a pants decomposition
$\alpha_1,\ldots,\alpha_p$.

For $1\leq i\leq p$, define
$\mu_i=\alpha_1+\cdots+\alpha_i$. These geodesic currents form a dominated
sequence because $\mu_i=\mu_{i-1}+\alpha_i$.

For each $i$, choose a closed curve $\gamma_i$ which crosses
$\alpha_i$ and is disjoint from every $\alpha_j$ with $j\neq i$.
Such a curve can be chosen in the union of the one or two pairs of
pants adjacent to $\alpha_i$. We then have
$i(\mu_{i-1},\gamma_i)=0$ and
$i(\mu_i,\gamma_i)=i(\alpha_i,\gamma_i)>0$. Hence
\[
    \mu_1\sinc\mu_2\sinc\cdots
    \sinc\mu_p.
\]

Let $P_1,\ldots,P_q$ be the open pairs of pants complementary to the
pants decomposition. For each $r$, choose a self-intersecting closed
curve $\eta_r$ contained in $P_r$ and filling $P_r$; for example,
choose a figure-eight curve. In particular,
$i(\eta_r,\eta_r)>0$. Define
$\mu_{p+r}=\mu_p+\eta_1+\cdots+\eta_r$ for $1\leq r\leq q$.

Again, each geodesic current is obtained from the preceding one by adding a
geodesic current. Hence this portion of the sequence is dominated.
Since the pairs of pants are disjoint and each $\eta_r$ is contained
in the interior of $P_r$, we have
$i(\mu_{p+r-1},\eta_r)=0$, whereas
$i(\mu_{p+r},\eta_r)\geq i(\eta_r,\eta_r)>0$. Thus
\[
    \mu_p\sinc\mu_{p+1}\sinc\cdots
    \sinc\mu_{p+q}.
\]

We now define the final block. For $1\leq i\leq p$, let $F_i$ be the
subsurface obtained by taking annular neighborhoods of
$\alpha_1,\ldots,\alpha_i$ together with every pair of pants incident
to at least one of these curves. Note that $F_p=S$.

An unchosen pants curve $\alpha_k$, with $k>i$, is either disjoint
from $F_i$ or parallel to a boundary component of $F_i$. This remains
true when both pairs of pants adjacent to $\alpha_k$ have already
entered $F_i$: because the annulus about $\alpha_k$ has not yet been
added, its two sides remain boundary components of $F_i$. On the
other hand, $\alpha_i$ is essential and nonperipheral in $F_i$.

For each $i$, choose a geodesic current $\nu_i$ supported in $F_i$
and filling $F_i$. Define
\[
    \mu_{p+q+i}
    =
    \mu_{p+q}+\nu_1+\cdots+\nu_i,
    \qquad 1\leq i\leq p.
\]
This block is dominated because each geodesic current is obtained from the
preceding one by adding $\nu_i$.

We claim that $\alpha_i$ witnesses strictness at the $i$th step of
this block. The geodesic currents $\alpha_1,\ldots,\alpha_p$ have zero
intersection with $\alpha_i$, and every $\eta_r$ is supported in the
interior of a complementary pair of pants, so
$i(\mu_{p+q},\alpha_i)=0$. If $j<i$, then $\alpha_i$ is disjoint from
$F_j$ or peripheral in $F_j$, and hence
$i(\nu_j,\alpha_i)=0$. Therefore
$i(\mu_{p+q+i-1},\alpha_i)=0$.

Since $\nu_i$ fills $F_i$ and $\alpha_i$ is essential and
nonperipheral in $F_i$, we have $i(\nu_i,\alpha_i)>0$. Consequently,
$i(\mu_{p+q+i},\alpha_i)>0$, and hence
\[
    \mu_{p+q}\sinc\mu_{p+q+1}
    \sinc\cdots
    \sinc\mu_{2p+q}.
\]

Combining the three blocks gives a strictly increasing sequence with
$2p+q=8g-8+3b$ geodesic currents. Finally, $F_p=S$ and $\nu_p$ fills $S$, so
the final geodesic current $\mu_{2p+q}$ fills $S$. Since
$2p+q=8g-8+3b$, this is the required sequence.
\end{proof}

\section{Structure of increasing sequences}
\label{sec:increasing_seq}
In this section, we give some technical lemmas about increasing and strictly increasing sequences of currents. They will allow us to prove that every increasing sequence $\vec \mu = (\mu_0, \dots, \mu_n)$ has $n \leq 8g-9$ in \Cref{sec:bound_increasing_seq}. Of potentially independent interest is a characterization of when an increasing sequence $\vec \mu$ is in fact strictly increasing (up to scaling) in \Cref{lem:increasing_implies_strictly}.

\subsection{Relating inequalities on currents to containment of their support}
Given a measured lamination $\lambda$, let \textbf{$|\lambda|$} denote its
support, and let \textbf{$\Sigma(\lambda)$} denote the smallest essential
subsurface of $S$ containing $|\lambda|$, which can be disconnected. If $\lambda$ is supported
on a simple closed curve, we identify $\Sigma(\lambda)$ with an annular neighborhood of that
curve.

Let $F\subseteq S$ be a subsurface, possibly disconnected, and let
$\lambda$ be a geodesic lamination contained in the complement of
the interior of $F$. We assume that no component of $\lambda$ is a
boundary component of $F$.

A geodesic current $\nu$ supported in $F$ is said to \textbf{fill $F$} if its
support intersects transversely every geodesic projecting to the
interior of $F$. For such a geodesic current, define the \textbf{systole of $\nu$ on $F$} to be
\[
    \sys_F(\nu)
    \coloneqq
    \inf_{\gamma\subset \Curves(F)}
    i(\nu,\gamma),
\]
where we recall that the infimum is taken over all essential and non-peripheral closed curves
contained in of $F$.

Let $\mathcal C_{F,\lambda}$ be the set of geodesic currents of the form
$\mu=\mu_F+\mu_\lambda$, where $\mu_F$ fills $F$ and
$\mu_\lambda$ is a measured lamination with
$|\mu_\lambda|=\lambda$. We allow $F=\varnothing$ or
$\lambda=\varnothing$, in which case the corresponding component is
omitted.
Recall that by Eq.~\eqref{eq:current_decomposition}, every nonzero geodesic current $\mu$ admits a unique decomposition as 
\[
\mu = \mu_F + \mu_\lambda,
\]
where $F$ is a subsurface of $S$, possibly disconnected, $\mu_F$ fills $F$ (if $\mu_F \neq 0$), and $\mu_\lambda$ is a measured lamination supported in the complement of the closure of $F$. 
In particular, no component of $|\mu_\lambda|$ is a boundary component of $F$.

If $\mu \in \C(S)$, and  $\mu=\mu_F+\mu_\lambda$ is its canonical decomposition, we write
$|\mu|=F\cup|\mu_\lambda|$. This notation records the filled
subsurface and the lamination in the canonical decomposition. Then more generally, if $\mu \in \C_{F,\lambda}$, we write
\[
 |\mu| = F \cup \lambda.
\]
Moreover, we write $|\mu| \subseteq F \cup \lambda$ to mean containment as sets. That is, we mean that $\mu \in \C_{F', \lambda'}$ where $F' \subseteq F$ and $\lambda' \subseteq F \cup \lambda$. We say $F$ is the \textbf{subsurface in the support} of $\mu$, and $\lambda$ is the \textbf{lamination in the support} of $\mu$. If $\mu = \mu_F + \mu_\lambda$, where $\mu_F$ fills $F$, and $\mu_\lambda$ is the measured lamination supported on $\lambda$, then we say $\mu_F$ is the \textbf{component that fills $F$}, and $\mu_\lambda$ is the \textbf{measured lamination component} of $\mu$.

 A subsurface $F' \subset F$ is a \textbf{boundary-parallel annulus} if it is an annulus homotopic to a boundary component of $F$.

We will now relate our ordering on geodesic currents to containment of their supports.

\begin{lem}
\label{lem:inc_implies_support_containment}
    Let $\mu, \nu \in \C(S)$. If $\mu \inc \nu$, then $|\mu| \subseteq |\nu|$.
\end{lem}
\begin{proof}
Since $\mu \inc \nu$, we first deduce that, for all $\eta \in \C(S)$
\begin{equation}
\label{eq:zero_intersection_condition}
i(\nu, \eta) = 0 \implies i(\mu, \eta) = 0.
\end{equation}
To see this, suppose there is some $\eta$ so that $i(\nu, \eta) = 0$ but $i(\mu, \eta) \neq 0$. Take any sequence of filling geodesic currents $\eta_n \to \eta$  (so that $i(\nu, \eta_n) \neq 0$). Then $\lim \frac{i(\mu, \eta_n)}{i(\nu, \eta_n)} = \infty$. This implies $\Dil(\nu, \mu) = \infty$. But by definition, $\mu \inc \nu$ means that this dilatation should be finite, giving us a contradiction.

    Write the canonical decompositions as
\[
\mu=\mu_F+\mu_\lambda,
\qquad
\nu=\nu_{F'}+\nu_{\lambda'}.
\]
Let $\lambda_0$ be a minimal component of $\lambda$. Since
$i(\nu,\nu_{\lambda'})=0$, Eq.~\eqref{eq:zero_intersection_condition} gives
$i(\mu,\nu_{\lambda'})=0$. Hence $\lambda_0$ either is contained in
$\lambda'$ or is disjoint from $\lambda'$.

Suppose that $\lambda_0$ is not contained in $\lambda'$. Every essential
closed geodesic in $\Sigma(\lambda_0)$ intersects $\mu_{\lambda_0}$.
If $\Sigma(\lambda_0)$ were not contained in $F'$, one could choose an
essential closed geodesic $\alpha\subset\Sigma(\lambda_0)$ disjoint from
$F'\cup\lambda'$. This would give $i(\nu,\alpha)=0$ but
$i(\mu,\alpha)>0$, a contradiction. Thus
$\lambda_0\subset F'\cup\lambda'$. This gives us 
\[
\lambda \subset F' \cup \lambda'.
\]

Now we show $F \subset F'$. First, we note that $\lambda'$ is disjoint
from the interior of $F$.  Indeed, otherwise, $i(\mu, \nu_{\lambda'}) \neq 0$, while $i(\nu, \nu_{\lambda'}) = 0$. Thus, $F$ lies in the complement of $\Sigma(\lambda')$. So, if $F$ were not contained in $F'$, one could
choose an essential, non-peripheral
closed geodesic $\alpha\subset F$ disjoint from
$F'\cup\lambda'$. Since $\mu_F$ fills $F$, this would give $i(\mu,\alpha)>0$ while
$i(\nu,\alpha)=0$, again a contradiction.
Therefore $F\subseteq F'$.

It follows that $|\mu|\subseteq|\nu|$.
\end{proof}

We get a stronger statement when $\mu \sinc \nu$. 
\begin{lem}
\label{lem:sinc_implies_strict_containment}
    Let $\mu, \nu \in \C(S)$. Then $|\mu| \subsetneq |\nu|$ and $\Dil(\nu, \mu) < \infty$ if and only if $\mu \sinc d \nu$ for some $d > 0$.
\end{lem}
In particular, if $\mu \sinc \nu$, then the support of $\mu$ is \textit{strictly} contained in the support of $\nu$.
\begin{proof}
    First, suppose $|\mu| \subsetneq |\nu|$  and $\Dil(\nu, \mu) < \infty$. Then, $|\mu| \subsetneq |\nu|$ means that there is some measured lamination $\lambda$ (which is possibly a simple closed curve, but not necessarily) so that $i(\mu, \lambda) = 0$ but $i(\nu, \lambda) \neq 0$. Then by Property 2 in \Cref{lem:increasing_properties}, there is some constant $d> 0$ so that $\mu \sinc d \nu$.

    Now suppose $\mu \sinc d \nu$ for some $d > 0$. Then by Property 1 in \Cref{lem:increasing_properties}, $\mu \inc d\nu$.  So by \Cref{lem:inc_implies_support_containment}, $|\mu| \subseteq |d \nu|$. As support is scale-invariant, this gives $|\mu| \subseteq |\nu|$. We must just show that the supports cannot be equal. But if $|\mu| = |\nu|$, then for all geodesic currents $\eta$, 
    \[
    i(\nu, \eta) = 0 \iff i(\mu, \eta) = 0
    \]
    In particular, there cannot be any $\eta \in \C(S)$ so that $i(d\nu, \eta) \neq 0$ but $i(\mu, \eta) =0$, contradicting the fact that $\mu \sinc d \nu$. So we conclude 
    \[
    |\mu| \subsetneq |\nu|.\qedhere
    \]
\end{proof}

It turns out that $|\mu| \subsetneq |\nu|$ does not imply that $\Dil(\nu, \mu) < \infty$, due to obstructions discussed in \Cref{sec:NUE}. However, if $|\nu|$ has no lamination components, then this does hold (\Cref{lem:bounded_sup}). Proving this will take a bit of work.

First, we will show that if $|\mu| \subseteq F$ for some subsurface $F$, then we can estimate $i(\mu,\gamma)$ in terms of the intersection of $\gamma$ with any filling collection of curves in $F$ (\Cref{lem:supported-current-upper-bound}). To do this, we need the following setup.

Given a geodesic current $\nu$ on $S$, it induces a pseudometric~\cite{BIPP21:Currents,DRMG23:Duals} $d_\nu$ on $\wt{S}$, the universal cover of $S$ (isometrically identified with $\mathbb{H}^2$) via
\[
d_\nu(x,y)\coloneqq \frac12 \nu(G[x,y)) + \frac12 \nu(G(x,y]) \qquad \text{ for every } x,y \in \wt{S}
\]
where $G[x,y)$ denotes the set of all complete geodesics in $G(S)$ intersecting the hyperbolic geodesic segment $[x,y)$ transversely.
By~\cite[Section~4]{BIPP21:Currents}, this induces, for every continuous path $c : [a,b] \to \wt{S}$, with $-\infty <a \leq b < \infty$ a notion of \textbf{$\nu$-length} $L_\nu(c)$.
If $\gamma$ is a continuous arc on $S$, we let $L_\nu(\gamma)$ denote the length of $L_\nu(\wt{\gamma})$, where $\wt{\gamma}$ denotes any arc lift to the universal cover $\wt{S}$.
The following is proven in~\cite{BIPP21:Currents}.

\begin{lem}[{\cite[Lemmas~4.3, 4.4]{BIPP21:Currents}}]
\label{lem:nu-length}
Let $\nu$ be a geodesic current.
The following properties are satisfied by $L_\nu$:
\begin{itemize}
    \item \label{item:length_larger_distance}
    For any continuous path $c \colon [a,b] \to \wt{S}$, $L_\nu(c) \geq d_\nu(c(a),c(b))$. Furthermore, if $c$ parametrizes a hyperbolic geodesic path in $\wt{S}$, equality holds. 
    \item \label{item:composition} If $c_1 \circ c_2$ denotes a concatenation of two continuous paths $c_1,c_2$, then $L_\nu(c_1 \circ c_2)=L_\nu(c_1)+L_\nu(c_2)$. 
    \item \label{item:upperbound_intersection} If $\gamma'$ is a continuous path in $S$ in the free homotopy class of a hyperbolic geodesic $\gamma$ on $S$, we have $L_\nu(\gamma') \geq i(\nu,\gamma)$. Furthermore, if $\gamma'$ is the hyperbolic geodesic, then equality holds. 
\end{itemize} 
\end{lem}

If $\Gamma$ is a finite union of arcs and curves on $S$, we define 
\begin{equation}
\label{eq:nu-length-union}
L_\nu(\Gamma) \coloneqq \sum_{\gamma \in \Gamma} L_\nu(\gamma).
\end{equation}

The following result extends a technique in the proof of ~\cite[Lemma~4.7]{Minsky1993} to upper bound the intersection number of a geodesic current with any non-simple curve in terms of the intersection with a filling collection of simple closed curves. 

\begin{lem}\label{lem:supported-current-upper-bound}
Let $F\subseteq S$ be a subsurface, possibly disconnected, and let
$\Gamma$ be a finite collection of simple closed geodesics cutting each
component of $F$ into disks and boundary-parallel annuli. If
$\nu\in\mathcal C(S)$ satisfies $|\nu|\subseteq F$, then there is a
constant $A=A(\nu,\Gamma)>0$ such that
\[
    i(\nu,\gamma)\leq A i(\Gamma,\gamma)
\]
for every closed curve $\gamma$ on $S$.
\end{lem}
\begin{proof}
In each boundary-parallel annular component $a \in \mathcal A$ of
$F\setminus\Gamma$, choose a properly embedded geodesic arc $\alpha_a$ joining its two
boundary components. Let $\Omega = \{ \eta_a \colon a \in \mathcal A\}$ be the collection of these arcs,
and set 
\[
M=\Gamma\cup\Omega\cup\partial F.
\]

Let $\gamma$ be a closed geodesic. Only components of
$\gamma\cap F$ can contribute to $L_\nu(\gamma)$, since
$|\nu|\subseteq F$.
Cut $\gamma \cap F$ at its intersections with $\Gamma$, obtaining arc components $\gamma_i$ so that $\gamma=\gamma_1 \circ \dots \circ \gamma_k$. 

Every component of $\gamma\cap F$ lies in a disk or a
boundary-parallel annulus of $F\setminus\Gamma$. A component $\gamma_i^d$ lying
in a disk can be replaced, relative to its endpoints, by a path ${\gamma_i^d}'$ in
the boundary of that disk which traverses each point in $M$ at most twice.
A component $\gamma_j^a$ lying in a boundary-parallel annulus can similarly be
replaced by a concatenation of paths $\eta_j \circ {\gamma_j^a}' \circ \eta_k$ in $M$, where $\eta_j, \eta_k$ are arcs in $\Gamma \cup \Omega$, and the path ${\gamma_j^a}'$ is parallel to a boundary component of $F$, but pushed outside of $F$. (We allow any of these arcs to be empty). In particular, $L_\nu({\gamma_j^a}')=0$. 
Regluing the paths components $\eta_j \circ {\gamma_j^a
}' \circ \eta_k$ and ${\gamma_j^d
}'$, we obtain a closed continuous path $\gamma'$ in the free homotopy class of $\gamma$, and hence, in the homotopy class of the hyperbolic geodesic representative of $\gamma$.
By Equation~\eqref{eq:nu-length-union}, $L_\nu(\gamma')$ is the sum of the $\nu$-lengths of each of these arcs. Trivially, each arc component has length at
most $2L_\nu(M)$. Since the number of components of
$\gamma\cap F$ is at most $2i(\Gamma,\gamma)$, by Lemma~\ref{lem:nu-length}, it follows that
\[
    i(\nu,\gamma)
    \leq
    4L_\nu(M)i(\Gamma,\gamma).
\]
Thus the result holds with $A=4L_\nu(M)$.
\end{proof}

We now show that if $\nu$ has no lamination component in its support, then $|\mu| \subseteq |\nu|$ does imply that the dilatation $\Dil(\nu, \mu)$ is bounded.

\begin{lem}
\label{lem:bounded_sup}
 Let $\mu, \nu \in \C(S)$, and let $F \subseteq S$ be a subsurface. Suppose $|\nu| = F$, and $|\mu| \subseteq F$. Then
 \[
  \Dil(\nu,\mu) < \infty.
 \]

\end{lem}
\begin{proof}
Choose a finite collection of simple closed geodesics $\Gamma$ that cut $F$ into disks and boundary-parallel annuli. Let $\gamma$ be any closed geodesic in $S$. By
\cite[Lemma~3.3]{Sap23:LengthComparison}, applied in each component of $F$, we have
\[
    \sys_F(\nu)^2 i(\Gamma,\gamma)
    \leq
    12i(\nu,\Gamma)i(\nu,\gamma)
\]
Hence
$i(\nu,\gamma)\geq c_\nu i(\Gamma,\gamma)$, where
$c_\nu=\sys_F(\nu)^2/(12i(\nu,\Gamma))>0$.

By \Cref{lem:supported-current-upper-bound}, there is
$A_{\mu}>0$ such that
$i(\mu,\gamma)\leq A_{\mu}i(\Gamma,\gamma)$. Therefore
\[
    i(\mu,\gamma)
    \leq
    \frac{A_{\mu}}{c_\nu}i(\nu,\gamma)
\]
for every closed curve $\gamma$. It follows that
$\operatorname{Dil}(\nu,\mu)\leq A_{\mu}/c_\nu<\infty$.
\end{proof}

\subsection{Non-uniquely ergodic components of geodesic currents}
\label{sec:NUE}
The reason that lamination components in the support of a geodesic current $\nu$ can cause unbounded dilatation is exactly because of so-called non-uniquely ergodic laminations. We give an overview of these here. We then show in \Cref{lem:increasing_implies_strictly} that non-uniquely ergodic laminations are exactly the obstruction that keep an increasing sequence of currents from being strictly increasing.

We say that a measured lamination $\lambda$ is \textbf{ergodic} if its induced geodesic flow invariant measure $\widetilde \lambda$ is ergodic on the unit tangent bundle $T^1 X$, or, equivalently, if $\lambda=a \lambda_1 + b \lambda_2$ for some non-zero measured laminations $\lambda_1,\lambda_2$ and $a, b \geq 0$, then either $a=0$ or $b=0$.
Recall that every nonzero measured lamination $\mu$ admits an ergodic decomposition
\[
\mu=\mu_1+\cdots+\mu_k,
\]
where the $\mu_j$ are nonzero, ergodic, pairwise nonproportional measured laminations satisfying $i(\mu_j,\mu_\ell)=0$ for all $j,\ell$. Here the weights are included in the components $\mu_j$. This decomposition is unique up to reordering.

If $\nu$ is a measured lamination, we write $\nu\ll\mu$ if $\nu$ is a nonnegative linear combination of the ergodic components of $\mu$. Thus, if $\mu=\sum_{j=1}^k\mu_j$, then $\nu\ll\mu$ means that
\[
\nu=\sum_{j=1}^k c_j\mu_j
\]
for uniquely determined coefficients $c_j\geq 0$. Otherwise, we say $\nu\not\ll\mu$.

We have the following two results, essentially contained in~\cite{Walsh2014Horoboundary}.

\begin{lem}[{\cite[Lemma~6.3]{Walsh2014Horoboundary}}]\label{lem:walsh_lem63}
Let $\mu_0,\ldots,\mu_n$ be pairwise nonproportional ergodic measured laminations such that $i(\mu_j,\mu_k)=0$ for all $j,k$. For every $C>0$, there exists a closed geodesic $\alpha\subset\Sigma(\mu_0)$ such that
\[
i(\mu_0,\alpha)>C\,i(\mu_j,\alpha)
\]
for every $j=1,\ldots,n$.
\end{lem}

\begin{proof}
Notice that the conclusion of \cite[Lemma~6.3]{Walsh2014Horoboundary} holds simultaneously for all $j=1,\ldots,n$.
\end{proof}

\begin{lem}[{\cite[Lemma~6.4]{Walsh2014Horoboundary}}]
Let $\mu$ and $\nu$ be measured laminations, with $\mu\neq 0$, and write $\mu=\sum_{j=1}^k\mu_j$ as its ergodic decomposition.

If $\nu\ll\mu$ and $\nu=\sum_{j=1}^k c_j\mu_j$, then
\begin{equation}
\label{eq:sups_walsh}
\sup_{\substack{\eta\in\Curves(S)\\ i(\mu,\eta)>0}}
\frac{i(\nu,\eta)}{i(\mu,\eta)}=\sup_{\substack{\eta\in\mathcal{ML}(S)\\ i(\mu,\eta)>0}}
\frac{i(\nu,\eta)}{i(\mu,\eta)}
=\max_{1\leq j\leq k}c_j.
\end{equation}
If $\nu\not\ll\mu$, then the same supremum is $+\infty$.

Furthermore, the same supremum is obtained if it is taken over closed
geodesics. More precisely, there exists a sequence of closed geodesics
$\alpha_k$ such that
\[
\frac{i(\nu,\alpha_k)}{i(\mu,\alpha_k)}
\longrightarrow
\sup_{\substack{\eta\in\mathcal{ML}(S)\\i(\mu,\eta)>0}}
\frac{i(\nu,\eta)}{i(\mu,\eta)}.
\]
When $\nu\ll\mu$, the curves $\alpha_k$ may be chosen in the supporting
subsurface of an ergodic component $\mu_j$ for which $c_j=\max_\ell c_\ell$.
\label{lem:walsh_lem64}
\end{lem}

\begin{proof}
The proof is the same as~\cite[Lemma 6.4]{Walsh2014Horoboundary}, the only addition being the first equality in Equation~\eqref{eq:sups_walsh}, which is obtained from the same argument as in Walsh's result. Indeed, if we assume $\nu=\sum_{j=1}^k c_j\mu_j$, for every $\eta \in \Curves_f(S)$  
\[
\sup_{\substack{\eta\in\Curves_f(S)}}
\frac{i(\nu,\eta)}{i(\mu,\eta)} = \sup_{\substack{\eta\in\Curves_f(S)}} \frac{\sum_{j=1}^k c_j i(\mu_j,\eta)}{\sum_{j=1}^k i(\mu_j,\eta)} \leq \max_{j=1}^k c_j
\]
where in the last equality we used Lemma~\ref{lem:ratiosum}, as $i(\mu_j,\eta)>0$, since $\eta$ is filling.
Finally, by Lemma~\ref{lem:dilations}, the same inequality follows for $\eta$ over $\Curves(S)$ with $i(\mu, \eta)>0$.
\end{proof}

Recall that any measured lamination $\lambda$ can be decomposed into a disjoint sum of measured laminations $\lambda_1 + \cdots + \lambda_k$ where the $|\lambda_i|$ are pairwise disjoint and of two types: either a simple closed curve, or a geodesic lamination so that every leaf is dense in its support~\cite{CEG06:NotesThurston}. We call any of the $\lambda_i$ a \emph{minimal} component of $\lambda$.

\begin{lem}\label{lem:non-uniquely_ergodic}
Suppose that $\mu,\mu'\in\C(S)$ satisfy $\mu\prec\mu'$ and
\[
|\mu|=|\mu'|=F\cup\lambda.
\]
Write their canonical decompositions as $\mu=\mu_F+\mu_\lambda$ and $\mu'=\mu'_F+\mu'_\lambda$. Then $\lambda$ contains a non-uniquely ergodic minimal component and
\[
\mu'_\lambda\not\ll\mu_\lambda.
\]
\end{lem}

\begin{proof}
Write $\lambda=\lambda_1\cup\cdots\cup\lambda_m$ as a union of its minimal components, and decompose
\[
\mu_\lambda=\sum_{r=1}^m\mu_r,
\qquad
\mu'_\lambda=\sum_{r=1}^m\mu'_r,
\]
where $|\mu_r|=|\mu'_r|=\lambda_r$.

Suppose first that every $\lambda_r$ is uniquely ergodic. Since $\mu_r$ and $\mu'_r$ have the same support, there is a constant $c_r>0$ such that $\mu'_r=c_r\mu_r$. In particular, $\mu'_\lambda\ll\mu_\lambda$. \Cref{lem:walsh_lem64} therefore gives
\[
\operatorname{Dil}(\mu_\lambda,\mu'_\lambda)
=\max_{1\leq r\leq m}c_r<\infty.
\]
Moreover, because $\mu_F$ fills $F$ and $\mu'_F$ is supported in $F$, \Cref{lem:bounded_sup} gives $\operatorname{Dil}(\mu_F,\mu'_F)<\infty$.

For every closed geodesic $\gamma$, set $a=i(\mu_F,\gamma)$, $b=i(\mu_\lambda,\gamma)$, $a'=i(\mu'_F,\gamma)$, and $b'=i(\mu'_\lambda,\gamma)$. If $a'\leq A a$ and $b'\leq B b$, then
\[
a'+b'\leq\max\{A,B\}(a+b).
\]
Consequently,
\[
\operatorname{Dil}(\mu,\mu')
\leq
\max\bigl\{
\operatorname{Dil}(\mu_F,\mu'_F),
\operatorname{Dil}(\mu_\lambda,\mu'_\lambda)
\bigr\}
<\infty,
\]
contrary to $\mu\prec\mu'$. Thus $\lambda$ must contain a non-uniquely ergodic minimal component.

The same argument shows directly that $\mu'_\lambda\not\ll\mu_\lambda$. Indeed, if $\mu'_\lambda\ll\mu_\lambda$, then \Cref{lem:walsh_lem63} would again imply $\operatorname{Dil}(\mu_\lambda,\mu'_\lambda)<\infty$, and the preceding estimate would give $\operatorname{Dil}(\mu,\mu')<\infty$, contradicting $\mu\prec\mu'$.
\end{proof}

\begin{lem}\label{lem:increasing_implies_strictly}
Let $\vec\mu=(\mu_0,\ldots,\mu_n)$ be an increasing sequence of nonzero geodesic currents. Suppose that $|\mu_i|$ contains no non-uniquely ergodic minimal lamination for every $i$. Then there exist constants $d_0,\ldots,d_n>0$ such that
\[
(d_0\mu_0,\ldots,d_n\mu_n)
\]
is strictly increasing.
\end{lem}

\begin{proof}
By \Cref{lem:inc_implies_support_containment}, $\mu_i \inc \mu_{i+1}$ implies that
$|\mu_i|\subseteq|\mu_{i+1}|$ for every $i$. This containment must be
strict. Indeed, if $|\mu_i|=|\mu_{i+1}|$, then
\Cref{lem:non-uniquely_ergodic} would imply that their common
support contains a non-uniquely ergodic minimal lamination, contrary to
the hypothesis.

Since $\mu_i \inc \mu_{i+1}$, we have that $\Dil(\mu_{i+1}, \mu_i) < \infty$ for all $i$. By \Cref{lem:sinc_implies_strict_containment}, this, together with $|\mu_i| \subsetneq |\mu_{i+1}|$ gives, for each
$i=0,\ldots,n-1$, a constant $c_i>0$ such that
\[
\mu_i\sinc c_i\mu_{i+1}.
\]
Set $d_0=1$ and define recursively
\[
d_{i+1}=d_i c_i.
\]
The relation $\mu_i\sinc c_i\mu_{i+1}$ remains true if both sides are scaled by $d_i$, and so
\[
d_i\mu_i\sinc d_ic_i\mu_{i+1}
=d_{i+1}\mu_{i+1}.
\]
Thus $(d_0\mu_0,\ldots,d_n\mu_n)$ is strictly increasing.
\end{proof}

The following lemma says that in some cases, if $\mu \inc \mu'$ then we can find a third geodesic current lying between them. If $|\mu| = F \cup \lambda \cup \delta$ for geodesic laminations $\lambda$ and $\delta$ disjoint from each other and from $F$, then we note that
\[
 \mu = \mu_F + \mu_\lambda + \mu_\delta
\]
where $\mu_F$ fills $F$, and $\mu_\lambda$ and $\mu_\delta$ are disjoint measured laminations supported on $\lambda$ and $\delta$. We say the \textbf{component of $\mu$ supported on $F \cup \lambda$} is $\mu' = \mu_F + \mu_\lambda$.

\begin{lem}\label{lem:middle_curve}
Let $\mu,\nu\in\mathcal C(S)$ satisfy
$\operatorname{Dil}(\mu,\nu)<\infty$. Suppose that there are a subsurface
$F\subseteq S$ and geodesic laminations $\lambda$ and $\delta$, pairwise
disjoint, such that
\[
|\nu|\subseteq F\cup\lambda
\qquad\text{and}\qquad
|\mu|=F\cup\lambda\cup\delta.
\]
Write
\[
\mu=\mu_F+\mu_\lambda+\mu_\delta
\]
according to these components, and set
$\mu'=\mu_F+\mu_\lambda$. Then
\[
\operatorname{Dil}(\mu',\nu)<\infty.
\]
\end{lem}

\begin{proof}
Write $\nu=\nu_F+\nu_\lambda$, where $|\nu_F|\subseteq F$ and
$|\nu_\lambda|\subseteq\lambda$. Since $\mu_F$ fills $F$,
\Cref{lem:bounded_sup} gives
\[
A:=\Dil(\mu_F,\nu_F)<\infty.
\]
We claim that
$B\coloneqq \Dil(\mu_\lambda,\nu_\lambda)<\infty$ as well. There
is nothing to prove if $\nu_\lambda=0$, so suppose otherwise.

Assume for contradiction that
$\Dil(\mu_\lambda,\nu_\lambda)=\infty$. By
\Cref{lem:walsh_lem64}, this means that
$\nu_\lambda\not\ll\mu_\lambda$. Moreover, there are closed geodesics
$\alpha_k$ contained in the supporting subsurface of an ergodic component
of $\nu_\lambda$ such that
\[
\frac{i(\nu_\lambda,\alpha_k)}
     {i(\mu_\lambda,\alpha_k)}
\longrightarrow+\infty.
\]
Since $F$ and $\delta$ are disjoint from $\lambda$, these geodesics may
be chosen so that
\[
i(\mu_F,\alpha_k)=i(\mu_\delta,\alpha_k)=i(\nu_F,\alpha_k)=0.
\]
Consequently,
\[
\frac{i(\nu,\alpha_k)}{i(\mu,\alpha_k)}
=
\frac{i(\nu_\lambda,\alpha_k)}
     {i(\mu_\lambda,\alpha_k)}
\longrightarrow+\infty,
\]
contradicting the hypothesis
$\operatorname{Dil}(\mu,\nu)<\infty$. Therefore $B<\infty$.

For every closed geodesic $\gamma$, we now have
\[
i(\nu,\gamma)
=i(\nu_F,\gamma)+i(\nu_\lambda,\gamma)
\leq A\,i(\mu_F,\gamma)+B\,i(\mu_\lambda,\gamma)
\leq\max\{A,B\}\,i(\mu',\gamma).
\]
Thus
$\operatorname{Dil}(\mu',\nu)\leq\max\{A,B\}<\infty$.
\end{proof}

\section{Bounding the lengths of increasing sequences}
\label{sec:bound_increasing_seq}

The purpose of this section is to reduce the length bound for increasing
sequences to a topological count. The obstruction to such a count is
that consecutive geodesic currents may have the same canonical support while
carrying different transverse measures on some non-uniquely ergodic
minimal lamination. We remove this obstruction in three steps. First,
we isolate one non-uniquely ergodic minimal component and record how
its ergodic measures occur along the sequence. Second, we replace
these ergodic measures by a strictly nested sequence of multicurves,
in such a way that the resulting sequence remains increasing. We
iterate this substitution until no non-uniquely ergodic minimal
component remains, and then rescale the resulting sequence to make it
strictly increasing. The desired bound then follows by counting the
successive enlargements of the canonical supports.

\begin{lem}\label{lem:minimal-component-interval}
Let $(\mu_0,\ldots,\mu_n)$ be an increasing sequence, with
$\mu_k=\mu_{F_k}+\mu_{\lambda_k}$ for $k = 1, \dots, n$. Let $\delta$ be a minimal
component of some $\lambda_k$. Then there is a maximal interval
$I = \{i,\ldots,i+j\}$ so that $\delta$ occurs as a component of
$\lambda_k$ if and only if $k \in I$. If $i+j<n$, then $\delta\subset F_{i+j+1}$.
\end{lem}

\begin{proof}
By \Cref{lem:inc_implies_support_containment},
$\lambda_k\subseteq F_{k+1}\cup\lambda_{k+1}$ for each $k$. Since $\delta$ is
minimal, either $\delta\subseteq\lambda_{k+1}$ or
$\delta\subseteq F_{k+1}$. In the latter case,
$F_{k+1}\subseteq F_\ell$ for every $\ell>k+1$, so $\delta$ cannot
subsequently occur in any $\lambda_\ell$.
\end{proof}

\begin{lem}\label{lem:ergodic-coefficient-persistence}
Let $(\mu_0,\ldots,\mu_n)$ be an increasing sequence, with
$\mu_k=\mu_{F_k}+\mu_{\lambda_k}$ for $k = 1, \dots, n$. Let $\delta$ be a minimal
component of some $\lambda_k$.
Suppose that $\delta_1,\ldots,\delta_r$ are the projectively distinct
ergodic measures supported on $\delta$, which occur in the sequence. Let $I$ be the active interval for $\delta$. For each $k \in I$, write
\[
\mu_k=\widehat\mu_k+\sum_{\ell=1}^r c^k_\ell\delta_\ell,
\]
where $\widehat\mu_k$ has no component supported on $\delta$. Then for each $\ell = 1, \dots, r$,
$c^{k+1}_\ell=0$ implies $c^k_\ell=0$.
\end{lem}

\begin{proof}
Since $\mu_k\inc\mu_{k+1}$, we have
$\Dil(\mu_{k+1},\mu_k)<\infty$. Thus there is a constant $D_k<\infty$
such that $i(\mu_k,\gamma)\leq D_k i(\mu_{k+1},\gamma)$ for every
geodesic current $\gamma$. 

If $c^{k+1}_\ell=0$ but $c^k_\ell>0$, Lemmas~\ref{lem:walsh_lem63}
and~\ref{lem:walsh_lem64} give a sequence of  closed geodesics
$\alpha_m\subset\Sigma(\delta)$ such that
\[
\frac{i(\mu_k,\alpha_m)}
     {i(\mu_{k+1},\alpha_m)}
\longrightarrow+\infty,
\]
contradicting $\Dil(\mu_{k+1},\mu_k)<\infty$. Therefore
$c^{k+1}_\ell=0$ implies $c^k_\ell=0$.

It follows that once $\delta_\ell$ occurs, it continues to occur for as
long as $\delta$ remains in the lamination component. 
\end{proof}

We record the following two facts, before arguing that non-uniquely ergodic components in an increasing sequence can be removed without changing the length of the sequence.
\begin{lem}\label{lem:replacement-multicurves}
Let $\Sigma$ be a connected subsurface admitting a pants decomposition
with $\kappa$ curves. For every $r\leq\kappa$, there exist nonzero
multicurves $\nu_1,\ldots,\nu_r\subset\Sigma$ such that
\[
\nu_1\sinc\nu_2\sinc\cdots\sinc\nu_r
\]
and $i(\nu_\ell,\nu_m)=0$ for all $\ell,m$.
\end{lem}

\begin{proof}
It follows readily from~\Cref{lem:build_increasing}.
\end{proof}

Next, we record the following fact:
\begin{lem}
    Let $(\mu_0, \dots,\mu_n)$ be an increasing sequence. Then at most finitely many distinct minimal, non-uniquely ergodic laminations occur in the support of $\mu_0, \dots, \mu_n$.
\end{lem}
\begin{proof}
    For each $i$, $|\mu_i| = F_i \cup \lambda_i$ for some lamination $\lambda_i$. Each $\lambda_i$ has finitely many minimal components, and so finitely many minimal laminations occur in the support of $\mu_0, \dots, \mu_n$. In particular, finitely many minimal laminations in the support of $\mu_0, \dots, \mu_n$ are non-uniquely ergodic.
\end{proof}

The following lemma will allow us to replace any increasing sequence with a strictly increasing sequence of the same length.
\begin{lem}\label{lem:remove-nue-component}
Let $(\mu_0, \dots, \mu_n)$ be an increasing sequence. Then there exists an increasing sequence $(\mu_0', \dots, \mu_n')$ of the same length, with strictly fewer non-uniquely ergodic laminations appearing the in the support of $\mu_0', \dots, \mu_n'$ than in the support of $\mu_0, \dots, \mu_n$.
\end{lem}
\begin{proof}
Let $\delta$ be a minimal, non-uniquely ergodic lamination that occurs in the support of at least one $\mu_0, \dots, \mu_n$. Let $\{i,\ldots,i+j\}$ be the maximal interval for $\delta$ supplied by
\Cref{lem:minimal-component-interval}. Let
$\delta_1,\ldots,\delta_r$ be the projectively distinct ergodic measures
supported on $\delta$ that occur on this interval. By
\cite[Page~3]{LenzhenMasur2010}, we have $r\leq\kappa$, where $\kappa = \xi(\Sigma(\delta))$
is the complexity of
$\Sigma(\delta)$.

Write
\[
\mu_k=\widehat\mu_k+\sum_{\ell=1}^r c^k_\ell\delta_\ell
\]
for $k=i,\ldots,i+j$, where $\widehat\mu_k$ has no component supported
on $\delta$. Reindex the $\delta_\ell$ according to their first
occurrence in the sequence $\mu_i, \dots, \mu_{i+j}$. Measured laminations that first occur at the same index may be ordered
arbitrarily within that group.
By \Cref{lem:replacement-multicurves}, we can choose an increasing sequence of disjoint, simple closed multicurves
$\nu_1,\ldots,\nu_r$.
Then we define
\[
\mu'_k=\widehat\mu_k+\sum_{\ell=1}^r c^k_\ell\nu_\ell
\]
for $k=i,\ldots,i+j$. For $k$ outside this interval, set $\mu'_k=\mu_k$. We
claim that $(\mu'_0,\ldots,\mu'_n)$ is increasing.

First consider the entrance comparison between $\mu_{i-1}'$ and $\mu_i'$. By assumption, $\mu_{i-1} \inc \mu_i$. Thus, $|\mu_{i-1}| \subseteq |\mu_i|$. Moreover,$\delta$ does not occur
in $\mu_{i-1}$, so we have
$|\mu_{i-1}|\subseteq|\widehat\mu_i|$. Since $\mu_{i-1} \inc \mu_i$ and $|\mu_{i-1}|\subseteq|\widehat\mu_i|$, \Cref{lem:middle_curve}
therefore gives $\Dil(\widehat\mu_i,\mu_{i-1})<\infty$. Since
$\widehat\mu_i$ is a component of $\mu'_i$, it follows that
$\Dil(\mu'_i,\mu_{i-1})<\infty$.

On the other hand, the multicurves $\nu_1, \dots, \nu_r$ lie in $\Sigma(\delta)$,
which is disjoint from $|\mu_{i-1}|$. Any closed curve that fills (and is contained in) $\Sigma(\delta)$ therefore intersects $\mu'_i$ but not
$\mu_{i-1}$. Hence $\Dil(\mu_{i-1},\mu'_i)=\infty$, and consequently
\[\mu_{i-1}\inc\mu'_i.\]

Now fix $k$ with $i<k\leq i+j$. Then $\mu_{i-1} \inc \mu_k$ by assumption. Thus, $\Dil(\mu_k, \mu_{k-1}) < \infty$. As $\widehat \mu_{k-1}$ is dominated by $\mu_{k-1}$, we still have $\Dil(\mu_k, \widehat{\mu_{k-1}}) < \infty$. But the support of $\widehat{\mu_{k-1}}$ does not contain $\delta$, so $|\widehat{\mu_{k-1}}| \subseteq |\widehat{\mu_k}|$. So, \Cref{lem:middle_curve} gives
$\Dil(\widehat\mu_k,\widehat\mu_{k-1})<\infty$. Moreover,
\Cref{lem:ergodic-coefficient-persistence} and the relations
$\nu_1\sinc\cdots\sinc\nu_r$ imply
\[
\Dil\left(
 \sum_{\ell=1}^r c^k_\ell\nu_\ell,
 \sum_{\ell=1}^r c^{k-1}_\ell\nu_\ell
\right)<\infty.
\]
Using the componentwise maximum estimate, we obtain
$\Dil(\mu'_k,\mu'_{k-1})<\infty$.

It remains to prove $\Dil(\mu'_{k-1},\mu'_k)=\infty$. Suppose first
that a new ergodic component $\delta_\ell$ appears at index $k$. By the
ordering of the components, only
$\nu_1,\ldots,\nu_{\ell-1}$ occur in $\mu'_{k-1}$, whereas $\nu_\ell$
occurs in $\mu'_k$. The relation
$\nu_{\ell-1}\sinc\nu_\ell$ supplies a geodesic current $\gamma$ in $\Sigma(\delta)$ that intersects $\nu_\ell$ but not $\nu_1, \dots, \nu_{\ell-1}$. Thus, $\gamma$ intersects
$\mu'_k$ but not $\mu'_{k-1}$. Therefore
$\Dil(\mu'_{k-1},\mu'_k)=\infty$.

Suppose instead that no new ergodic component supported on $\delta$
appears at index $k$. The two components supported on $\delta$ are then
mutually absolutely continuous, so their dilatation is finite by
\Cref{lem:walsh_lem64}. Since
$\Dil(\mu_{k-1},\mu_k)=\infty$, the parts away from $\delta$ must
satisfy
\[
\Dil(\widehat\mu_{k-1},\widehat\mu_k)=\infty.
\]
If $|\widehat\mu_{k-1}|\subsetneq|\widehat\mu_k|$, then there is some geodesic current $\gamma$ with support in $|\widehat\mu_k| \setminus |\widehat\mu_{k-1}|$. Then $\gamma$ intersects $\widehat\mu_k$ but not $\widehat\mu_{k-1}$. Since the support of $\widehat\mu_k$ is disjoint from
$\delta$, we also have that $\gamma$ does not intersect $\delta$. Thus, $i(\gamma, \mu_{k-1}')=0$ while $i(\gamma, \mu_k') \neq 0$. This gives $\Dil(\mu_{k-1}',\mu_k')=\infty$ in this case.

If $|\widehat\mu_{k-1}|=|\widehat\mu_k|$, \Cref{lem:non-uniquely_ergodic} implies that $|\widehat \mu_k|$ has some other non-uniquely ergodic lamination $\lambda$ in its support. Then \cite[Lemmas 6.3 and 6.4]{Walsh2014Horoboundary}, recorded in Lemmas \ref{lem:walsh_lem63},
and~\ref{lem:walsh_lem64}, give a sequence of closed geodesics
$\alpha_m$, contained in $\Sigma(\lambda)$, such that
\[
\frac{i(\widehat\mu_k,\alpha_m)}
     {i(\widehat\mu_{k-1},\alpha_m)}
\longrightarrow+\infty.
\]
Since $\Sigma(\lambda) \subseteq |\widehat \mu_k|$, the curves $\alpha_m$ are all disjoint from $\Sigma(\delta)$, and are thus disjoint from $\nu_1, \dots, \nu_r$. This implies that $\frac{i(\mu_k',\alpha_m)}
     {i(\mu_{k-1}',\alpha_m)}
\to+\infty$, and so $\Dil(\mu'_{k-1},\mu'_k)=\infty$. Thus, in all cases,
\[\mu'_{k-1}\inc\mu'_k.\]

Finally, suppose $i+j<n$. We must then show $\mu_{i+j}' \inc \mu_{i+j+1}'$. Since $\mu_{i+1} \inc \mu_{i+j+1}$ and $\widehat \mu_{i+1}$ is dominated by $\mu_{i+1}$, we have $\Dil(\mu_{i+j+1}, \widehat \mu_{i+1}) < \infty$. Suppose $\mu_{i+j+1} = \mu_{F_{i+j+1}} \cup \mu_{\lambda_{i+j+1}}$ is the canonical decomposition. By
\Cref{lem:minimal-component-interval}, we have
$\delta\subset F_{i+j+1}$, and so $\Sigma(\delta) \subseteq F_{i+j+1}$. 
And by \Cref{lem:bounded_sup}, since every $\nu_\ell$ is contained in
$\Sigma(\delta)$ for each $\ell = 1, \dots, r$, we have $\Dil(\mu_{F_{i+j+1}}, \sum_{\ell=1}^r c^k_\ell\nu_\ell) < \infty$. Since $\mu_{F_{i+j+1}}$ is dominated by $\mu_{i+j+1}$, we then get $\Dil(\mu_{i+j+1}, \sum_{\ell=1}^r c^k_\ell\nu_\ell) < \infty$. Therefore,
$\Dil(\mu_{i+j+1},\mu'_{i+j})<\infty$.

Every $\nu_\ell$ occurring in $\mu'_{i+j}$ is an essential
non-peripheral simple multicurve in $\Sigma(\delta)$. Moreover, the $\nu_1, \dots, \nu_r$ are pairwise disjoint by construction, so in particular $i(\nu_\ell,\nu_1)=0$ for all $\ell$. Since
$\widehat\mu_{i+j}$ is disjoint from $\Sigma(\delta)$, we have
$i(\mu'_{i+j},\nu_1)=0$. On the other hand,
$\Sigma(\delta)\subset F_{i+j+1}$ and $\nu_1$ is essential and
non-peripheral in $\Sigma(\delta)$. Since $\mu_{F_{i+j+1}}$ fills
$F_{i+j+1}$, we have $i(\mu_{i+j+1},\alpha)>0$. Therefore $\Dil(\mu'_{i+j},\mu_{i+j+1})=\infty$, and hence
\[\mu'_{i+j}\inc\mu_{i+j+1}.\]

We have produced an increasing sequence $(\mu_0', \dots, \mu_n')$ of the same length as our original sequence in which
$\delta$ no longer occurs. Since all components supported on $\delta$ were replaced by weighted multicurves, no new non-uniquely ergodic minimal component has been
introduced. Therefore, the new sequence has one fewer minimal non-uniquely ergodic component than the original.
\end{proof}

\begin{lem}\label{lem:replacement}
Let $\vec\mu=(\mu_0,\ldots,\mu_n)$ be an increasing sequence of nonzero
geodesic currents. Then there exists a strictly increasing sequence
$\vec\mu'=(\mu'_0,\ldots,\mu'_n)$ of the same length.
\end{lem}

\begin{proof}
Only finitely many minimal non-uniquely ergodic laminations occur among
the lamination components of $\mu_0,\ldots,\mu_n$: there are finitely
many geodesic currents, and each of their lamination components has finitely
many minimal components.

Apply \Cref{lem:remove-nue-component} successively to these
minimal non-uniquely ergodic components. Each application removes one
such component and introduces only multicurves, so the procedure
terminates after finitely many steps.

The resulting sequence is increasing and none of its supports contains
a non-uniquely ergodic minimal lamination. Lemma
\ref{lem:increasing_implies_strictly} therefore provides positive
rescaling constants that turn it into a strictly increasing sequence
of the same length.
\end{proof}

Given a closed orientable surface $S$ of genus $g \geq 2$, a set of pairwise disjoint simple closed geodesics $\alpha_1, \dots, \alpha_{3g-3}$ defines a \textbf{pants and annuli decomposition} $\calP \cup \A$ where
\begin{align*}
 \calP & = \{P_1, \dots, P_{2g-2}\}\\
 \A & = \{\alpha_1, \dots, \alpha_{3g-3}\}
\end{align*}
where $P_1, \dots, P_{2g-2}$ are the (open) pairs of pants in the complementary regions of $\alpha_1, \dots, \alpha_{3g-3}$. Because the pairs of pants are open, all elements of
$\mathcal P\cup\mathcal A$ are pairwise disjoint.
\begin{rem}
 Clearly, a pants and annuli decomposition is equivalent to a maximal collection of pairwise disjoint, pairwise nonisotopic simple
closed curves $\alpha_1, \dots, \alpha_{3g-3}$. However, we will want to work with the open pairs of pants cut out by these curves separately from the curves themselves. Hence, the definition.
\end{rem}

Write a geodesic lamination $\lambda$ as a union of its minimal
components, $\lambda=\delta_1\cup\cdots\cup\delta_r$. Recall
$\Sigma(\lambda)=\Sigma(\delta_1)\cup\cdots\cup\Sigma(\delta_r)$.
Thus $\Sigma(\lambda)$ is allowed to be disconnected. If $\delta_j$ is
a simple closed geodesic, then $\Sigma(\delta_j)$ denotes an annular
neighborhood of $\delta_j$.

\begin{lem}\label{lem:support-boundary-compatibility}
Let $(\mu_0,\ldots,\mu_n)$ be an increasing sequence, and write
$\mu_i=\mu_{F_i}+\mu_{\lambda_i}$. Then the boundary components of the subsurfaces
$F_0,\ldots,F_n,\Sigma(\lambda_0),\ldots,\Sigma(\lambda_n)$ have pairwise zero intersection.
\end{lem}

\begin{proof}
Let $i < j$. By \Cref{lem:inc_implies_support_containment}, the fact that $(\mu_0, \dots, \mu_n)$ is an increasing sequence gives that $F_i\subseteq F_j$. Thus, we immediately get that $i(\partial F_i, \partial F_j) = 0$. 

Moreover, \Cref{lem:inc_implies_support_containment} also gives $\lambda_i\subseteq F_j\cup\lambda_j$ for all $i<j$. Thus, 
\[
\Sigma(\lambda_i) \subseteq F_j \cup \Sigma(\lambda_j).
\]
Since $F_j$ and $\Sigma(\lambda_j)$ are disjoint by assumption, this immediately gives 
\[
i(\partial \Sigma(\lambda_i), \partial F_i \cup \partial \Sigma(\lambda_j)) = 0.
\]
Therefore, the boundaries of $F_0,\ldots,F_n,\Sigma(\lambda_0),\ldots,\Sigma(\lambda_n)$ all have pairwise zero intersection.
\end{proof}

\begin{lem}\label{lem:pants_annuli}
Let $\vec\mu=(\mu_0,\ldots,\mu_n)$ be an increasing sequence of
geodesic currents. Then there is a pants-and-annuli decomposition
$\mathcal P\cup\mathcal A$ of $S$ such that, whenever
$\mu_i\in\mathcal C_{F_i,\lambda_i}$, both $F_i$ and
$\Sigma(\lambda_i)$ and their boundary components are unions of elements of
$\mathcal P\cup\mathcal A$.
\end{lem}

\begin{proof}
By \Cref{lem:support-boundary-compatibility}, the boundary
components of all the subsurfaces $F_i$ and
$\Sigma(\lambda_i)$ may be realized simultaneously as a simple multicurve.
Let $\mathcal B$ be the multicurve obtained by taking one representative
from each isotopy class appearing among these boundary components.

Extend $\mathcal B$ to a pants decomposition $\mathcal A$ of $S$, and
let $\mathcal P$ be the collection of open pairs of pants complementary
to $\mathcal A$. Every boundary component of every $F_i$ and
$\Sigma(\lambda_i)$ belongs to $\mathcal A$. It follows that each
$F_i$ and each $\Sigma(\lambda_i)$ and their boundary components is a union of elements of
$\mathcal P\cup\mathcal A$.
\end{proof}

Given an increasing sequence $\vec\mu$, we call any pants-and-annuli
decomposition satisfying the conclusion of \Cref{lem:pants_annuli} a
\textbf{pants-and-annuli decomposition suited to} $\vec\mu$.

\begin{lem}\label{lem:n_at_most_8g-0}
Let $\vec\mu=(\mu_0,\ldots,\mu_n)$ be an increasing sequence of nonzero
geodesic currents on a closed surface $S$ of genus $g$. Then
\[
n\leq 8g-9.
\]
Equivalently, an increasing sequence contains at most $8g-8$ nonzero
currents.
\end{lem}

\begin{proof}
By \Cref{lem:replacement}, we may replace $\vec\mu$ by a strictly
increasing sequence of the same length. Let
$\mathcal P\cup\mathcal A$ be a pants-and-annuli decomposition suited
to this sequence. Thus $|\mathcal P|=2g-2$ and
$|\mathcal A|=3g-3$.

Write $\mu_i=\mu_{F_i}+\mu_{\lambda_i}$. To include the initial geodesic current
in the count, set $F_{-1}=\varnothing$ and
$\Sigma(\lambda_{-1})=\varnothing$. For each $i=0,\ldots,n$, make the
following definitions.

We say that $F_i$ \textbf{chooses a pair of pants} $P\in\mathcal P$ if
$P\subseteq F_i$ but $P\not\subseteq F_{i-1}$. We say that $F_i$
\textbf{chooses an annulus} $A\in\mathcal A$ if its core curve is essential and
nonperipheral in $F_i$, but not in $F_{i-1}$. Finally, we say that
$\lambda_i$ \textbf{chooses an annulus} $A\in\mathcal A$ if
$A\subseteq\Sigma(\lambda_i)$ but
$A\not\subseteq\Sigma(\lambda_{i-1})$. In this last definition, the
core curve of $A$ is allowed to be peripheral in
$\Sigma(\lambda_i)$; in particular, $\lambda_i$ itself may be that
curve.

We claim that at every index $i=0,\ldots,n$, at least one choice is
made. For $i=0$, this follows from $\mu_0\neq0$. Indeed, either
$F_0\neq\varnothing$, in which case $F_0$ contains a pair of pants or
an essential annulus, or $\lambda_0\neq0$, in which case
$\Sigma(\lambda_0)$ contains an annulus from $\mathcal A$.

For $i\geq1$, the fact that we have a strictly increasing sequence allows us to apply
\Cref{lem:sinc_implies_strict_containment}, which states
$|\mu_{i-1}|\subsetneq|\mu_i|$. If
$F_{i-1}\subsetneq F_i$, then $F_i$ chooses either a pair of pants or
an annulus. If $F_{i-1}=F_i$, then the strict containment occurs in the
lamination part. Hence some component of
$\Sigma(\lambda_i)$ was not contained in
$\Sigma(\lambda_{i-1})$, and $\lambda_i$ chooses an annulus.

Each pair of pants can be chosen by the subsurface part at most once,
because $F_i\subseteq F_{i+1}$ for every $i$. Similarly, each annulus
can be chosen by the subsurface part at most once.

Each annulus can also be chosen by the lamination part at most once.
Indeed, \Cref{lem:sinc_implies_strict_containment} gives
\[
\Sigma(\lambda_i)\subseteq
F_{i+1}\cup\Sigma(\lambda_{i+1}).
\]
Thus, after an annulus has appeared in the lamination part, it either
remains there or is absorbed into the subsurface part. Once it is
absorbed into $F_j$, the inclusions $F_j\subseteq F_{j+1}$ prevent it
from returning to a later lamination part.

There are therefore at most $2g-2$ choices of pairs of pants,
$3g-3$ choices of annuli by the subsurface parts, and $3g-3$ choices
of annuli by the lamination parts. The total number of available
choices is
\[
(2g-2)+(3g-3)+(3g-3)=8g-8.
\]
Since at least one new choice not made at any earlier index is made at each of the $n+1$ indices
$0,\ldots,n$, we obtain $n+1\leq8g-8$, and hence
$n\leq8g-9$.
\end{proof}

\subsection{Proof of \Cref{mainthm:max_rank}}
\label{sec:proof_thm_max_rank}
We are now ready to prove \Cref{mainthm:max_rank}. We recall its statement here.
\MaxRank*
\begin{proof}
By \Cref{cor:decomp_increasing_seq},
\[
\PSimplex(\vec \mu) \simeq\cup_{j=1}^N \PSimplex(\vec \mu^j)
\]
where $\vec \mu^j$ are increasing sequences for each $j=1,\cdots,N$.

By   \Cref{cor:upper_bound_rank} and \Cref{lem:n_at_most_8g-0}, we have
\[
\rank(\PSimplex(\vec \mu^j))\leq 8g-9,
\]
for each $j=1,\cdots, N$,
and thus by induction on \Cref{prop:rank_union} we get 
\[
\rank(\PSimplex(\vec \mu)) \leq 8g-9.\qedhere
\]
\end{proof}
\section{Simplices on strictly increasing sequences}
\label{sec:strictly_increasing_seq}

In this section, we show that when a sequence $\vec \mu$ is \textit{strictly increasing}, then the geometry of $\PSimplex(\vec \mu)$ is particularly simple. 
In fact, we will show that if $\vec \mu = (\mu_0, \dots, \mu_n)$ is a strictly increasing sequence, then not only does its simplex have rank exactly $n$, but in fact, it is coarsely isometric to $\R^n_+$ (\Cref{lem:qi_to_Rn}). This will allow us to deduce that its flat and geometric ranks, as well as asymptotic dimension, are exactly $n$. Together with \Cref{lem:build_increasing} this will settle \Cref{mainthm:criterion}.

\begin{lem}
\label{lem:qi_to_Rn}
    Let $\vec \mu = (\mu_0, \dots, \mu_n)$ be a jointly filling and strictly increasing sequence. Then \[
    \PSimplex(\vec \mu)\simeq_{2,C} \R^n_+,\] where $\R^n_+$ is equipped with the Euclidean metric, and the additive constant $C$ in the quasi-isometry depends on data associated to $\vec \mu$, and $n$.
\end{lem}

\begin{proof} 
 Consider the region
 \[
  \mathcal R \coloneqq \{(x_0, \dots, x_n) \in \R^{n+1}_+\ | \ x_n = 1, x_i > 2x_{i+1}, \forall i\}
 \]
 with the metric $d_L(\vec x, \vec y) = \max_i \{|\log \frac{x_i}{y_i}|\}$.

\subsection*{1. $\mathcal{R}$ is quasi-isometric to $\R^n_+$} First, we note that $\mathcal R$ is isometric to the set $\mathcal S = \{(X_0, \dots, X_n)\ | \ X_n = 0; X_i \geq X_{i+1} + \log 2\}$ with the $\ell^\infty$-metric $ d_\infty(\vec X, \vec Y) \coloneqq \max_i\{|X_i - Y_i|\}$ via the isometry
 \begin{align*}
  I : \mathcal S &\to \mathcal R \\
    (X_0, \dots, X_n) &\mapsto (e^{X_0}, \dots, e^{X_n})
 \end{align*}
as the map is clearly a bijection, and
\begin{align*}
 d_\infty(\vec X, \vec Y) & = \max_i\{|X_i - Y_i|\} \\
  & = \max_i\{|\log e^{X_i} - \log e^{Y_i}|\} \\
  & = \max_i\{|\log \frac{e^{X_i}}{e^{Y_i}}|\} \\
  & = d_L(I(\vec X), I(\vec Y)).
\end{align*}
Moreover, $\mathcal S$ is isometric to the region $\{(X_0, \dots, X_n) \ | \ X_0 \geq X_1 \geq \dots \geq X_n = 0\} $, which is the positive convex hull of the vectors $\vec v_0 = \vec e_0, \vec v_1 = \vec e_0 + \vec e_1, \dots, \vec v_{n-1} = \vec e_0 + \dots + \vec e_{n-1}$, where $\vec e_0, \dots, \vec e_n$ are the standard basis vectors. Finally, \Cref{lem:sectors_QI} gives us that $\mathcal S$ is quasi-isometric to $\R^n_+$.

\subsection*{2. $\mathcal{R}$ is quasi-isometric to $\PSimplex(\vec \mu)$}

 On the other hand, we claim that the natural map:
\begin{align*}
 \phi: \mathcal R & \to \PSimplex(\vec \mu) \\
 \vec x  &\mapsto [\vec x \cdot \vec \mu]
\end{align*}
where $\vec x \cdot \vec \mu$ is as in Eq.~\eqref{eq:scalar_product}, is a quasi-isometry.

First, note that $\phi$ is injective. By \Cref{lem:strictly-increasing-independent}, the currents
$\mu_0,\ldots,\mu_n$ are linearly independent. Consequently, if
$\vec x,\vec y\in \mathcal R$ satisfy
\[
    [\vec x\cdot\vec\mu]=[\vec y\cdot\vec\mu],
\]
then there is $c>0$ such that $\vec x=c\vec y$. Since vectors in
$\mathcal R$ satisfy $x_n=y_n=1$, we have $c=1$, and hence
$\vec x=\vec y$.

In fact, let
\[
 \mathcal R \cdot \vec \mu = \{\vec x \cdot \vec \mu \ | \ \vec x \in \mathcal R\}
\]
 We take $\mathcal R \cdot \vec \mu$ as our preferred lift of the image of $\phi$ to $\Cf(S)$. 

\textbf{First, we prove $\phi$ is a quasi-isometric embedding.} For the upper bound, take any $\vec x, \vec y \in \R^{n+1}_+$, and any $\gamma \in \Cf(S)$. Then,
\begin{align*}
\frac{i(\vec x \cdot \vec \mu, \gamma)}{i(\vec y \cdot \vec \mu, \gamma)}
 & = \frac{\sum_i x_i\;  i(\mu_i, \gamma)}{\sum_i y_i\;  i(\mu_i, \gamma)}\\
 & \leq \max_i  \frac{x_i\;  i(\mu_i, \gamma)}{y_i \; i( \mu_i, \gamma)}\\
 & \leq \max_i  \frac{x_i}{y_i}
\end{align*}
where we use \Cref{lem:ratiosum} in the second line.

Therefore,
\begin{align*}
 \dsym(\vec x \cdot \mu, \vec y \cdot \mu)
 & = \log \sup_\gamma \frac{i(\vec x \cdot \vec \mu, \gamma)}{i(\vec y \cdot \vec \mu, \gamma)} +
 \log \sup_\gamma \frac{i(\vec y \cdot \vec \mu, \gamma)}{i(\vec x \cdot \vec \mu, \gamma)} \\
 &\leq \log \max_i  \frac{x_i}{y_i} + \log \max_i  \frac{y_i}{x_i} \\
 &\leq 2 \max_i \left|\log  \frac{x_i}{y_i} \right|\\
 & =2 d_L(\vec x, \vec y).
\end{align*}

For the lower bound, take any $\vec x \in \mathcal R$. Because $\vec \mu$ is strictly increasing, we know that, for all $i$, there is some geodesic current $\gamma_i$ so that
\[
 i(\mu_i, \gamma_i) \neq 0, \text{ but } i(\mu_{i-1}, \gamma_i) = 0.
\]
As $i(\mu_{i-1}, \gamma) \leq i(\mu_i, \gamma)$ for geodesic currents $\gamma$, we have in fact that, for all $j \leq i-1$ and $k \geq i$,
\[
 i(\mu_j, \gamma_i) = 0, \text{ but } i(\mu_k, \gamma_i) \neq 0.
\]

Let
\[
 C = \max_{i,j} i(\mu_i, \gamma_j)
\]
and note that this constant depends only on the sequence $\vec \mu$ and choices of $\gamma_i$, but not on $\vec x \in \mathcal R$. Then for each $i$,
\begin{align*}
 i(\vec x \cdot \vec \mu, \gamma_i) & = \sum_j x_j i(\mu_j ,\gamma_i) \\
  & = \sum_{j \geq i} x_j i(\mu_j, \gamma_i) \\
  & \leq \sum_{j \geq i} \frac{1}{2^{j-i}}x_i i(\mu_j, \gamma_i) \\
  & \leq 2 C x_i
\end{align*}
where we used the fact that $\vec x \in \mathcal R$, and so $x_i \geq 2 x_{i+1}$, and therefore, $x_j < \frac{1}{2^{j-i}} x_i$ for each $j$. In particular, we have
\[
 x_i i(\mu_i, \gamma_i) \leq i(\vec x \cdot \vec \mu, \gamma_i) \leq 2C x_i
\]
for each $i$. Thus, for $\vec x, \vec y \in \mathcal{R}$, and each $i$,
\begin{align*}
\sup_\gamma \frac{i(\vec x \cdot \vec \mu, \gamma)}{i(\vec y \cdot \vec \mu, \gamma)}
 & \geq \frac{i(\vec x \cdot \vec \mu, \gamma_i)}{i(\vec y \cdot \vec \mu, \gamma_i)}\\
 & \geq \frac{ x_i i(\mu, \gamma_i)}{2C y_i}
\end{align*}
where $C_i : = \frac{i(\mu_i, \gamma_i)}{2C}$ is a constant depending only on $\vec \mu$ and the choice of $\gamma_0, \dots, \gamma_n$. And so,
\[
 \sup_\gamma \frac{i(\vec x \cdot \vec \mu, \gamma)}{i(\vec y \cdot \vec \mu, \gamma)}  \geq \max C_i \frac{x_i}{y_i}.
\]

Setting $D = \min C_i$, we get
\begin{align*}
 \dsym (\vec x \cdot \mu, \vec y \cdot \mu) & = \log \sup_\gamma \frac{i(\vec x \cdot \vec \mu, \gamma)}{i(\vec y \cdot \vec \mu, \gamma)} +
 \log \sup_\gamma \frac{i(\vec y \cdot \vec \mu, \gamma)}{i(\vec x \cdot \vec \mu, \gamma)} \\
 & \geq \log D \max_i \frac{x_i}{y_i} +
 \log D \max_i \frac {y_i}{x_i}\\
 & \geq \log \max_i \left|\frac{x_i}{y_i}\right| + 2 \log D\\
 & = d_L(\vec x, \vec y) + 2 \log D.
\end{align*}
We note that $D < 1$ means that the $\log D$ term is negative.

Putting the two bounds together, we get
\[
 d_L(\vec x, \vec y) + 2 \log D \leq \dsym(\vec x \cdot \vec \mu, \vec y \cdot \vec \mu) \leq 2 d_L(\vec x, \vec y).
\]
In particular, $\phi: \mathcal R \to \PSimplex(\vec \mu)$ is a $(2, -2\log D)$-quasi-isometric embedding.

\textbf{Lastly, we prove that the image of $\phi$ is coarsely dense in $\PSimplex(\vec \mu)$.} Set 
\[
\mathcal R' \coloneqq \{(x_0, \dots, x_n) \in \R^{n+1}_+ \ | \ x_n = 1\}.
\]
Then $\mathcal R \subset \mathcal R'$. Defining $\mathcal R' \cdot \vec \mu$ analogously to $\mathcal R \cdot \vec \mu$ above, we see $\mathcal R' \cdot \vec \mu$ is a lift of $\PSimplex(\vec \mu)$ to $\C(S)$. 

Let $\vec x = (x_0, \dots, x_n) \in \mathcal R' \setminus \mathcal R$. Then there is some $i$ so that $x_i < 2x_{i+1}$. Let 
\[
\vec y = (2x_0, \dots, 2x_{i-1}, 2x_{i+1}, x_{i+1}, x_{i+2}, \dots, x_n).
\]
Note that if $x_j > 2 x_{j+1}$ for some $j$, then $y_j > 2 y_{j+1}$ for $y_j$ the $j^{th}$ entry of $\vec y$. Thus, after at most $n$ such replacements, we will arrive at an element of $\mathcal R$.

We will prove that the distance between $\vec x \cdot \vec \mu$ and $\vec y \cdot \vec \mu$ is at most $\log 3$. First, we note that since we have not decreased the coefficients of $\vec x$ to get $\vec y$, then for all $\gamma \in \Curves(S)$, 
\[
i(\vec x \cdot \vec \mu, \gamma) \leq i(\vec y \cdot \vec \mu, \gamma).
\]

On the other hand, because $\vec \mu$ is strictly increasing, we know that for every geodesic current $\gamma$, $i(\mu_i, \gamma) \leq i(\mu_{i+1}, \gamma)$. Thus,
\begin{align*}
    i(\vec y \cdot \vec \mu, \gamma) & = 
    2x_{i+1} i(\mu_i, \gamma) + x_{i+1} i(\mu_{i+1}, \gamma) + 2\sum_{j=1}^{i-1} x_j i(\mu_j, \gamma) + \sum_{j = i+2}^n x_j i(\mu_j, \gamma) \\
    & \leq 3 x_{i+1} i(\mu_{i+1}, \gamma) + 2\sum_{j=1}^{i-1} x_j i(\mu_j, \gamma) + \sum_{j = i+2}^n x_j i(\mu_j, \gamma) \\
    & \leq 3 \sum_{j=1}^n x_j i(\mu_j, \gamma) \\
    & = 3 i(\vec x \cdot \vec \mu,\gamma).
\end{align*}

Therefore, we see that 
\[
\dsym(\vec x \cdot \vec \mu, \vec y \cdot \vec \mu) \leq \log 3.
\]
We can do at most $n$ such replacements until we arrive at a vector $\vec y \in \mathcal R$. Thus, for any $\vec x \in \mathcal R'$, 
\[
\dsym(\vec x \cdot \vec \mu, R \cdot \vec \mu) \leq n \log 3.
\]
Therefore, the image of $\mathcal R$ is coarsely dense in $\PSimplex(\vec \mu)$. As $\mathcal R$ is quasi-isometric to $\R^n_+$, we get our result.
\end{proof}
\begin{rem}
    The quasi-isometry constants in \Cref{lem:qi_to_Rn} depend on the length $n$ of $\vec \mu$ and data related to the intersection numbers of $\vec \mu$. We show later in \Cref{lem:n_at_most_8g-0} that $n \leq 8g-8$, but the quasi-isometry constants still depend on all the geodesic currents $\mu_0, \dots, \mu_n$, and, in particular, on optimal choices of geodesic currents $\gamma_i$ for which $i(\mu_i, \gamma_i) = 0$ but $i(\mu_{i+1}, \gamma_{i+1}) \neq 0$.
\end{rem}

\subsection{Proof of \Cref{mainthm:criterion}}
\label{sec:proof_of_str_inc_thm}
We are now ready to prove \Cref{mainthm:criterion}. We recall its statement here.

\StrIncr*
\begin{proof}
    By \Cref{lem:qi_to_Rn}, $\PSimplex(\vec \mu)$ is quasi-isometric to $\R^n_+$. Flat rank and geometric rank are clearly quasi-isometry invariants, and so is asymptotic dimension (\Cref{lem:adim_qi}). Moreover, in \Cref{lem:build_increasing}, we show there is a strictly increasing sequence where $n = 8g-9$. The first statement follows immediately. The second part of the statement is \Cref{lem:build_increasing}.
\end{proof}

\section{Geometric rank of $\PCf(S)$ is infinite}
\label{sec:infinite}
In this section we give a proof of Theorem~\ref{mainthm:infinite_geo_rank}. We restate it here.

\InfiniteRank*

\begin{proof}
For the moment, let $\rho$ be an arbitrary filling geodesic current. For each $\vec x = (x_1, \dots, x_d)$, we define $e^{\vec x} = (e^{x_1}, \dots, e^{x_d})$.  For each $d, R > 0$, we will find a sequence $\vec \mu \coloneqq \vec \mu(d,R)$ of $d$ geodesic currents that depend on $R$, so that the map 
\begin{align*}
    f_R^d : [0,R]^d & \to \PCf(S) \\
     \vec x & \mapsto [\rho + e^{\vec x} \cdot \vec \mu]
\end{align*}
is a $(2,2\log(5))$ quasi-isometric embedding, for $[0,R]^d$ with the $\ell_\infty$ metric, and $\PCf(S)$ with the Lipschitz metric. Thus, for each $d$, we will have exhibited a family of uniform quasi-isometric embeddings of arbitrarily large boxes into $\PCf(S)$. This will prove that the geometric rank is unbounded.

For now, we fix $d \in \N$ and $R > \log d$. Note that if $f_R^d$ is defined for some $R$, then for any $r < R$ we can simply define $f_r^d$ by restricting the domain of $f_R^d$ appropriately. Thus, it is enough to define our families of functions $f_R^d$ for $d, R$ as above.

We first construct the sequence of currents $\vec \mu$. To define these, compare Figure~\ref{fig:twist}.
Let $P \subset S$ be an embedded pair of pants, with standard generators of the fundamental group $a,b$. For each $n \in \N$, let $\tau_n \coloneqq [a^nb] \subset P$ denote the $n$-times self-intersecting curve obtained by wrapping around $a$ for $n$-times and then following $b$. To simplify notation in what follows, we extend this definition to all $x \in \R$ with $x \geq 1$ by 
\[
\tau_x \coloneqq \tau_{\lfloor x \rfloor}.
\]

\begin{figure}[H]
   
     \centering
\includegraphics[width=50mm]{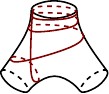}
    \caption{The curve $\tau_3$ on a pair of pants}
     \label{fig:twist}
\end{figure}
It is easy to see that, for all $n,k \in \N$~\footnote{Note that the geometric intersection number $i(c,c)$ of a \emph{primitive} curve $c$ (such as $\tau_n$) with itself differs from its self-intersection number $\operatorname{SI}(c)$ (i.e., minimal number of transverse double points of $c$ among all homotopy representatives) by a factor of 2, i.e., $i(c,c)=2\operatorname{SI}(c)$.}
\begin{equation}
\label{eq:tau_pattern}
i(\tau_n, \tau_k) = 2 \min \{n, k\} .
\end{equation}
See the left-hand graph of \Cref{fig:intersection_graph}.

Next, we choose a filling current $\rho$ for which
\begin{equation}
\label{eq:lower_n}
i(\rho, \tau_n) \leq n
\end{equation}
for all $n \in \N$, as follows.
Note that $\lim_n \tau_n/n = a$. Hence, for any filling geodesic current $\rho$, we have $\lim_n \frac{i\left(\rho,\tau_n/n\right)}{i(\rho,a)}=1$.
Thus, there exists $n_0 \in \mathbb N$ so that for all $n \geq n_0$, 
\[
i\left(\rho,\tau_n\right) \leq c_1 \cdot n
\]
where $c_1 \coloneqq (3/2)i(\rho,a)$.
Let $c_2\coloneqq \max_{k=1,\cdots,n_0} i\left(\rho,\tau_k\right)$.
Then, 
\[
i(\rho,\tau_n) \leq c_1 n + c_2.
\]

Rescaling $\rho$ by a large number we can guarantee Equation~\eqref{eq:lower_n} holds.

For every $k\in \mathbb{N}$, define
\begin{align*}
    n_k \coloneqq  e^{6kR}.
\end{align*}

We let
\[
\vec \mu \coloneqq \vec \mu(d,R) =  e^{4dR}\left(\frac{1}{\sqrt{n_1}} \tau_{n_1}, \dots, \frac{1}{\sqrt{n_d}} \tau_{n_d}\right)
\]
where we recall that we set $\tau_x \coloneqq \tau_{\lfloor x \rfloor }$.
As we see on the right-hand side of \Cref{fig:intersection_graph}, scaling each $\tau_{n_j}$ by $\frac{1}{\sqrt{n_j}}$ allows us to isolate the $k^{th}$ term by intersecting with $\tau_{n_k}$.

We then formally define $f \coloneqq f_R^d : [0,R]^d \to \PCf(S)$ by
\begin{equation}
\begin{array}{rcl}
f(\vec x)  & \coloneqq & \rho + e^{\vec x} \cdot \vec \mu  \\[1em]
& = & \rho + e^{4dR}\sum_{j=1}^d  \frac{e^{x_j}}{\sqrt{n_j}}  \tau_{n_j}.
\end{array}
\label{eq:candidate_current}
\end{equation}

\begin{figure}
\centering 
\includegraphics[scale=.25]{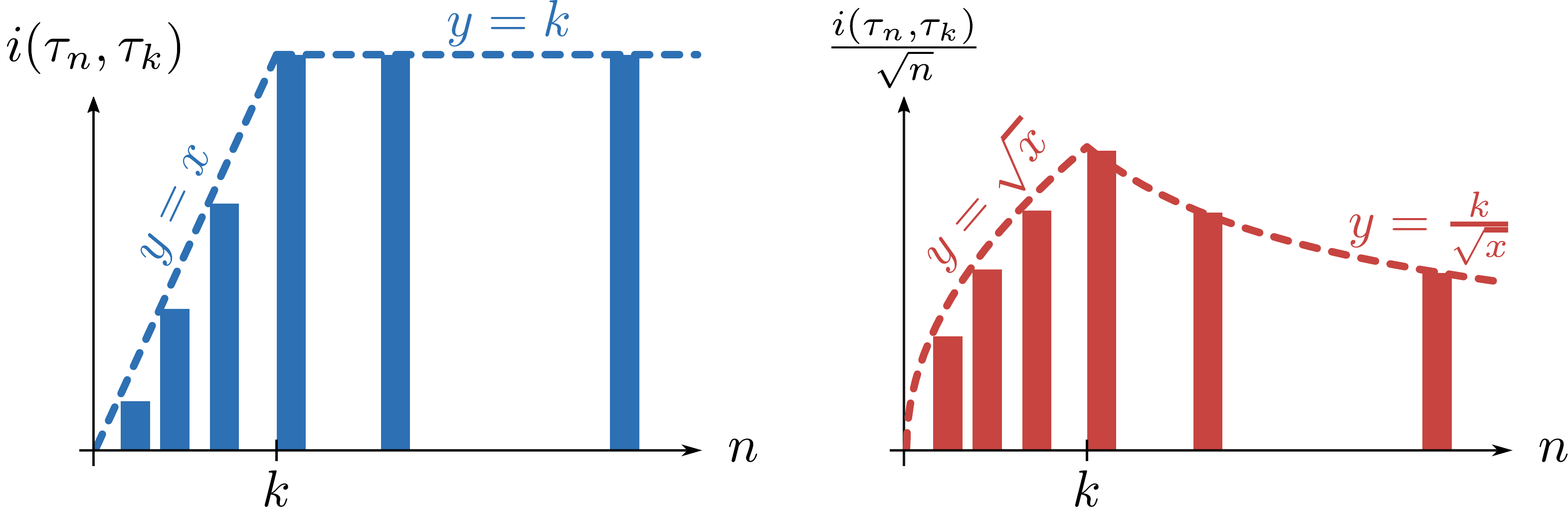}
\caption{The effect of intersecting $\tau_n$ with $\tau_k$, as a function of $n$. On the left, we give the unnormalized intersection function. On the right, we see that scaling each $\tau_n$ by $\frac{1}{\sqrt n}$ allows us to pick out the $k$-th term. }
\label{fig:intersection_graph}
\end{figure}
We wish to estimate $\Dil(f(\vec x), f(\vec y))$. First, we note that
\begin{equation}
\label{eq:Dil_at_least_1}
    \Dil(f(\vec x), f(\vec y)) \geq 1
\end{equation}
for all $\vec x, \vec y \in [0,R]^d$. Indeed, recall that $P$ was the pair of pants that contains $\tau_n$ for all $n \in \mathbb N$. Let $\alpha$ be any closed curve that does not intersect the interior of $P$ (for example, a boundary curve of $P$). Then $i(f(\vec x), \alpha) = i(\rho, \alpha)$ for all $\vec x$, giving us the desired lower bound.

Now, we show the following lower bound in terms of $\vec x$ and $\vec y$:
\[
d_{Lip}(f(\vec x), f(\vec y)) \geq ||\vec x - \vec y||_\infty  -2\log(5).
\]
In fact, choose any $k$ with $1 \leq k \leq d$. We will show that for all $\vec x \in [0,R]^d$,
\[
e^{x_k} e^{(4d + 3k)R} \leq  i(f(\vec x), \tau_{n_k}) \leq 5e^{x_k} e^{(4d + 3k)R}.
\]
 This will allow us to bound the log of the ratio of such terms from below.

We now compute:
\begin{align*}
    i(f(\vec x), \tau_{n_k}) & = i(\rho + e^{\vec x} \cdot \vec \mu, \tau_{n_k}) \\
    & = i(\rho, \tau_{n_k}) + e^{4dR} \sum_{j=1}^d \frac{e^{x_j}}{\sqrt{n_j}}i(\tau_{n_j}, \tau_{n_k}).
\end{align*}
A priori, $n_j, n_k \not \in \N$. So we first establish an estimate on $i(\tau_{n_j}, \tau_{n_k})$ as follows. Using that $\tau_x \coloneqq \tau_{\lfloor x \rfloor}$ and \Cref{eq:tau_pattern}, we have
\[
i(\tau_{n_j}, \tau_{n_k}) = 2\min \{\lfloor n_j\rfloor,\lfloor n_k\rfloor\}.
\]
For all $x \geq 1$, we have $\frac 12 x \leq \lfloor x \rfloor \leq x$, and so this gives
\[
\min\{n_j, n_k\} \leq i(\tau_{n_j}, \tau_{n_k}) \leq 2\min\{n_j, n_k\}.
\]

To bound $i(f(\vec x), \tau_{n_k})$ from below, we only need to consider the $k^{th}$ term of the sum. This gives
\begin{align*}
     i(f(\vec x), \tau_{n_k}) & \geq e^{4dR} e^{x_k} \frac{i(\tau_{n_k},\tau_{n_k})}{\sqrt{n_k}} \\
     & \geq e^{4dR} e^{x_k} \frac{n_k}{\sqrt{n_k}}\\
     & = e^{4dR} e^{x_k} \frac{e^{6kR}}{e^{3kR}}.
\end{align*}
Therefore, we have
\begin{equation}
\label{eq:int_tau_k_lower_bound}
    i(f(\vec x), \tau_{n_k}) \geq  e^{x_k}e^{(4d + 3k)R}.
\end{equation}
Now we proceed to bound $i(f(\vec x), \tau_{n_k})$ from above. Note that, since $i(\rho, \tau_n) < n$ for all $n \in \N$, we also have $i(\rho, \tau_x) < x$ for all $x \in \R_{\geq 1}$. So, continuing the estimate on $i(f(\vec x), \tau_{n_k})$ above, we have
\begin{align*}
  i(f(\vec x), \tau_{n_k}) & = i(\rho, \tau_{n_k}) + e^{4dR} \sum_{j=1}^d \frac{e^{x_j}}{\sqrt{n_j}}i(\tau_{n_j}, \tau_{n_k})\\
  & \leq n_k + 2 e^{4dR}\sum_{j=1}^d \frac{e^{x_j}}{\sqrt{n_j}} \min\{n_j,n_k\}  \\
    & = n_k + 2e^{4dR}\left (\sum_{j=1}^{k-1} e^{x_j} \sqrt{n_j} +  \sum_{j=k}^{d} \frac{e^{x_j}}{\sqrt{n_j}} n_k \right ).
\end{align*}

We will bound each term separately. First, using the fact that $k \leq d$ and $x_k \geq 0$, we have
\begin{align}
     \label{eq:bound_n_k}
    n_k & = e^{6kR} \notag \\
     & \leq e^{x_k}e^{(4d+3k)R}.
\end{align}

Next, consider the terms for $j \leq k-1$. We have $\sqrt{n_j} = e^{3jR} \leq e^{(3k-3)R}$. Moreover, $x_j, x_k \in [0,R]$, so $x_j \leq x_k + R$. Thus,
\begin{align}
\label{eq:bound_j<k}
    2e^{4dR} e^{x_j} \sqrt{n_j} & \leq 2e^{x_j} e^{(4d+3k -3)R} \notag \\
    & = 2e^{-3R}\cdot e^{x_j}e^{(4d+3k)R} \notag \\
    & \leq 2e^{-2R}\cdot e^{x_k}e^{(4d+3k)R}.
\end{align}
Now, we look at the terms for $j \geq k$. We have $\frac{n_k}{\sqrt n_j} = e^{(6k-3j)R}$. When $j =k $, this gives
\begin{equation}
\label{eq:bound_j=k}
    2e^{4dR}e^{x_k}\frac{n_k}{\sqrt{n_k}} = 2e^{x_k}e^{(4d+3k)R}
\end{equation}
and when $j \geq k+1$, it gives
\begin{align}
\label{eq:bound_j>k}
    2e^{4dR}e^{x_j}\frac{n_k}{\sqrt{n_j}} & \leq 2e^{x_j}e^{(4d + 3k-3)R} \notag \\
    & = 2e^{-3R} \cdot e^{x_j}e^{(4d+3k)R}\notag \\
    & \leq 2e^{-2R}\cdot e^{x_k}e^{(4d+3k)R}.
\end{align}

Inequalities~\eqref{eq:bound_n_k} and~\eqref{eq:bound_j=k} give two terms bounded above in terms of $e^{x_k}e^{(4d+3k)R}$, and inequalities~\eqref{eq:bound_j<k} and~\eqref{eq:bound_j>k} give $d-1$ terms bounded above by $2e^{-2R}\cdot e^{x_k}e^{(4d+3k)R}$. Combining these gives,
\begin{align*}
    i(f(\vec x), \tau_{n_k}) & \leq 3e^{x_k}e^{(4d+3k)R} + 2(d-1)e^{-2R}e^{x_k}e^{(4d+3k)R}.
\end{align*}
We chose $R > \log d$, so in particular, we have $(d-1)e^{-2R }< 1$. Thus,
\begin{equation}
\label{eq:int_tau_k_upper_bound}
i(f(\vec x), \tau_{n_k})  \leq 5 e^{x_k}e^{(4d+3k)R}.
\end{equation}

Combining~\eqref{eq:int_tau_k_lower_bound} with~\eqref{eq:int_tau_k_upper_bound} gives, for all $k$ with $1 \leq k \leq d$,
\[
\Dil(f(\vec y), f(\vec x)) \geq \max\left\{\frac 15 e^{y_k - x_k}, 1\right\}
\]
where the lower bound of 1 was established in~\eqref{eq:Dil_at_least_1}.

Thus,
\begin{align*}
d_{Lip}(f(\vec x), f(\vec y))&= \log \Dil(f(\vec y), f(\vec x)) + \log \Dil(f(\vec x), f(\vec y)) \\
& \geq \max_k\{y_k - x_k, 0\} + \max\{ y_k - x_k,0\} - 2\log(5)\\
& \geq ||\vec x - \vec y||_\infty - 2\log(5).
\end{align*}

Lastly, we need to bound $d_{Lip}(f(\vec x), f(\vec y))$ from above. So let $\gamma$ be any closed curve. Let $\vec x, \vec y \in [0,R]^d$. Then, using the fact that $x_j \leq ||\vec x - \vec y||_\infty +  y_j$ for all $j$, we get
\begin{align*}
    i(f(\vec x), \gamma) & = i(\rho, \gamma) + e^{4dR}\sum_{j=1} ^d e^{x_j} i(\tau_{n_j}, \gamma) \\
    & \leq  i(\rho, \gamma) + e^{4dR}\sum_{j=1} ^d e^{||x_j - y_j||}e^{y_j} i(\tau_{n_j}, \gamma) \\
    & \leq e^{||\vec x - \vec y||_\infty} \left(
i(\rho, \gamma) + e^{4dR}\sum_{j=1} ^d e^{y_j} i(\tau_{n_j}, \gamma)
    \right)\\
    & = e^{||\vec x - \vec y||_\infty} i(f(\vec y), \gamma).
\end{align*}
Therefore,
\[
\Dil(f(\vec x), f(\vec y)) \leq e^{||\vec x - \vec y||_\infty}.
\]
Thus,
\begin{align*}
    d_{Lip}(f(\vec x), f(\vec y)) & = \log \Dil(f(\vec y), f(\vec x)) + \log \Dil(f(\vec x), f(\vec y)) \\
    & \leq 2||\vec x - \vec y||_\infty.
\end{align*}
In particular, for all $d \in \N$ and all $R > \log d$, the function $f_R^d=f: [0,R]^d \to \PCf(S)$ is a $(2,2\log(5))$-quasi-isometry. Therefore, the geometric rank of $\PCf(S)$ is infinity.
\end{proof}

To prove infinite geometric rank above, we gave families of maps $f_R^d: [0,R]^d \to \PC_f(S)$ whose images are contained in the filling simplex
\[ \PSimplex(\rho, \tau_{n_1},\cdots,\tau_{n_d})
\]
for the closed curves $\tau_{n_1}, \dots, \tau_{n_d}$ defined above.
We prove that this simplex is always quasi-isometric to a line. Thus, the method of the above proof cannot prove that \textit{flat rank} of $\PCf(S)$ is infinite, since it works with simplices of finite (in fact, uniformly bounded) rank.

\begin{lem}
\label{lem:tau_simplex_is_line}
Let $\rho$ be a filling geodesic current in $S$, and let $P \subset S$ be an embedded pair of pants, with standard fundamental group generators $a,b$, and associated $\tau_n=[a^nb]$ curves. Then, for any sequence $(n_i)_{i=1,\cdots,m}$, the filling simplex $\PSimplex(\rho, \tau_{n_1},\cdots,\tau_{n_m})$ is quasi-isometric to $\mathbb{R}$.
\end{lem}
\begin{proof}
Let $\PSimplex \coloneqq \PSimplex (\rho, \tau_{n_1},\cdots,\tau_{n_m})$ and $\lambda \coloneqq \sum_{j=1}^m \tau_{n_j}$.
Since each $\tau_{n_j}$ fills $P$ and there are finitely many of them in the simplex, by \Cref{lem:bounded_sup}, there exists a constant $C>0$, so that, for every $j$,
\begin{equation}
\label{eq:C_lipschitz}
1/C i(\lambda,\gamma) \leq i(\tau_{n_j},\gamma) \leq C i(\lambda, \gamma)
\end{equation}
for every closed curve $\gamma$.
For every 
\[
\mu = t \rho + \sum_{j=1}^m a_j \tau_{n_j} \in \PSimplex
\]
we let $A \coloneqq \sum_{j=1}^m a_j$. Then, by~Eq.~\eqref{eq:C_lipschitz}, we get
\[
ti(\rho,\gamma) + \frac{1}{C}Ai(\lambda,\gamma) \leq i(\mu,\gamma) \leq t i(\rho,\gamma) + CAi(\lambda,\gamma)
\]
and thus
\[
\dsym(\mu, t\rho + A \lambda)<2\log(C).
\]
Hence, every $\mu \in \PSimplex$ is at uniform finite distance (independent of $\mu$) from the simplex $\PSimplex(\rho, \lambda)$, and since $\PSimplex(\rho, \lambda) \subset \PSimplex$, it follows that that $\PSimplex$ is quasi-isometric to $\PSimplex(\rho, \lambda)$.
But since $\lambda$ fills $P$  while $\rho$ fills all of $S$, it follows that $(\lambda, \rho)$ is an increasing sequence (By \Cref{lem:bounded_sup}, $\Dil(\rho, \lambda) < \infty$, and the fact that $P$ is properly contained in $S$ implies that the other dilatation is infinite.) Moreover, by \Cref{lem:increasing_implies_strictly}, the fact that $\rho$ is filling, and so has no lamination components gives that $(\lambda, d \rho)$ is strictly increasing for some $d > 0$. As $\PSimplex(\rho, \lambda) = \PSimplex(d \rho, \lambda)$,
\Cref{lem:qi_to_Rn} implies that $\PSimplex$ is quasi-isometric to $\mathbb{R}$.
\end{proof}

\section{Rank of Teichm\"uller space with the Lipschitz metric}
\label{app:dsym-rank-minimal}

In this section we will prove Theorem~\ref{mainthm:teichrank}, in the slightly more general setting of finite-type possibly punctured surfaces. Let $S=S_{g,n}$ be a connected orientable surface of finite type, without
boundary, genus $g$, $n$ punctures, and with \textbf{Euler characteristic} $\chi(S)=2-2g-n<0$.  Denote the \textbf{complexity of $S$} by
\[
 \xi(S)=3g-3+n
\]
and assume $\xi(S)>0$, in which case it coincides with the number of curves of a pair of pants decomposition of $S$.

Theorem~\ref{mainthm:teichrank} follows from Theorem~\ref{thm:min-rank} below with $n=0$, since in this case $\xi(S)=3g-3$.

Write $d_T$ for the Teichm\"uller metric.  For $X,Y\in\Teich(S)$, we will use the following version of the symmetrized Thurston metric
\[
 d_L(X,Y)\coloneqq\max\{d_{\Th}(X,Y),d_{\Th}(Y,X)\}.
\]
This version is used, for instance, in~\cite{CR07:Comparison}. Clearly, $d_L\le \dsym\le 2d_L$.

The main result of this section is the computation of the geometric and flat rank of $\Teich(S)$ with respect to $\dsym$ or, equivalently since they are bilipschitz equivalent, to $d_L$. 

\begin{theorem}
\label{thm:min-rank}
Let $g \geq 1, n \geq 0$. Let $S = S_{g,n}$ be a connected, orientable surface of genus $g$ and $n$ punctures, with negative Euler characteristic. Then,
\[
\rank\bigl(\Teich(S),\dsym\bigr)=\rankF\bigl(\Teich(S),\dsym\bigr)=\xi(S).
\]
\end{theorem}

\subsection{Proof idea}

The proof uses Minsky's~\cite{Minsky1996ProductRegions} and Choi-Rafi's~\cite{CR07:Comparison} thick-thin decomposition theorems to express $d_L$ and $d_T$, up to uniform additive error, as a maximum of metrics on components of a thick-thin decomposition,
in such a way that the only difference between the two metrics appears in the annular terms. In fact, at the asymptotic cone level, the difference in the annular terms emanates only from core curves with collar widths growing 
super-linearly with respect to the scale, which we shall call ``deep''. We call the complementary components ``shallow'' (i.e., components that are thick, as well as thin components whose collar
widths do not grow super-linearly with the scale).  
  Product coordinates then separate the deep annular factors from shallow complementary factors, and $d_L$-- and
$d_T$--cones are bilipschitz in the latter. Work of Bowditch~\cite[Theorem~1.8]{Bow16}, or, alternatively, Durham ~\cite[Theorem~1.1]{Durham} bounds the dimension of compact subsets of the $d_T$-asymptotic cone. We use this to bound the dimension of compact subsets of the $d_L$ cone. 
Finally, growing quasi-cubes pass to an asymptotic $d_L$-cone, and give an upper bound of $\xi(S)$ to the geometric rank. The lower bound of $\xi(S)$ on the flat rank follows easily from~\cite{CR07:Comparison}, giving the result.

\subsection{Product-region formulas}
\label{subsec:min-inputs}
Let $c_0> 0$ denote the Margulis constant defined as follows. We can decompose any hyperbolic surface $X$ into a $c_0$-\textit{thin part}, where each point has injectivity radius at most $c_0$, and its complementary \textit{thick part}. The Margulis constant $c_0$ satisfies the property that the thin part is a disjoint union of collars about simple closed geodesics, and neighborhoods of the cusps (see, eg,~\cite[Section~2.4]{Minsky1996ProductRegions}). Then there is an $\epsilon_0$, called the \textbf{collar cutoff}, which depends on the choice of $c_0$, so that the boundaries of each collar and cusp neighborhood have length at most $\epsilon_0$. 
Lastly, let $\epsilon_1 > 0$ be small enough so that the product region theorems~\cite{Minsky1996ProductRegions,CR07:Comparison} hold.  We call $\epsilon_1$ the \textbf{product region cutoff}.

Note that these constants can be chosen independently of $X \in \Teich(S)$. In particular, consider the annular components of the thin part of $X$. Let $\Gamma$ be the collection of their core curves. Then there are finitely many topological types of multicurves $\Gamma$, as $\Gamma$ ranges over all topological types of $k$-components simple multi-curves, for $1 \leq k \leq \xi(S)$ \cite[Lemma~2.16]{ES22:GeodesicCount}. 

As we can always choose our constants to be smaller, we fix $0 < \epsilon_1 < c_0 < \epsilon_0$ so that 
\[\frac{\epsilon_0}{\epsilon_1}>2.\]
Furthermore, fix $\epsilon>0$ with \[\epsilon<\min \{e^{-1}, \epsilon_1\}.\]

We will refer to $\epsilon$ as simply the \textbf{cutoff}, since it is the smallest cutoff constant.
For a multicurve $\Gamma$,
put
\[
 \Thin_\eps(S,\Gamma)
 =\{X:\ell_X(\gamma)<\eps\text{ for every }\gamma\in\Gamma\}.
\]
Let
\[
 U_\eps\coloneqq\{(t,y)\in\R\times(0,\infty):y>1/\eps\},
\]
Following \cite{Minsky1996ProductRegions,CR07:Comparison}, we define a bijective map
\[
\Pi_\Gamma : \Thin_\eps(S,\Gamma) \to \Teich(S \setminus \Gamma) \times \prod_{\gamma \in \Gamma} U_\epsilon
\]
where for each $X \in \Thin_\eps(S,\Gamma)$, we say the \textbf{product coordinates} of $X$ are
\begin{equation}
 \Pi_\Gamma(X)
 =\bigl(X_0,(z_\gamma(X))_{\gamma\in\Gamma}\bigr),
 \qquad
 z_\gamma(X)=\left(t_\gamma(X),\frac1{\ell_X(\gamma)}\right).
 \label{eq:min-product-coordinates}
\end{equation}
Here $X_0\in\Teich(S\setminus\Gamma)$ denotes the pinched surface obtained by setting boundary length Fenchel-Nielsen coordinates to zero and retaining all the other coordinates.
For a disconnected complement it is a tuple, equipped with the
sup metric.  For each $\gamma \in \Gamma$, the \textbf{twist} $t_\gamma(X)$, is defined as the signed relative $X$-displacement along $\gamma$ between the seams of the two pants glued across it, normalized by the \textbf{length} $\ell_\gamma(X)$, ie, the $X$-hyperbolic length of $\gamma$, so that one positive Dehn
twist adds $1$ to $t_\gamma$ \cite[\S3.4]{Minsky1996ProductRegions}.

For each $\gamma \in \Gamma$, Choi-Rafi define a modified annular metric $d_{L(\gamma)}$ on $U_\epsilon$ \cite[Section~3.3]{CR07:Comparison}. Its definition is quite technical, so rather than state it formally, we will cite its properties as proven in \cite{CR07:Comparison}.

 Choi--Rafi's product theorem
\cite[Theorem~3.5]{CR07:Comparison} gives
\begin{equation}
 d_L(X,Y)
 =\max\left\{d_L(X_0,Y_0),
       \max_{\gamma\in\Gamma}
       d_{L(\gamma)}(z_\gamma(X),z_\gamma(Y))\right\}+O(1).
 \label{eq:min-CR-product}
\end{equation}
where all error terms here and below are bounded uniformly in terms of $\epsilon$ and the topology of $S$.

In comparison, Minsky's thick-thin decomposition theorem
\cite[Theorem~6.1 and equation~(2.5)]{Minsky1996ProductRegions} gives
\begin{equation}
 d_T(X,Y)
 =\max\left\{d_T(X_0,Y_0),
       \max_{\gamma\in\Gamma}
       \frac12d_{\Hh}^{\mathrm{std}}
       (z_\gamma(X),z_\gamma(Y))\right\}+O(1).
 \label{eq:min-minsky}
\end{equation}
where $d_{\Hh}^{\mathrm{std}}$ denotes the standard hyperbolic metric on the upper-half plane model.

We also use three consequences recorded in the same sources.  First, for an
essential nonperipheral curve $\beta$ in the thick part of $X$,
\begin{equation}
 K_0^{-1}\ell_X(\beta)
 \le\ell_{X_0}(\beta)
 \le K_0\ell_X(\beta),
 \label{eq:min-pinching}
\end{equation}
where $K_0>0$ depends on the topology of $S$ and the fixed cutoff $\epsilon$.
This is the length comparison in the proof of
\cite[Theorem~3.5, p.~754]{CR07:Comparison}; see also the construction in
\cite[proof of Theorem~6.1, pp.~280--281]{Minsky1996ProductRegions}.  Second, if $X, Y \in \Teich(S)$
have no $\epsilon$-short curves in common, then
\begin{equation}
 d_T(X,Y)\le A_0d_L(X,Y)+B_0
 \label{eq:min-no-common}
\end{equation}
by \cite[Proposition~4.1]{CR07:Comparison}.  
Again, $A_0, B_0$ are uniform, depending only on the topology of $S$ and the cutoff $\epsilon$.
Finally, the finite-type form of Wolpert's
length distortion inequality, recorded in \cite[equation~(3)]{CR07:Comparison}, gives
\begin{equation}
 d_L(X,Y)\le2d_T(X,Y).
 \label{eq:min-wolpert}
\end{equation}

The following is~\cite[Lemma~3.4]{CR07:Comparison}.

\begin{lem}
Let $\gamma$ be an $\epsilon$-short curve. For $j = 1,2$, let
\[
 z_j=(t_j,1/\ell_j)\in U_\eps,
\]
and suppose $\ell_1\le\ell_2$. 
Then,
\begin{enumerate}
    \item If
    \[
    |t_1-t_2|\ell_1 \leq \log \frac1{\ell_1}, \quad \text{ then }\quad
    d_{L(\gamma)}(z_1,z_2) \asymp_+ \log \frac{\ell_2}{\ell_1}.
    \]
    \item If 
    \begin{align*}
    |t_1-t_2|\ell_1> \log \frac1{\ell_1},\quad \text{ then }\quad
    d_{L(\gamma)}(z_1,z_2) &\asymp_+ \log \left( \frac{|t_1-t_2|\ell_2}{\log(1/\ell_1)}\right) \\
    &=\log \frac{\ell_2}{\ell_1} + \log \left(\frac{|t_1-t_2|\ell_1}{\log(1/\ell_1)}\right),
    \end{align*}
\end{enumerate}
where $\asymp_+ $ denotes equality up to additive error depending on the topology of $S$ and the cutoff $\epsilon$.
\label{lemma:CR}
\end{lem}

 Thus, for $d_{L(\gamma)}$ twisting only becomes significant when it exceeds the scale of $\log 1/\ell_1$ of the collar width. 

Next, we give the following estimate on hyperbolic distance in horoballs. In what follows, define $\log^+ : \R_{\geq 0} \to \R$ by
\[
\log^+(x) = \left \{\begin{matrix}
\max\{0,\log x\} & \text{ if } & x > 0\\
0 &\text{ if } &  x = 0.\\
\end{matrix} \right.
\]

\begin{lem}
\label{lem:min-annulus-pre}
For $j = 1,2$, let
\[
 z_j=(t_j,1/\ell_j)\in U_\eps,
\]
and suppose $\ell_1\le\ell_2$. 
Suppose $\ell_1\leq \ell_2$.
Then,
\begin{equation}
d_{\Hh}^{\mathrm{std}}(z_1,z_2)=\log \frac{\ell_2}{\ell_1} +2\log^+ \frac{|t_1-t_2|}{\ell_1} + O(1),
\label{eq:min-H-coarse}
\end{equation}
where the $O(1)$ error depends only on the topology of $S$ and the cutoff $\epsilon$.
\end{lem}
\begin{proof}

Set
\[
d \coloneqq d_{\Hh}^{\mathrm{std}}(z_1,z_2).
\]
The standard upper-half plane distance formula gives
\begin{equation}
\cosh(d)=1+\frac{(t_1-t_2)^2 + (1/\ell_1-1/\ell_2)^2}{2/(\ell_1\ell_2)}
\label{eq:min-H-exact-first}
\end{equation}
Set
\[
r=\frac{\ell_2}{\ell_1} \geq 1,
\qquad
s=\frac{|t_1-t_2|}{\ell_1}.
\]
Expanding the numerator in~\eqref{eq:min-H-exact-first} and rearranging terms we get
\begin{equation}
    2\cosh(d)=r+r^{-1}+rs^2 = r\left(1+r^{-2}+s^2\right).
    \label{eq:min-H-rs}
\end{equation}
Since $d \geq 0$, 
\begin{equation}
d=\log(2\cosh d)+O(1)
\label{eq:min-arcosh}
\end{equation}
with an additive error bounded by $\log(2)$.
Combining~\eqref{eq:min-H-rs} and~\eqref{eq:min-arcosh}, 
\begin{equation}
d=\log \frac{\ell_2}{\ell_1}+\log\left (1+r^{-2}+s^2\right)+O(1).
\label{eq:min-H-intermediate}
\end{equation}
It remains to estimate the second logarithm.
Because $r \geq 1$, we have $0<r^{-2}\leq 1$.
If $0 \leq s \leq 1$, we have
\[
1 \leq 1+r^{-2}+s^2 \leq 3,
\]
and so
\[
\log \left(1+r^{-2}+s^2\right)=O(1)
\]
If $s \geq 1$, then
\[
s^2\leq 1+r^{-2}+s^2 \leq 3s^2,
\]
so
\[
2\log s \leq \log \left (1+r^{-2}+s^2\right)\leq 2\log s + \log 3
\]
Since $\log^+ s= \log s$ in this case, we obtain
\[
\log \left(1+r^2{-2}+s^2\right)=2\log^+ s + O(1)
\]
Substitution into~\eqref{eq:min-H-intermediate} gives the desired uniform estimate.
\end{proof}

The above results allow us to compare $d_{L(\gamma)}$ and $d_{\Hh}^{\mathrm{std}}$ as follows.
\begin{lem}
\label{lem:min-annulus}
For $j = 1,2$, let
\[
 z_j=(t_j,1/\ell_j)\in U_\eps,
\]
and suppose $\ell_1\le\ell_2$. 
One has
\begin{align}
 d_{L(\gamma)}(z_1,z_2)
 &=\log \frac{\ell_2}{\ell_1}+\log^+\left(\frac{|t_1-t_2|\ell_1}{\log\frac1{\ell_1}}\right)+O(1),
 \label{eq:min-ann-L}
\end{align}
Consequently,
\begin{equation}
 \frac12d_{\Hh}^{\mathrm{std}}(z_1,z_2)
 \le d_{L(\gamma)}(z_1,z_2)+\log \log\frac1{\ell_1}+O(1),
 \label{eq:min-ann-error}
\end{equation}
where the  $O(1)$ error only on the topology of $S$ and on the cutoff $\epsilon$.
\end{lem}
Here $d_{\Hh}^{\mathrm{std}}$ denotes the standard hyperbolic metric on the upper half plane model.
\begin{proof}
Since $\ell_1 \leq \ell_2$, we have
\[
\log \frac{\ell_2}{\ell_1}\geq 0.
\]
The two cases of \Cref{lemma:CR}, combine into combine to give
\[
d_{L(\gamma)}(z_1,z_2) = \log \frac{\ell_2}{\ell_1}+ \log^+\left(\frac{|t_1-t_2|\ell_1}{\log\frac1{\ell_1}}\right)+O(1)
\]
which is~\eqref{eq:min-ann-L}.

Now, since $\epsilon > e^{-1}$, we have $\log\frac1{\ell_1} \geq 1$. Thus, we claim
\begin{equation}
\log^+ |t_1-t_2|\ell_1 \leq \log^+ \left( \frac{|t_1-t_2|\ell_1}{\log\frac1{\ell_1}}\right)+\log \log\frac1{\ell_1}.
\label{eq:min-log-inequality}
\end{equation}
Indeed:
\begin{itemize}
\item If $0 \leq |t_1-t_2|\ell_1 \leq 1$, then the left hand side is zero.
\item If $1 <|t_1-t_2|\ell_1 \leq \log\frac1{\ell_1}$ then
\[
\log^+ |t_1-t_2|\ell_1 = \log |t_1-t_2|\ell_1 \leq \log \log\frac1{\ell_1}, \text{ while }\log^+\left(|t_1-t_2|\ell_1/\log\frac1{\ell_1}\right)=0.
\]
\item If $|t_1-t_2|\ell_1>\log\frac1{\ell_1}$ then
\[
\log^+ |t_1-t_2|\ell_1 = \log |t_1-t_2|\ell_1 = \log \left(\frac{|t_1-t_2|\ell_1}{\log\frac1{\ell_1}}\right) + \log \log\frac1{\ell_1}.
\]
\end{itemize}

Thence
\[
\frac 12 \log \frac{\ell_2}{\ell_1}+ \log^+ |t_1-t_2|\ell_1 \leq \log \frac{\ell_2}{\ell_1}+ \log^+ \left(\frac{|t_1-t_2|\ell_1}{\log\frac1{\ell_1}}\right)+\log \log\frac1{\ell_1}.
\]
Using~\eqref{eq:min-H-coarse} and~\eqref{eq:min-ann-L}, and absorbing their uniform additive errors into a single constant $C$, we get
\[
\frac12 d_{\Hh}^{\mathrm{std}}(z_1,z_2) \leq d_{L(\gamma)}(z_1,z_2)+\log \log\frac1{\ell_1} + C
\]
as desired.
\end{proof}

Now we can characterize the asymptotic cones of the horoballs for the Lipschitz metric.
\begin{lem}
\label{lem:min-annular-cones}
For each $\gamma \in \Gamma$, the map
\[
 F(t,y)=(t,y\log y)
\]
is a quasi-isometry from $(U_\eps,d_{L(\gamma)})$ onto a \emph{hyperbolic
horoball}: \[
 F(U_\epsilon) = \left\{(t,y) \in \Hh^2 \colon y \geq \frac{ \log(1/\epsilon)}{\epsilon} \right\},
 \]
Consequently, every asymptotic cone of $(U_\eps,d_{L(\gamma)})$ is bilipschitz equivalent to an $\R$--tree.
\end{lem}

\begin{proof}
Note that $F$ maps $U_\epsilon$ bijectively onto a
hyperbolic horoball, so in particular, its image is coarsely dense.  In fact, the function $\ell \longmapsto \frac{\log(1/\ell)}{\ell}$ maps $(0,\epsilon]$ bijectively onto  $\left[ 
 \frac{\log(1/\epsilon)}{\epsilon},\infty\right).$
 Thus
 \[
 F(U_\epsilon) = \left\{(t,y) \in \Hh^2 \colon y \geq \frac{ \log(1/\epsilon)}{\epsilon} \right\}.
 \]

So all we need to show is that $F$ is a quasi-isometric embedding. Pick $z_1, z_2 \in U_\epsilon$, so that $z_i = (t_i, 1/\ell_i)$ for $i = 1,2$. Assume without loss of generality that $\ell_1 \leq \ell_2$. Then by \Cref{lem:min-annulus-pre}, we have
\begin{equation}
d_{\Hh}^{\mathrm{std}}(z_1,z_2)=\log \frac{\ell_2}{\ell_1} +2\log^+ \frac{|t_1-t_2|}{\ell_1} + O(1).
\label{eq:min-H-coarse-pre}
\end{equation}
Thus, 
\begin{equation}
\begin{aligned}
\frac12d_{\Hh}^{\mathrm{std}}(F(z_1),F(z_2))& 
 ={}\frac12d_{\Hh}^{\mathrm{std}}(F(t_1,1/\ell_1),F(t_2,1/\ell_2))\\
 &={}
 \frac12 \log \left( \frac{\log(1/\ell_1)/\ell_1}{\log(1/\ell_2)/\ell_2}\right)
 +
 \log^+ \left( \frac{|t_1-t_2|\ell_1}{\log(1/\ell_1)}\right)+O(1).
    \label{eq:min-F-formula}
 \end{aligned}
 \end{equation}
 The first term can be expanded as
 \[
\log\left(\frac{\log(1/\ell_1)/\ell_1}{\log(1/\ell_2)/\ell_2}\right)
 = \log \frac{\ell_2}{\ell_1} + \log \left( \frac{\log (1/\ell_1)}{\log(1/\ell_2)}\right).
 \]
 So
 \begin{equation*}
 \begin{aligned}
 \frac12 d_{\Hh}^{\mathrm{std}} \left(F(t_1,1/\ell_1), F(t_2,1/\ell_2)\right) = \\
 & = \frac12 \log \frac{\ell_2}{\ell_1}
 +
 \frac12 \log \left(
 \frac{\log(1/\ell_1)}{\log(1/\ell_2)}\right)
 +
 \log^+ \left(\frac{|t_1-t_2|\ell_1}{\log(1/\ell_1)}\right)+O(1).
 \end{aligned}
 \end{equation*}
 Finally, we estimate the additional logarithmic term.
 Since
 \[
\log\frac1{\ell_1}=\log\frac1{\ell_2}+\log\frac{\ell_2}{\ell_1},
 \]
 we have
 \begin{align*}
 \log \left( \frac{\log(1/\ell_1)}{\log(1/\ell_2)}\right)
 &=
 \log \left(1+\frac{\log(\ell_2/\ell_1)}{\log(1/\ell_2)}\right).
 \end{align*}
 Because $\ell_2\leq \epsilon < e^{-1}$, we have
 $\log(1/\ell_2)\geq 1$. Hence
 \[
 0 \leq \log \left(\frac{\log(1/\ell_1)}{\log(1/\ell_2)}\right)
 \leq \log \frac{\ell_2}{\ell_1}.
 \]
 Here we used $\log(1+x) \leq x$ for $x \geq 0$.

 Comparing this formula with~\eqref{eq:min-ann-L} in \Cref{lem:min-annulus}, and using the non-negativity of the $\log \left(\frac{\log \dots}{\log \dots}\right)$ term we obtain a uniform constant $C>0$ so that
 \begin{align}
 \frac12 d_{L(\gamma)}\left((t_1,1/\ell_1),(t_2,1/\ell_2)\right)-C & \leq \frac12 d_{\Hh}^{\mathrm{std}}\left(F(t_1,1/\ell_1),F(t_2,1/\ell_2) \right) \\
 &\leq  d_{L(\gamma)}\left((t_1,1/\ell_1),(t_2,1/\ell_2)\right)+C.
\label{eq:inequalities}
 \end{align}
 Therefore  $F$ is a quasi-isometry onto a hyperbolic horoball.

 After rescaling to form an asymptotic cone, the additive constant $C$ disappears, so $F$ induces a uniform bilipschitz equivalence of asymptotic cones. A hyperbolic horoball is convex in $\Hh$, and its asymptotic cones are $\mathbb R$-trees (see, e.g., ~\cite[Proposition~9.9]{Kap00:Kapovich2000HyperbolicMA}). Therefore every asymptotic cone of the modified annular factor is bilipschitz equivalent to an $\mathbb R$-tree, as desired.
\end{proof}

For $X,Y\in\Teich(S)$, define
\begin{equation}
 \Gamma_\eps(X,Y)
 =\{\alpha:\ell_X(\alpha)<\eps\text{ and }\ell_Y(\alpha)<\eps\}
\label{eq:simultaneously-short}
\end{equation}
and
\begin{equation}
 P_\eps(X,Y)=\max\left(
 \{2\}\cup
 \left\{
  \log\frac1{\ell_X(\alpha)},
  \log\frac1{\ell_Y(\alpha)}:
  \alpha\in\Gamma_\eps(X,Y)
 \right\}\right).
 \label{eq:min-P}
\end{equation}
Roughly, $P_\epsilon(X,Y)$ is the depth in the thin part over curves that are mutually short for both $X$ and $Y$. 
We then have the following global comparison between the two metrics.

\begin{lem}
\label{lem:min-global}
There are constants $A,B$ depending only on the topology of $S$ and the cutoff constant $\epsilon$ such that
\begin{equation}
 d_T(X,Y)
 \le A d_L(X,Y)+B\log\bigl(P_\eps(X,Y)\bigr)+B
 \label{eq:min-global}
\end{equation}
for all $X,Y\in\Teich(S)$.
\end{lem}

\begin{proof}
Let $\Gamma=\Gamma_\eps(X,Y)$. Let $K_0$ be the length comparison constant from Eq.~\eqref{eq:min-pinching}. Choose a second
cutoff $\eta$ so small that
\begin{equation}
 K_0\eta<\eps,
 \label{eq:min-nested}
\end{equation} 
Cut $X$ and $Y$ along the $\Gamma$ curves. Let $X_0, Y_0$ be the points in $\Teich(S \setminus \Gamma)$ obtained by pinching the boundary components of $X\setminus \Gamma$ and $Y \setminus \Gamma$ to punctures and preserving the rest of the Fenchel-Nielsen coordinates of $X$ and $Y$, respectively.
We claim that $X_0,Y_0$ have no common $\eta$--short essential closed curves: otherwise
\eqref{eq:min-pinching} and~\eqref{eq:min-nested} would give an additional common
$\eps$--short curve of $X,Y$ not in $\Gamma$.  Hence~\eqref{eq:min-no-common} gives
\[
 d_T(X_0,Y_0)\le A_0d_L(X_0,Y_0)+B_0.
\]
If the complement is disconnected, apply this componentwise and take the
maximum; the constants are uniform over the finitely many component types.
Then
\[
 \max\left\{\log\frac1{\ell_X(\gamma)},
                 \log\frac1{\ell_Y(\gamma)}\right\}
 \le P_\eps(X,Y),
\]
so, combining this inequality and Eq.~\eqref{eq:min-ann-error} from \Cref{lem:min-annulus}, we get
\begin{equation}
\frac12d_{\Hh}^{\mathrm{std}}(z_\gamma(X),z_\gamma(Y))
 \le d_{L(\gamma)}(z_\gamma(X),z_\gamma(Y))
      +\log P_\eps(X,Y)+C.
\label{eq:inequalities2}
\end{equation}
Using the fact that the curves in $\Gamma$ are $\epsilon$-short for both $X$ and $Y$, we can insert Equation~\eqref{eq:inequalities2} into Equation~\eqref{eq:min-minsky}, which allows us to replace the annular terms of $d_T$ by the annular terms of $\dsym$. Finally, we use
\eqref{eq:min-CR-product} to obtain a global comparison with $\dsym$.   This proves 
\begin{equation}
d_T(X,Y)
 \le A d_L(X,Y)+B\log\bigl(2+P_\eps(X,Y)\bigr)+B.
\label{eq:intermediate_eq}
 \end{equation}
 As $P_\epsilon \geq 2$, we can always bound $2+ P_\epsilon \leq 2P_\epsilon$, so after increasing $B$ to $(1+ \log 2)B$, we obtain
~\eqref{eq:min-global}.
\end{proof}

\subsection{Asymptotic cones and the deep multicurve}

For the rest of the upper-bound proof, choose a non-principle ultrafilter $\omega$ of $\N$. Fix a $d_L$--asymptotic cone of $\Teich(S)$ 
with basepoints $X_i$ and scales $R_i \to_\omega \infty$, and denote it by 
\[
\mathcal X= Cone_\omega(X,X_i,R_i;d_L).
\]
See \Cref{subsec:asymptotic_cone} for the definition and relevant background. We will eventually choose the $X_i$ dependent on a sequence of quasi-isometric embeddings $f_i: [0,R_i]^m \to \Teich(S)$.

For $\epsilon>0$, we define the $\eps$--depth of a curve $\alpha$ as
\begin{equation}
 \delta_X(\alpha)
 =\max\left\{0,\log\frac\eps{\ell_X(\alpha)}\right\}.
 \label{eq:min-depth}
\end{equation}
Thus $\delta_X(\alpha)>0$ exactly when $\alpha$ is $\eps$--short. We will simply refer to $\delta_X(\alpha)$ as the \textbf{depth} of $\alpha$, since the  $\epsilon$ is fixed once and for all as the cutoff constant.
We will be interested in sequences of curves with large depth with respect to the scales $R_i$, i.e.,
whose depth goes to infinity. In fact, more precisely, we will single out those sequences of curves whose
depth goes to infinity super-linearly in $R_i$ (along $\omega$). 
Let $(\alpha_i)$ be a sequence of
curves $\alpha_i \subset X_i$.  Relative to $(X_i)$ and the scales $(R_i)$, we call
$(\alpha_i)$ \textbf{shallow}  if
\[
 \delta_{X_i}(\alpha_i)=O_\omega(R_i),
\]
and \textbf{deep} if
\[
 \frac{\delta_{X_i}(\alpha_i)}{R_i}
 \longrightarrow_\omega\infty.
\]
Note that the curves in a shallow sequence $(\alpha_i)$ can be both $\epsilon$-short or not, but the curves in a deep sequence are all $\epsilon$-short.
Observe also that the ultralimit of
$\delta_{B_i}(\alpha_i)/R_i$ shows that shallow and deep are
mutually exclusive and exhaustive for sequences.

Since  $\epsilon$ is smaller than the Margulis constant, we have that the set $\Gamma_i$ of $\epsilon$-short curves on any $X_i$ are disjoint. The following lemma shows that any deep sequence must be chosen from particular sets of curves $\Gamma_i$ in the $X_i$ that have the same size for each $i$. This will mean that, up to mapping class group action, deep sequences can be sampled from a single multicurve $\Gamma$, so they are essentially constant.

\begin{lem}
\label{lem:min-deep-selection}
Let $X_i$ denote the basepoints of the $d_L$-asymptotic cone $\mathcal{X}$. 
There exists $k\le\xi(S)$ so that for each $i$, there is a $k$--component simple closed multicurves $\Gamma_i$ such that any sequence $(\alpha_i)$ where $\alpha_i \in \Gamma_i$ for each $i$ is deep, and any sequence $(\beta_i)$ where $\beta_i \not \in \Gamma_i$ for each $i$ is shallow, i.e.,
\begin{equation}
 \min_{\gamma\in\Gamma_i}
 \frac{\delta_{X_i}(\gamma)}{R_i}\longrightarrow_\omega\infty,
 \qquad
 \sup_{\beta\notin\Gamma_i}\delta_{X_i}(\beta)=O_\omega(R_i).
 \label{eq:min-deep-split}
\end{equation}
\end{lem}
\begin{proof}
Let $m_i$ be the number of $\epsilon$ short curves in $X_i$, for each $i$. The $\epsilon$-short curves in $X_i$ are pairwise disjoint, so $m_i \leq \xi(S)$.  Write the curves as $\alpha_i^1,\ldots,\alpha_i^{m_i}$, ordered so that if
$a_i^j\coloneqq\delta_{X_i}(\alpha_i^j)$ then
\[
a_i^1 \geq \cdots \geq a_i^{m_i}>0,
\]
Pad the sequence by setting
\begin{equation}
\label{eq:padding}
a_i^j=0 \text{ for } m_i < j \leq \xi(S)
\end{equation}
For each fixed $1 \leq j \leq \xi(S)$, define
\[
 \lambda_j=\lim_\omega\frac{a_i^j}{R_i}\in[0,\infty].
\]
Since $a_i^j \geq a_i^{j+1}$ for every $i$, we have
\[
\lambda_1 \geq \cdots \lambda_{\xi(S)}
\]
Hence there is a single integer $0\leq k \leq \xi(S)$ such that 
\[
\lambda_j = \infty\quad
\Longleftrightarrow\quad 1 \leq j \leq k.
\]
Notice that if $j \leq k$,
then the set
\[
\{ i \colon m_i \geq j \}
\]
is $\omega$--large. Indeed, on its complement we have $a_i^j=0$ by~\eqref{eq:padding}, which is incompatible with $\lambda_j = \infty$.
Thus, if $k>0$, then on an $\omega$--large set the curves
\[
\alpha_i^1, \cdots, \alpha_i^k
\]
all exist. On this set we define
\[
\Gamma_i = \{ \alpha_i^1,\cdots, \alpha_i^k \}.
\]
On the complementary $\omega$--null set, choose any $k$-component simple multicurve; when $k=0$, put $\Gamma_i=\varnothing$.
Since the selected depths are decreasing,
\[
\min_{\gamma \in \Gamma_i} \delta_{X_i}(\gamma)=a_i^k
\]
on an $\omega$--large set, and hence
\[
\lim_\omega \min_{\gamma \in \Gamma_i} \frac{\delta_{X_i}(\gamma)}{R_i}=\lambda_k = \infty
\]
when $k>0$. This proves that any sequence sampled from the $\Gamma_i$ is deep.

Now let $(\beta_i)$ be a sequence so that $\beta_i \not \in \Gamma_i$ for each $i$. Then either $\beta_i$ has zero depth, or it has  positive-depth, but it does not belong to $\Gamma_i$. If $\beta_i$ does not belong to $\Gamma_i$, then it has depth at most $a_i^{k+1}$. Since
\[
\lambda_{k+1}=\lim_\omega\frac{a_i^{k+1}}{R_i}<\infty
\]
we obtain
\[
\sup_{\beta \notin \Gamma_i}\delta_{X_i}(\beta)
\leq a_i^{k+1}=O_\omega(R_i).
\]
This is~\eqref{eq:min-deep-split}. So we have shown that any sequence that is not sampled from the $\Gamma_i$ is shallow.
\end{proof}

We next show that the collection of deep multicurves can be stabilized. 

\begin{lem}
\label{lem:min-deep-persistence}
There exists a an asymptotic cone $\mathcal X'$ for $d_L$ that is isometric to $\mathcal X$, and a 
simple closed multicurve $\Gamma$, so that if $(\alpha_i)$ is any sequence where $\alpha_i \in \Gamma$ for each $i$ then $(\alpha_i)$ is deep, and if $(\beta_i)$ is any sequence where $\beta_i \not \in \Gamma$ for each $i$, then $(\beta_i)$ is shallow with respect to the basepoints of $\mathcal X'$.  Furthermore, the same holds with respect to any $d_L$-admissible sequence $(Y_i)$. In particular, for every $d_L$--admissible sequence $(Y_i)$,
\begin{equation}
  \min_{\gamma\in\Gamma}
 \frac{\delta_{Y_i}(\gamma)}{R_i}\longrightarrow_\omega\infty,
 \qquad
 \sup_{\beta\notin\Gamma}\delta_{Y_i}(\beta)=O_\omega(R_i).
 \label{eq:min-persistent-split}
\end{equation}
\end{lem}

\begin{proof}
Let $(\Gamma_i)$ be the sequence of $k$-component multicurves given by \Cref{lem:min-deep-selection}.
There are finitely many $\MCG(S)$--orbits of $k$--component simple multicurves~\cite[Lemma~2.16]{ES22:GeodesicCount}. Thus, at least one $\MCG(S)$ orbit occurs infinitely often. For every infinite set, either the set or its complement is in $\omega$. So there is an $\omega$-large set $A$, a sequence $\phi_i \in \MCG(S)$ for $i \in A$, and a $k$-component simple closed multicurve $\Gamma$, so that $\phi_i(\Gamma_i) = \Gamma$ for all $i$.

Set $\mathbf \phi = (\phi_i)$, and set $X_i' = \phi_i(X_i)$, for each $i$. Then we can define a new asymptotic cone $\mathbf \phi(\mathcal X)$ using the same nonprincipal ultrafilter $\omega$ and the same sequence of scales $R_i$, but with basepoints $X_i'$. We let $\mathcal X' : = \mathbf \phi(\mathcal X)$, 
which we call a \textbf{level-dependent change of marking} of $\mathcal X$. 

In fact, $\mathcal X'$ is isometric to $\mathcal X$, as the $\MCG(S)$-invariance of $d_L$ gives the cone isometry
\[
[Y_i]\longmapsto[\phi_iY_i].
\]
Since
$\delta_{\phi_iX_i}(\phi_i\alpha)=\delta_{X_i}(\alpha)$, the conclusions
of \Cref{lem:min-deep-selection} now hold with the fixed
$\Gamma$.

Admissibility of $(Y_i)$ implies $d_L(X_i, Y_i) = O_\omega(R_i)$. Thus, for any curve $\alpha$,
\begin{equation}
 |\delta_X(\alpha)-\delta_Y(\alpha)|\le d_L(X,Y).
 \label{eq:min-depth-Lip}
\end{equation}
Thus an $O_\omega(R_i)$ move on the basepoint subtracts only $O_\omega(1)$ from the scaled
depths of the selected curves and adds only $O_\omega(R_i)$ to the depths
of the unselected curves.  So a sequence $(\alpha_i)$ is deep (respectively, shallow) for the sequence $X_i$ if and only if it is deep (respectively, shallow) for the sequence $Y_i$.
\end{proof}

\subsection{The product cone and its factors}

From now on, we can assume without loss of generality that the conclusions of \Cref{lem:min-deep-persistence} hold for our original asymptotic cone $\mathcal X$. In particular, we can assume there is a simple closed multicurve $\Gamma$ so that, for all $\gamma \in \Gamma$, the constant sequence $(\gamma)_{i \in \N}$ is deep for any sequence $[Y_i]$ in $\mathcal X$. In particular, on the sequence of basepoints $(X_i)$, we have 
\[
X_i \in \Thin_\eps(S, \Gamma)
\]
on an $\omega$-large set of indices $i$.

 Cut $S$ along $\Gamma$, and replace the boundaries of the resulting connected components with punctures. Let $Q_1, \dots, Q_q$ be the resulting punctured surfaces. We then consider the map 
\[
\Pi_\Gamma \colon \Teich(S) \longrightarrow \prod_{r=1}^q \Teich(Q_r) \times \prod_{\gamma \in \Gamma}U_\epsilon
\]
where
\begin{equation}
 \Pi_\Gamma(X_i)
 =\bigl((X_{i,r})_{r=1}^q,
        (z_\gamma(X_i))_{\gamma\in\Gamma}\bigr).
 \label{eq:min-base-product-coordinates}
\end{equation}
is defined just as in Eq.~\eqref{eq:min-product-coordinates}. Note that in Eq.~\eqref{eq:min-product-coordinates}, we defined the domain of $\Pi_\Gamma$ to be $\Thin_\eps(S, \Gamma)$, but it is well-defined on all of $\Teich(S)$.

We now obtain a factor-wise isometric embedding at the level of cone factors.

\begin{lem}
\label{lem:min-cone-product}
Let
\[
 \mathcal X=\Cone_\omega(\Teich(S),X_i,R_i;d_L),
 \qquad
 \mathcal X_r=\Cone_\omega(\Teich(Q_r),X_{i,r},R_i;d_L),
\]
and set
\[
 \mathcal A_\gamma
 =\Cone_\omega(U_\eps,z_\gamma(X_i),R_i;d_{L(\gamma)}).
\]
Then product coordinates give an isometric embedding
\begin{equation}
 \Psi:\mathcal X\longrightarrow
 \prod_{r=1}^q\mathcal X_r
 \times\prod_{\gamma\in\Gamma}\mathcal A_\gamma
 \label{eq:min-cone-product}
\end{equation}
where we take the sup metric on the right hand side.
\end{lem}

\begin{proof}
We define $\Psi$ by
\[
 \Psi([Y_i])
 =\bigl(([Y_{i,r}])_{r=1}^q,
        ([z_\gamma(Y_i)])_{\gamma\in\Gamma}\bigr).
\]
where $Y_{i,r}$ and $z_\gamma(Y_i)$ are the terms of $\Pi_\Gamma(Y_i)$ for each $i$. 

We first show this is well-defined. Let $(Y_i)$ be $d_L$ admissible, and define
\[
I_Y \coloneqq \{ i : Y_i \in \Thin_\epsilon(S,\Gamma)\}.
\]
By \Cref{lem:min-deep-persistence}, $I_Y$ is $\omega$-large.

Suppose
\[
[Y_i]=[Z_i]\quad\text{in }\mathcal X,
\]
ie,
\[
\lim_\omega \frac{d_L(Y_i,Z_i)}{R_i}=0.
\]

Since intersections of $\omega$-large sets are $\omega$-large, we have that for all $i$ in the $\omega$-large set $I_Y \cap I_Z$, both $Y_i, Z_i \in \Thin_\eps(\Gamma, S)$. Thus, the Choi-Rafi product region formula \cite[Theorem~3.5]{CR07:Comparison}, recorded in Eq~\eqref{eq:min-CR-product}, applies for all such $i$. In fact, set
\[
M_i = \max \left\{\max_r d_L(Y_{i,r},Z_{i,r}),\max_{\gamma \in \Gamma} d_{L(\gamma)}(\bigl(z_\gamma(Y_i),z_\gamma(Z_i)\bigr)\right\}.
\]
The \cite[Theorem~3.5]{CR07:Comparison} gives a uniform constant $C>0$  (depending on the topology of $S$ and the cutoff parameters) so that
\begin{equation}
\bigl|d_L(Y_i,Z_i)-M_i\bigr|\leq C.
\label{eq:M_i}
\end{equation}
on an $\omega$-large set.

So we have
\[
d_L(Y_{i,r},Z_{i,r}) \leq d_L(Y_i,Z_i) + C
\]
and
\[
d_{L(\gamma)}\bigl(z_\gamma(Y_i),z_\gamma(Z_i)\bigr)
\leq 
d_L(Y_i,Z_i)+C.
\]
Since $R_i \to_\omega \infty$, dividing by $R_i$ and passing to the $\omega$-limit, we get
\[
[Y_{i,r}]=[Z_{i,r}]\quad\text{and}\quad[z_\gamma(Y_i)]=[z_\gamma(Z_i)].
\]
So $\Psi$ is well-defined.

Now we show that $\Psi$ is an isometry. Choose any $[Y_i], [Z_i] \in \mathcal X$. Again, the Choi-Rafi product region theorem applies on an $\omega$-large set. Thus, dividing Eq.~\eqref{eq:M_i} by $R_i$ and taking the $\omega$--limit, the error $C/R_i$ disappears.
Since there are finitely many terms, the ultralimit metric commutes with taking the maximum and thus
\begin{align*}
d_L^\omega([Y_i],[Z_i]) = \max \biggl \{& \max_r d_{L,r}^\omega([Y_{i,r}],[Z_{i,r}]),\\
&
\max_{\gamma \in \Gamma}
d_{L(\gamma)}^\omega
\bigl([z_\gamma(Y_i)],[z_\gamma(Z_i)]\bigr)\biggr\}.
\end{align*}
This proves that $\Psi$ is an isometric embedding.
\end{proof}

Now, we show that on each of the components $\mathcal X_r$ above, the Lipschitz and the Teichmuller metrics are bilipschitz equivalent.
\begin{lem} \label{lem:bilip_on_shallow_comps}
Let $Q_r$ be the result of pinching the boundaries of some component of $S\setminus \Gamma$.  For each $i$, let $X_{i,r}\in\Teich(Q_r)$ be the point obtained by keeping only the Fenchel-Nielsen coordinates of $Q_r$. Then there is a bilipschitz
homeomorphism
\begin{equation}
 \Phi: \Cone_\omega(\Teich(Q_r),X_{i,r},R_i;d_L)
 \longrightarrow
 \Cone_\omega(\Teich(Q_r),X_{i,r},R_i;d_T),
 \label{eq:min-shallow-bilip}
\end{equation}
In fact, we show that we can take $\Phi$ to be the identity on representing sequences, and that
\[
 \frac12d_L^\omega\le d_T^\omega\le A d_L^\omega
\]
for $A$ the same constant as in Eq.~\eqref{eq:min-global} of \Cref{lem:min-global}.
Here we set the supremum over an empty curve to be $0$, so if $\xi(Q_r)=1$ (i.e. $Q_r$ is a single pair of pants), the
conclusion above is trivial.
\end{lem}

\begin{proof}
  If $\xi(Q_r)=1$, there is nothing to prove as $\Teich(Q_r)$ is trivial. So, assume $\xi(Q_r)>1$.

We define the map $\Phi$ as the identity map on admissible sequences. First, we show that $\Phi$ is well-defined. So let $[Y_{i,r}] \in \Cone_\omega(\Teich(Q_r),X_{i,r},R_i;d_L)$  be a $d_L$-admissible sequence. We must show that it is $d_T$-admissible. 

By \Cref{lem:min-global}, 
\begin{equation}
d_T(X_{i,r}, Y_{i,r}) \leq A d_L(X_{i,r}, Y_{i,r}) + B \log P_\epsilon(X_{i,r}, Y_{i,r}) + B
\label{eq:min-global-rewrite}
\end{equation}
for some constants $A,B$. Following Eq~\eqref{eq:min-P},
\[
P_\epsilon(X_{i,r}, Y_{i,r})  =\max\left(
 \{2\}\cup
 \left\{
  \log\frac1{\ell_{X_{i,r}}(\alpha)},
  \log\frac1{\ell_{Y_{i,r}}(\alpha)}
 \right\}\right)
\]
where the maximum is taken over curves $\alpha$ that are $\epsilon$-short for both $X_{i,r}$ and $Y_{i,r}$. To show that $[Y_{i,r}]$ is $d_T$-admissible, we must first bound $P_\epsilon(X_{i,r},Y_{i,r})$.

So suppose $\beta$ is
$\eps$--short for both $X_{i,r}$ and $Y_{i,r}$. First we show that essential simple closed curves in $Q_r$  are shallow. That is, we will show that 
\begin{equation}
 \sup_{\beta\in\Simples(Q_r)}\delta_{X_{i,r}}(\beta)=O_\omega(R_i).
 \label{eq:min-comp-shallow}
\end{equation}
Indeed, if $Q_r$ is a pair of pants, then it has no essential, non-peripheral simple closed curves, so the statement is vacuously true. Otherwise, for each $\beta\in\Simples(Q_r)$, Eq.~\eqref{eq:min-pinching} gives that $\ell_{X_i}(\beta) < K_0 \ell_{X_{i,r}}(\beta)$. So from Eq~\eqref{eq:min-depth}, we see that 
\[
\delta_{X_{i,r}}(\beta) < \delta_{X_i}(\beta) + \log K_0.
\]
By \Cref{lem:min-deep-persistence}, we have $\sup_{\beta\in\Simples(Q_r)}\delta_{X_i}(\beta)=O_\omega(R_i)$. Thus, the same is true for $\ell_{X_{i,r}}(\beta)$, which gives
\eqref{eq:min-comp-shallow}. 

Using the definition of $\delta_{X_{i,r}}(\beta)$ in Eq~\eqref{eq:min-depth}, we get
\[
\log\frac1{\ell_{X_{i,r}}(\alpha)} \leq \delta_{X_{i,r}}(\beta) + \log \frac 1\epsilon
\]
where $\log \frac 1\epsilon$ is a constant. Thus, $\log\frac1{\ell_{X_{i,r}}(\alpha)} = O_\omega(R_i)$.

Next, $\ell_{X_i,r}(\beta) \leq e^{d_L(X_{i,r}, Y_{i,r})} \ell_{Y_{i,r}}(\beta)$. Therefore, 
\[
\log\frac1{\ell_{Y_{i,r}}(\beta)} \leq \log\frac1{\ell_{X_{i,r}}(\alpha)} + d_L(X_{i,r}, Y_{i,r})
\]
and so,
\begin{equation}
 \log\frac1{\ell_{Y_{i,r}}(\beta)}
 \le \delta_{X_{i,r}}(\beta)+d_L(X_{i,r},Y_{i,r})+\log\frac1\eps.
 \label{eq:min-depth-to-P}
\end{equation}
By Eq.~\eqref{eq:min-comp-shallow}, we know $\delta_{X_{i,r}}(\beta) = O_\omega(R_i)$. Since $[Y_{i,r}]$ is $d_L$-admissible, we also have $d_L(X_{i,r},Y_{i,r}) = O_\omega(R_i)$. And $\log \frac 1 \epsilon$ is a constant. Therefore,
\[
 \log\frac1{\ell_{Y_{i,r}}(\beta)} =O_\omega(R_i).
\]
From here, we conclude $P_\epsilon(X_{i,r}, Y_{i,r}) = O_\omega(R)$. Therefore, 
\[
B \log P_\epsilon(X_{i,r}, Y_{i,r}) + B = o_\omega(R_i)
\]
and so $d_T(X_{i,r}, Y_{i,r}) = A d_L(X_{i,r}, Y_{i,r}) + o_\omega(R)$. Thus, if $[Y_{i,r}]$ is $d_L$-admissible, meaning that $d_L(X_{i,r}, Y_{i,r}) = O_\omega(R_i)$, then it is also $d_T$ admissible. Therefore, the map $\Phi$ is well-defined.

We must now show that $\Phi$ is bilipschitz. By the same argument as above, for two $d_L$--admissible sequences $[Y_{i,r}],[Z_{i,r}]$,
\begin{equation}
 P_\eps(Y_{i,r},Z_{i,r})=O_\omega(R_i),
 \qquad \text{ and so, } \qquad
 B\log P_\eps(Y_{i,r},Z_{i,r}) + B=o_\omega(R_i).
 \label{eq:min-P-sublinear}
\end{equation}

Therefore, dividing both sides of Eq.~\eqref{eq:min-global-rewrite} by $R_i$ and taking limits, we get that 
\[
d_T^\omega([Y_{i,r}], [Z_{i,r}]) \leq A d_L^\omega([Y_{i,r}], [Z_{i,r}]) 
\]

On the other hand, we have $d_L({Y_i,r}, Z_{i,r}) \leq 2 d_T({Y_i,r}, Z_{i,r})$ by\cite[equation~(3)]{CR07:Comparison}, as we recorded in Eq~\eqref{eq:min-wolpert}. Again, dividing both sides by $R_i$ and taking the $\omega$-limit gives 
\[
d_L^\omega([Y_{i,r}], [Z_{i,r}]) \leq 2 d_T^\omega([Y_{i,r}], [Z_{i,r}]) 
\]
Therefore, $\frac12d_L^\omega\le d_T^\omega\le A d_L^\omega$, as required.
\end{proof}

The following is the analog of Bowditch's~\cite[Theorem~1.8]{Bow16} and Durham's results~\cite[Theorem~1.1]{Durham} for $d_L$ and $\dsym$.
See \Cref{defi:lebesgue_dim} in Section~\ref{subsec:adim} to recall the precise definition of topological dimension.
\begin{lem}
\label{lem:min-dimension}
Every compact subset of every asymptotic cone of
$(\Teich(S),d_L)$ or $(\Teich(S),\dsym)$ has topological dimension at most
$\xi(S)$.
\end{lem}
\begin{proof}
Let $\mathcal X$ be an arbitrary $d_L$--cone and let
$K\subset\mathcal X$ be compact. Let $\Gamma$ be the simple closed multicurve given by \Cref{lem:min-deep-persistence}. (Note that the topological dimension is an isometry invariant, so we can assume that $\mathcal X$ already satisfies the conclusions of \Cref{lem:min-deep-persistence}). Pinching the boundaries of the components of $S\setminus \Gamma$ gives us the sequence of punctured surfaces $Q_1, \dots, Q_q$. Then \Cref{lem:min-cone-product} gives an isometric embedding
\begin{equation*}
 \Psi:\mathcal X\longrightarrow
 \prod_{r=1}^q\mathcal X_r
 \times\prod_{\gamma\in\Gamma}\mathcal A_\gamma
\end{equation*}
where $X_r$ is an asymptotic cone defined on $\Teich(Q_r)$ and $A_\gamma$ is an asymptotic cone defined on $U_\epsilon = \R \times (\frac 1 \epsilon, \infty)$.

Project $\Psi(K)$ to compact sets
$K_r\subset\mathcal X_r$ and $K_\gamma\subset\mathcal A_\gamma$.
By \Cref{lem:bilip_on_shallow_comps}, 
$K_r$ is bilipschitz
to a compact subset
of the $d_T$--cone $\Cone_\omega(\Teich(Q_r),X_{i,r},R_i;d_T)$. Bowditch's~\cite[Theorem~1.8]{Bow16} and Durham's~\cite[Theorem~1.1]{Durham} dimensions theorems therefore give
\[
 \dim K_r\le\xi(Q_r),
\]
On the other hand, \Cref{lem:min-annular-cones} gives
$\dim \mathcal A_\gamma = 1$, and so $\dim K_\gamma \leq 1$.  

The set
$\Psi(K)$ is compact, hence closed, in the product of its compact
projections.  Monotonicity and the finite-product theorem for dimension
\cite[Theorems~III.1 and III.4, pp.~26 and~33]{HW} give
\begin{equation}
 \dim K
 \le\sum_{r=1}^q\xi(Q_r)+|\Gamma|.
 \label{eq:min-dim-product}
\end{equation}
Cutting $S$ along each curve in $\Gamma$ and replacing the two new resulting boundary
curves by punctures lowers the sum of the complementary complexities by
one.  Hence
\begin{equation}
 \sum_{r=1}^q\xi(Q_r)=\xi(S)-|\Gamma|.
 \label{eq:min-complexity}
\end{equation}
Equations~\eqref{eq:min-dim-product} and~\eqref{eq:min-complexity} give
$\dim K\le\xi(S)$.

Finally, since $d_L$ and $\dsym$ are bilipschitz equivalent, the identity on representing
sequences is bilipschitz between the corresponding $d_L$-- and
$\dsym$--cones.  The same compact-subset bound therefore holds for
$\dsym$.
\end{proof}

\subsection{Proof of Theorem~\ref{thm:min-rank}}
\label{subsec:proof}

We give the proof of Theorem~\ref{thm:min-rank}, which states that the rank of $\Teich(S)$ with the Lipschitz metric is exactly $\xi(S)$.
\begin{proof}
Let $m=\rank(\Teich(S), \dsym)$. Then there exist $\lambda \geq 1, K > 0$ and a sequence of $(\lambda, K)$ quasi-isometric embeddings $f_i \colon [0,R_i]^m \to \Teich(S)$.
By \Cref{lem:min-cube-cone}, for every nonprincipal ultrafilter $\omega$, the asymptotic $\dsym$-cone $\mathcal{X}$ of
$\Teich(S)$ based at $f_i(0)$ and scaled by $R_i$ contains a compact subset $K$ bilipschitz to $[0,1]^m$. Hence, $m = \dim(K)$. On the other hand, by \Cref{lem:min-dimension}, $\dim(K) \leq \xi(S)$, and hence $m \leq \xi(S)$.

 The lower bound on flat rank follows easily from~\cite[Theorem~C]{CR07:Comparison}:  indeed, consider a pants decomposition $\{\gamma_1,\cdots,\gamma_{\xi(S)}\}$, and for each factor $U_\epsilon(\gamma_i)$ we define a ray $[0,\infty) \to U_\epsilon(\gamma_i)$ of marked hyperbolic structures with length coordinate given by  $\ell_{\gamma_i}(t) \coloneqq \epsilon/2 e^{-t}$, and twist coordinate $t_{\gamma_i}(t)\coloneqq 0$, for every $t \in [0,\infty)$. By~\cite[Theorem~C]{CR07:Comparison}, combined these give a quasi-isometric embedding of $\mathbb{R}^{\xi(S)}_+$ into $\Teich(S)$. Thus, the flat rank of $\Teich(S)$ with the Lipschitz metric is at least $\xi(S)$. By \Cref{lem:flat_rank_lower}, flat rank gives a lower bound on geometric rank. This yields the result.
\end{proof}

\section{Appendix}
\label{sec:appendix}
This next lemma allows us to use many sectors of $\R^n_+$ rather than $\R^n_+$ itself, when bounding sector rank from below.

\begin{lem}
\label{lem:sectors_QI}
 Let $\vec v_1, \dots,\vec v_n$ be $n$ linearly independent vectors in $\R^n_+$. Let $\mathcal R = \{\sum x_i \vec v_i \ | \ x_i \geq 0\}$ be the positive convex hull of the rays through these vectors. Then $\mathcal R$ is quasi-isometric to $\R^n_+$.
\end{lem}
\begin{proof}
As $\vec v_1, \dots, \vec v_n$ are linearly independent, there is some matrix $A$ sending them to the standard basis $\vec e_1, \dots, \vec e_n$. By linearity, $A(\sum x_i v_i) = \sum x_i e_i$, so in particular, it sends $\mathcal R$ to $\R^n_+$.

As $\vec v_1, \dots, \vec v_n$ form a basis for $\R^n$, the matrix $A$ defines a bijection on $\R^n$. Therefore, there are constants $c, C>0$ so that, for all $\vec x$,
\[
 c\norm{x} \leq \norm{A x} \leq C\norm{x}.
\]
Setting $\lambda = \max \{C, \frac 1c\}$, we see that the map $\vec x \mapsto A \vec x$ is a $(\lambda, 0)$-quasi-isometry from $\mathcal R$ to $\R^n_+$.
\end{proof}

The following is a classical inequality relating ratios of sums to the minimum and maximum of ratios.

\begin{lem}
\label{lem:ratiosum}
Let $(a_i)', (w_i)$ be a sequences of non-negative real numbers, and $(a_i)$ a sequence positive of real numbers so that $\sum_{i} w_i a_i' < \infty$ and $\sum_{i} w_i a_i < \infty$.
Then
\[
\inf_i \frac{a_i'}{a_i} \leq \frac{\sum_{i} w_i a_i' }{\sum_{i} w_i a_i} \leq \sup_i \frac{a_i'}{a_i}
\]
\end{lem}
\begin{proof}
Let $a \coloneqq \inf_i \frac{a_i'}{a_i}$ and $b \coloneqq \sup_i \frac{a_i'}{a_i}$.
Then
\[
a \leq \frac{a_i'}{a_i} \leq b
\]
so
\[
w_i a_i a \leq w_i a_i' \leq w_i a_i b
\]
so, adding over $i$
\[
\sum_i w_i a_i a \leq \sum_i w_i  a_i' \leq \sum_i a_i w_i b
\]
which, dividing by $\sum_i w_i a_i$ yields the inequality.
\end{proof}

\bibliographystyle{hamsalpha}
\bibliography{main}

\end{document}